\documentclass[12pt,leqno]{article}

\usepackage{mathtools}
\usepackage{amsthm}
\usepackage{amssymb}
\usepackage[svgnames]{xcolor}
\usepackage[
colorlinks=true,
linkcolor=MidnightBlue,
citecolor=MidnightBlue,
urlcolor=MidnightBlue,
pagebackref=true,
linktocpage=true]{hyperref}
\usepackage[alphabetic]{amsrefs}
\usepackage{pinlabel}
\usepackage{enumitem}
\usepackage{fullpage}

\newcommand{\N}{\mathbb{N}}
\newcommand{\Z}{\mathbb{Z}}
\newcommand{\Q}{\mathbb{Q}}
\newcommand{\R}{\mathbb{R}}
\newcommand{\C}{\mathbb{C}}
\newcommand{\A}{\mathbb{A}}

\newcommand{\F}{\mathsf{F}}
\newcommand{\W}{\mathsf{W}}
\newcommand{\p}{\mathsf{P}}
\newcommand{\V}{\mathsf{V}}
\newcommand{\I}{\mathsf{I}}
\newcommand{\Lv}{\mathrm{Lv}}
\newcommand{\Ch}{\mathrm{Ch}}
\newcommand{\Fi}{\mathrm{Fib}}
\newcommand{\Pal}{\mathrm{Pal}}
\newcommand{\Fa}{\mathcal{F}}
\newcommand{\Gal}{\mathcal{G}}
\newcommand{\X}{\mathcal{X}}
\newcommand{\Sc}{\mathcal{S}}
\newcommand{\BQ}{\mathcal{BQ}}
\newcommand{\PS}{\mathcal{PS}}
\newcommand{\BI}{\mathcal{BI}}
\newcommand{\lt}{\mathcal{L}}
\newcommand{\rt}{\mathcal{R}}
\newcommand{\Hom}{\mathrm{Hom}}
\newcommand{\Inn}{\mathsf{Inn}}
\newcommand{\Aut}{\mathsf{Aut}}
\newcommand{\Out}{\mathsf{Out}}
\newcommand{\Spe}{\mathsf{Spe}}
\newcommand{\Sym}{\mathsf{Sym}}
\newcommand{\Prim}{\mathsf{Prim}}
\newcommand{\Axis}{\mathsf{Axis}}
\newcommand{\Isom}{\mathsf{Isom}}
\newcommand{\id}{\textup{Id}}

\newcommand{\GL}{\mathsf{GL}}
\newcommand{\SL}{\mathsf{SL}}
\newcommand{\PGL}{\mathsf{PGL}}
\newcommand{\PSL}{\mathsf{PSL}}
\newcommand{\tr}{\mathsf{tr}}
\newcommand{\am}{\mathsf{am}}

\newcommand{\Ga}{\Gamma}
\newcommand{\al}{\alpha}

\newcommand{\de}{\delta}
\newcommand{\ep}{\epsilon}
\newcommand{\io}{\iota}
\newcommand{\ka}{\kappa}
\newcommand{\la}{\lambda}
\newcommand{\si}{\sigma}
\newcommand{\ti}{\widetilde}
\newcommand{\un}{\underline}
\newcommand{\tho}{\widetilde{\rho}}

\newcommand{\e}{\mathsf{e}}
\newcommand{\f}{\mathsf{f}}
\newcommand{\g}{\mathsf{g}}

\newcommand{\h}{\mathcal{H}}
\renewcommand{\H}{\mathbb{H}}
\newcommand{\hex}{\mathsf{Hex}}

\newcommand{\sslash}{\mathbin{/\mkern-6mu/}}

\numberwithin{equation}{section}
\numberwithin{figure}{subsection}

\newtheorem{introthm}{Theorem}

\swapnumbers 

\theoremstyle{plain}
\newtheorem{theorem}[equation]{Theorem}
\newtheorem{lemma}[equation]{Lemma}
\newtheorem{proposition}[equation]{Proposition}

\theoremstyle{definition}
\newtheorem{definition}[equation]{Definition}

\theoremstyle{remark}
\newtheorem{remark}[equation]{Remark}
\newtheorem{convention}[equation]{Convention}
\newtheorem{question}[equation]{Question}

\title{Bowditch's Q-conditions and Minsky's primitive stability}
\author{Jaejeong Lee and Binbin Xu}
\date{}

\begin{document}

\maketitle

\begin{abstract}
For the action of the outer automorphism group of the rank two free group on the corresponding variety of $\PSL(2,\C)$ characters, two domains of discontinuity have been known to exist that are strictly larger than the set of Schottky characters. One is introduced by Bowditch in 1998 (followed by Tan, Wong and Zhang in 2008) and the other by Minsky in 2013. We prove that these two domains are equal. We then show that they are contained in the set of characters having what we call the bounded intersection property.
\end{abstract}

\tableofcontents

\section{Introduction}

Let $\F_n$ denote the free group on $n\ge2$ generators, and set $G=\PSL(2,\C)$. The space of representations $\Hom(\F_n,G)$ admits a natural action of $\Aut(\F_n)\times\Aut(G)$ given by
\begin{alignat*}{5}
\Aut(\F_n)&\times\Aut(G)&&:&\;\Hom(\F_n&,G)&&\longrightarrow&&\;\Hom(\F_n,G)\\
(&f,g)&&:&&\rho&&\longmapsto&&\;\;g\circ\rho\circ f^{-1}
\end{alignat*}
The $G$-character variety $\X(\F_n)$ of $\F_n$ is then defined as the GIT quotient
\begin{align*}
\X(\F_n)=\mathrm{Hom}(\F_n,G)\sslash \boldsymbol{\left(\right.}\{Id\}\times\Inn(G)\boldsymbol{\left.\right)}
\end{align*}
by the action of inner automorphisms of $G$. It then admits an action by the outer automorphism group
\[
\Out(\F_n)=\Aut(\F_n)/\Inn(\F_n),
\]
since the inner  automorphism group $\Inn(\F_n)$ acts trivially. 

Regarding the dynamical properties of the action of $\Out(\F_n)$ on $\X(\F_n)$, one of the classical results is that this action is properly discontinuous on the open subset $\Sc\subset\X(\F_n)$ consisting of Schottky characters. One may then ask if $\Sc$ is a maximal domain of discontinuity for the action. Recently, Minsky \cite{Min} answered the question in the negative: an open subset $\PS$, consisting of \emph{primitive stable} characters, is strictly larger than $\Sc$
\[
\Sc\subsetneq\PS\subset\X(\F_n)
\]
and admits a properly discontinuous action of $\Out(\F_n)$. In particular, since $\Sc$ is known to be the interior of the closed subset consisting of characters coming from injective representations with discrete image, we see that $\PS$ contains characters coming from representations with non-discrete image. Thus Minsky's result shows that the geometric decomposition (discrete vs. dense) of the characters is different from the dynamical decomposition (properly discontinuous vs. ergodic) in this case. For more results about dynamics on representation spaces we refer the reader to the survey papers of Goldman \cite{Gol06} (for surfaces and locally homogeneous geometric structures on them), Lubotzky \cite{Lub11} (for free groups and various target groups), and Canary \cite{Can15} (for word hyperbolic groups into semi-simple Lie groups with no compact factors). 

\bigskip

We will focus on the special case of $n=2$ in the present paper, where another domain of discontinuity for the action of $\Out(\F_2)$ has been known to exist. Bowditch \cite{Bow} introduced a notion called the \emph{Q-conditions} and investigated the type-preserving characters in $\X(\F_2)$ in an attempt to characterize the quasi-Fuchsian once-punctured torus groups. Generalizing Bowditch's work, Tan, Wong and Zhang \cite{TWZ08} tested the Q-conditions for \emph{all} characters in $\X(\F_2)$ and showed that the characters satisfying the Q-conditions form an open subset $\BQ\subset\X(\F_2)$ and that $\Out(\F_2)$ acts on it properly discontinuously. As an interesting property of characters in $\BQ$, Bowditch and Tan-Wong-Zhang showed that a variation of McShane's identity holds for them.

In the first part of the paper we study the relation between these two domains $\BQ$ and $\PS$. Very roughly speaking, the primitive stability of $\rho\in\Hom(\F_2,\PSL(2,\C))$ means that under the orbit map of the $\rho$-action on $\H^3$, the axes in $\F_2$ of primitive words are sent to uniform quasi-geodesics (Definitions~\ref{def:ps} and \ref{def:altps}). On the other hand, the Q-conditions on $\rho$ mean that for every primitive word, its cyclically reduced length and the translation distance of its $\rho$-action on $\H^3$ are comparable (Definition~\ref{def:bq} and Theorem~\ref{thm:bqpd}). It is not hard to see from the definitions that we have the inclusion
\[
\PS\subset\BQ\subset\X(\F_2).
\]
A proof is given by Lupi in his thesis \cite{Lup}*{Proposition 2.9}, but we also provide a different proof in Proposition~\ref{prop:psbq} to make the paper self-contained.

A number of authors questioned if these two conditions are actually equivalent. See \cite{Min}*{p.64}, \cite{Can15}*{\S.7.1} and \cite{Lup}*{\S.4.2}. Also compare \cite{Gol06}*{Problem 3.2}. Our first theorem answers the question in the affirmative:
\begin{introthm}\label{I}
If a representation $\rho:\F_2\to\PSL(2,\C)$ satisfies the Q-conditions, then it is primitive stable. Consequently, we have the equality
\[
\PS=\BQ.
\]
\end{introthm}
\noindent We prove this in Section~\ref{sec:pf}. 

\begin{remark}
(a) Lupi showed the same equality in his thesis \cite{Lup}*{Proposition 3.4 and Theorem 4.3} when the target group $G$ is restricted to $\PSL(2,\R)$. His idea of proof, however, does not seem to directly generalize to the case of $\PSL(2,\C)$. We discuss this issue in Section~\ref{sec:lupi}.

\noindent(b) One may view Theorem~\ref{I} in analogy with a theorem of Delzant, Guichard, Labourie and Mozes \cite{DGLM} that, for a word hyperbolic group acting on a metric space by isometries, an orbit map is a \emph{quasi-isometric embedding} if and only if the action is \emph{displacing}. In a sense, primitive stability (resp. the Q-conditions) is a weakening of the former (resp. the latter) property obtained by looking only at actions of primitive elements. We discuss this analogy in more detail in Section~\ref{sec:pspd}.
\end{remark}

Let us briefly present our ideas of proof. The definitions of primitive stability and the Q-conditions involve the conjugacy classes of primitive elements of $\F_2$. The so-called \emph{Christoffel words} are representatives of such classes and have many desirable properties for our purposes. Without giving the precise definition (Definition~\ref{def:chris}) for now, let us exhibit a simple and motivational example in order to illustrate how to generate the Christoffel words and how the relevant geometry behaves asymptotically under the representations of Theorem~\ref{I}.

\begin{figure}[ht]
\labellist
\pinlabel {$\H^2$} at 30 280
\pinlabel {$a$} at 108 154
\pinlabel {$b$} at 194 162
\pinlabel {\footnotesize $r$} at 222 214
\pinlabel {\footnotesize $p$} at 74 206
\pinlabel {\footnotesize $q$} at 152 110
\pinlabel {\footnotesize $rpq$} at 300 134
\pinlabel {\footnotesize $pqr$} at 146 279
\pinlabel {\footnotesize $qrp$} at 12 128
\pinlabel {\footnotesize $q(pqr)q$} at 156 8
\pinlabel {$a$} at 494 100
\pinlabel {$p$} at 460 160
\pinlabel {$q$} at 556 46
\pinlabel {\footnotesize $x$} at 690 60
\pinlabel {\footnotesize $q(x)$} at 394 60
\pinlabel {\footnotesize $pq(x)=a(x)$} at 560 240
\pinlabel {$L$} at 790 120
\pinlabel {$R$} at 950 90
\endlabellist
\centering
\includegraphics[width=\textwidth]{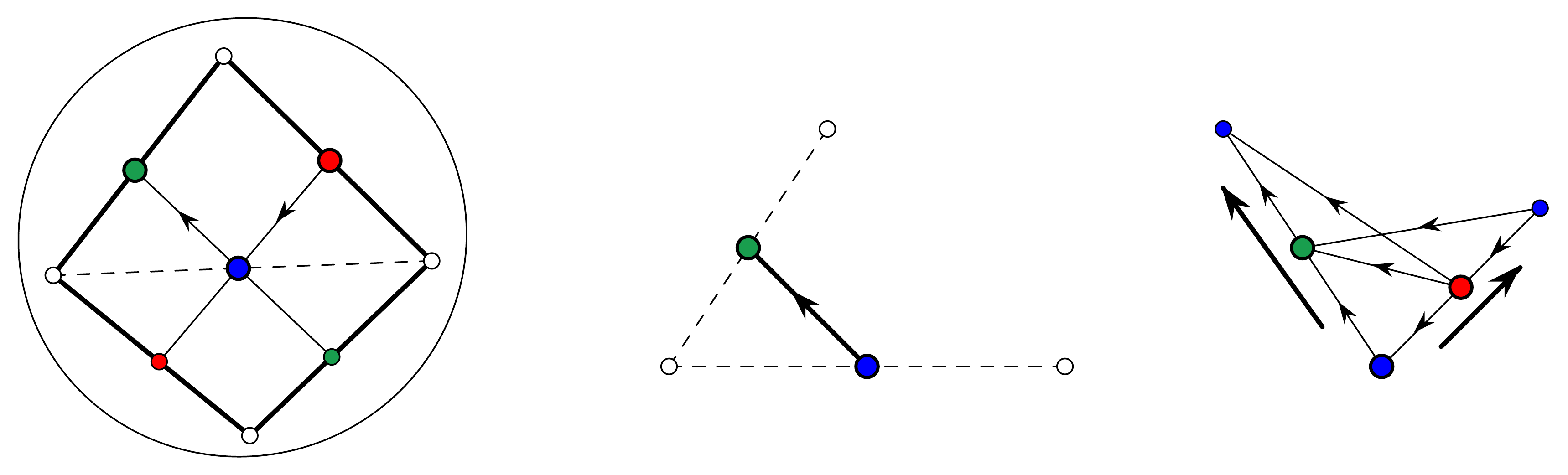}
\caption{A subgroup $\W(p,q,r)$ of $\Isom^+(\H^2)$, a hyperbolic translation $a=pq$, and two operations $L$ and $R$.}
\label{fig:torus}
\end{figure}
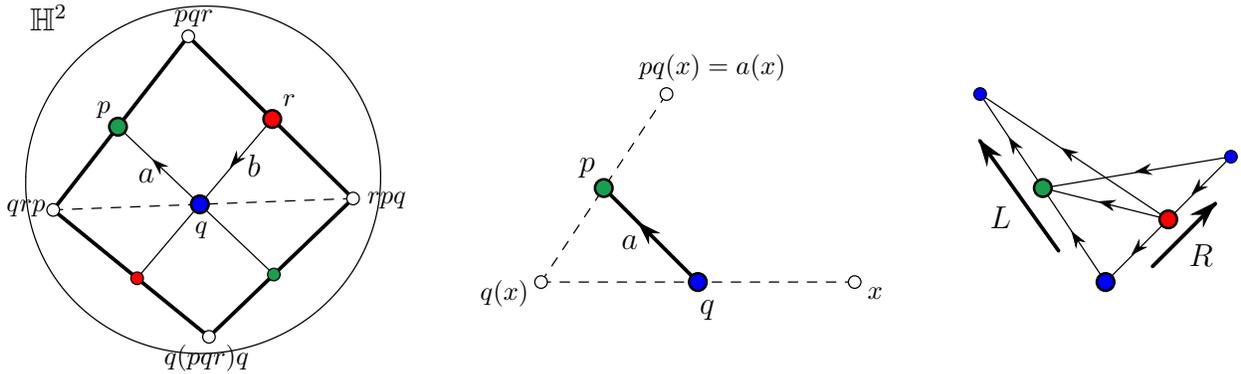

See Figure~\ref{fig:torus}. The first picture shows a subgroup $\W(p,q,r)$ of $\Isom^+(\H^2)$ generated by three orientation-preserving involutions $p$, $q$ and $r$, whose fixed points are specified in the Klein projective model of $\H^2$. If the product $pqr$ is of infinite order, the group $\W(p,q,r)$ is isomorphic to the free Coxeter group $\Z/2\ast\Z/2\ast\Z/2$ of rank three and contains the free group $\F(a,b)$, where $a=pq$ and $b=qr$, as an index two subgroup. Note that $ab^{-1}a^{-1}b=pqr(qq)pqr=(pqr)^2$. The element $pqr$ is an elliptic transformation in the picture; its fixed point and the fixed points of its three other conjugates form a quadrilateral. If $pqr$ is parabolic or hyperbolic, then the quadrilateral is ideal or hyper-ideal, respectively. In all possible cases, however, the elements $a=pq$ and $b=qr$ are always hyperbolic translations and they identify the opposite edges of the quadrilateral. The quotient surface is a hyperbolic torus either with a cone-type singularity or with a cusp or with a funnel, depending on the nature of $pqr$.

The second picture represents a hyperbolic translation $a=pq$ as a directed edge. The third picture shows two operations $L$ and $R$ (left and right slides) which inductively generate other basis triples of involutions in $\W(p,q,r)$ as well as other directed edges starting with the initial ones. The Christoffel words are the hyperbolic translations represented by the directed edges so generated. They are primitive elements of $\F(a,b)$. (For a more detailed account, see Section~\ref{sec:chris} as well as Appendix~\ref{sec:a}.)

\begin{figure}[ht]
\labellist
\pinlabel {$a$} at 410 54
\pinlabel {$b$} at 484 40
\pinlabel {$b$} at 534 90
\pinlabel {$ab$} at 450 70
\pinlabel {$ab$} at 346 96
\pinlabel {$ab^2$} at 478 100
\pinlabel {$ab^2$} at 630 124
\pinlabel {$abab^2$} at 470 170
\pinlabel {$abab^2$} at 180 110
\pinlabel {$ab(ab^2)^2$} at 550 170
\pinlabel {$abab^2ab(ab^2)^2$} at 360 150
\endlabellist
\centering
\includegraphics[width=\textwidth]{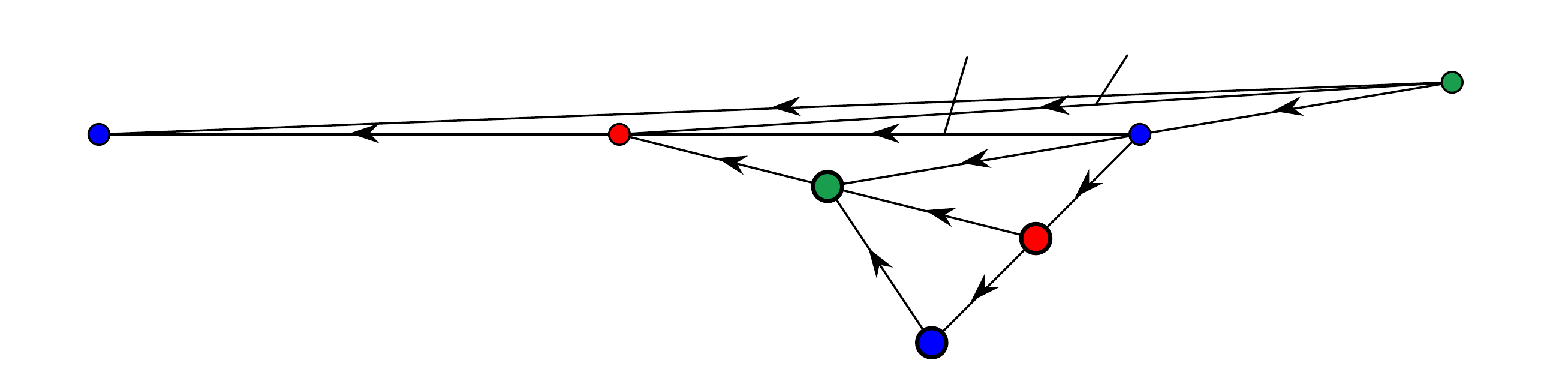}
\caption{Idea of proof in the case of $G=\PSL(2,\R)\cong\Isom^+(\H^2)$. }
\label{fig:torus2}
\end{figure}

Figure~\ref{fig:torus2} exhibits an example of the iteration $LRLR$ (read from right to left). Notice how quickly and uniformly the associated triangles degenerate with respect to the number of steps of $LR$-slides. In fact, if the initial triangle is acute then other triangles generated by the $LR$-slides are all obtuse. This asymptotic behavior of triangles immediately imply that the representation $\F(a,b)$ in $\Isom^+(\H^2)$ satisfies the Q-conditions and, after a little more work, primitive stability as well. We remark that this process may be viewed as being reverse to the classical \emph{trace minimizing algorithm} for checking discreteness of two-generator subgroups of $\PSL(2,\R)$. See, for example, \cite{Gil95}*{Chapter 2} and the references therein. 

In the general case of representations in $\PSL(2,\C)$, we exploit the fact (see Theorem~\ref{thm:coxeter}) that an irreducible representation of $\F(a,b)$ extends uniquely to the free Coxeter group $\W(p,q,r)$, thereby giving us a triple of complete geodesics in $\H^3$ which are axes of the involutions $p$, $q$ and $r$. For representations satisfying the Q-conditions (and hence irreducibility as well), we show that the triple of axes behaves analogously as in Figure~\ref{fig:torus2} under the $LR$-slides asymptotically. See Lemma~\ref{lem:angle} for a precise statement. Combined with the properties of Christoffel words, this enables us to show primitive stability.

\bigskip

In the second part of the paper (Section~\ref{sec:bi}) we investigate a condition which we call the \emph{bounded intersection property}. It is motivated by the work of Gilman and Keen \cite{GK09}, where they investigated palindromic elements of $\F_2$ in order to test the discreteness of irreducible representations $\rho:\F_2\to\PSL(2,\C)$. Slightly strengthening their idea we consider the primitive elements that are palindromic in one of the three chosen bases $(a,b)$, $(b,c)$, $(a,c)$ of $\F_2$, where $abc=1$. The bounded intersection property is a certain geometric condition on the $\rho$-images of these palindromic primitive elements. See Definition~\ref{def:bi} for details. This property also defines a subset $\BI\subset\X(\F_2)$, which we show in Proposition~\ref{prop:bi} to be invariant under the action of $\Out(\F_2)$.

We do not know if $\BI$ is an open subset nor if $\Out(\F_2)$ acts properly discontinuously on it. Nevertheless, we prove in Section~\ref{sec:pf2} the following.

\begin{introthm}\label{II}
If $\rho:\F_2\to\PSL(2,\C)$ satisfies the Q-conditions, then it has the bounded intersection property. That is, we have the inclusion $\BQ\subset\BI$. On the other hand, if $\rho$ is discrete and faithful, and has the bounded intersection property, then it satisfies the Q-conditions.
\end{introthm}

We were kindly informed by Caroline Series that she also proved Theorem~\ref{I} as well as a result similar to Theorem~\ref{II}. See \cite{Series}.

\bigskip

The rest of the paper is organized as follows. In Section~\ref{sec:2} we recall some classical results about automorphisms of $\F_2$. We introduce the Christoffel words in a basis of $\F_2$ and discuss their properties. In Section~\ref{sec:3} we mostly follow Fenchel \cite{Fen} and discuss some basic geometry of $\H^3$ with an emphasis on half-turns and Coxeter extensions. We also introduce the amplitude of a right-angled hexagon. In Section~\ref{sec:4} we recall the definition of primitive stability and that of the Q-conditions, and prove that a primitive stable representation satisfies the Q-conditions. In Section~\ref{sec:pf} we give a proof of Theorem~\ref{I}. In Section~\ref{sec:bi} we discuss palindromic elements of $\F_2$, introduce the bounded intersection property and prove Theorem~\ref{II}. In Appendix~\ref{sec:a} we discuss automorphisms of $\F_2$ in view of its Coxeter extension.

\paragraph{Acknowledgments.}
The present work was initiated during our respective collaboration with Ser Peow Tan. We are deeply thankful to him for generously bridging us together and encouraging our collaboration. We would like to thank Misha Kapovich and Xin Nie for numerous helpful discussions, and Caroline Series for her interest and encouragement. We also thank the referee for many helpful suggestions that clarified our earlier argument.

J. Lee would like to thank Inkang Kim for his support and encouragement; he was supported by the grant NRF-2014R1A2A2A01005574 and NRF-2017R1A2A2A05001002. B. Xu thanks Korea Institute for Advanced Study, where he finished the major part of the collaboration as a postdoc; he was supported by the grant DynGeo FNR INTER/ANR/15/11211745.

\section{The free group of rank two and Christoffel words}\label{sec:2}

We recall classical facts about the (outer) automorphism group of the free group $\F_2$ of rank two. One often studies the group $\Out(\F_2)$ via its identification with the mapping class group of the one-holed torus. Since the representations of our interest are not (yet) known to be coming from a geometric structure on the surface, we do not discuss this identification here but rather proceed purely algebraically.

The eventual goal is to introduce Christoffel bases, which are explicit representatives of conjugacy classes of bases of $\F_2$. Their nice properties, especially Lemma~\ref{lem:chris}, will be essential in our proof of Theorem~\ref{I}.

\subsection{The automorphism group of \texorpdfstring{$\F_2$}{F2}}

The following results of Nielsen are all from his seminal paper \cite{Niel17}.

Let $\F_2$ denote the free group on two generators. A \emph{basis} of $\F_2$ is an ordered pair of free generators. An element of $\F_2$ is called \emph{primitive} if it is a member of a basis. The abelianization $\F_2/[\F_2,\F_2]$ of $\F_2$ is isomorphic to the free abelian group $\Z^2$ of rank two. Bases and primitive elements of $\Z^2$ are defined analogously.

To be specific, let us fix one basis $(a,b)$ of $\F_2$ once and for all, and denote it by
\[
\e=(a,b).
\]
In order to emphasize this we shall often write
\[
\F_2=\F(\e)=\F(a,b)
\]
interchangeably. The abelianization homomorphism
\begin{align}\label{eqn:abel}
\pi_\e:\F_2\to\Z^2
\end{align}
is then given by $\pi_\e(a)=(1,0)$ and $\pi_\e(b)=(0,1)$. 

Nielsen showed the following.
\begin{theorem}[Nielsen]\label{thm:Nielsen}
Let $x$ and $y$ be elements of $\F(a,b)$.
\begin{enumerate}[label=\textup{(\alph*)},nosep,leftmargin=*]
\item If the pair $(x,y)$ is a basis of $\F(a,b)$ then the $2\times2$ matrix with columns $\pi_\e(x)$ and $\pi_\e(y)$ belongs to $\GL(2,\Z)$, that is, has determinant $\pm1$. Conversely, any matrix in $\GL(2,\Z)$ determines exactly one conjugacy class of a basis of $\F(a,b)$.
\item The pair $(x,y)$ is a basis of $\F(a,b)$ if and only if the commutator $xyx^{-1}y^{-1}$ is conjugate either to the commutator $aba^{-1}b^{-1}$ or to its inverse $bab^{-1}a^{-1}$.
\end{enumerate}
\end{theorem}

Let us briefly discuss only the item (a) of the theorem. Note that a basis $(x,y)$ of $\F_2$ corresponds to a unique automorphism $f\in\Aut(\F_2)$ defined by $f(a)=x$ and $f(b)=y$. Nielsen found a set of generators of $\Aut(\F_2)$: they are the so-called elementary Nielsen transformations and correspond to the bases
\begin{align}\label{eqn:elementary}
(b,a),\; (a^{-1},b),\; (a,b^{-1}),\; (a,ab),\; (ab,b),
\end{align}
respectively. Since an analogous result for $\Aut(\Z^2)\cong\GL(2,\Z)$ is well-known, we see that the abelianization homomorphism $\pi_\e$ induces a surjective homomorphism
\[
\Aut(\F_2)\twoheadrightarrow\Aut(\Z^2).
\]
Nielsen also showed that the kernel of this map is exactly the group $\Inn(\F_2)$ of all inner automorphisms of $\F_2$. See also \cite{LS77}*{Proposition I.4.5}. Hence we obtain an isomorphism
\[
\Out(\F_2)=\Aut(\F_2)/\Inn(\F_2)\cong\Aut(\Z^2)\cong\GL(2,\Z).
\]
Therefore, the conjugacy classes of bases of $\F_2$ are in one-to-one correspondence with the bases of $\Z^2$. This implies, in turn, that the conjugacy classes of primitive elements of $\F_2$ are in one-to-one correspondence with the primitive elements of $\Z^2$. See also \cite{OZ81}*{Corollary 3.2}.

\subsection{Unoriented primitive classes and the Farey triangulation}\label{sec:farey}

Both Bowditch's Q-conditions and Minsky's primitive stability are concerned with conjugacy classes of primitive elements of $\F_2$. As we shall see later in Section~\ref{sec:4}, these conditions also make sense for a slightly bigger classes. To be more precise, we consider the equivalence relation $\sim$ on $\F_2$, where
\begin{align*}
a'\sim a\quad\Longleftrightarrow\quad a'\text{ is conjugate either to  }a\text{ or to }a^{-1}.
\end{align*}
Let us denote by $\Prim$ the set of all primitive elements of $\F_2$. For a primitive element $x\in\Prim$ we call its $\sim$-equivalence class $[x]\in\Prim/_\sim\subset\F_2/_\sim$ the \emph{unoriented primitive class} of $x$.

In view of the above one-to-one correspondence between primitive conjugacy classes of $\F_2$ and primitive elements of $\Z^2$, we see that the unoriented primitive classes of $\F_2$ correspond to the projectivized primitive elements of $\Z^2$
\begin{align}\label{eqn:prim}
\Prim/_\sim\;\longleftrightarrow\;\p^1(\Q)=\Q\cup\{\infty\}
\end{align}
and the action of $\Aut(\F_2)$ on the unoriented primitive classes factors through $\PGL(2,\Z)=\GL(2,\Z)/\{\pm\id\}$:
\[
\Aut(\F_2)\twoheadrightarrow\PGL(2,\Z)\curvearrowright\p^1(\Q).
\]
Note that the kernel of this homomorphism is $\Inn^\mathfrak{e}(\F_2):=\Inn(\F_2)\rtimes\langle\mathfrak{e}\rangle$, where $\mathfrak{e}$ is an involutory automorphism of $\F_2=\F(a,b)$ defined by $\mathfrak{e}(a,b)=(a^{-1},b^{-1})$.

\begin{convention}\label{conv:slope}
When we write an element $p/q$ of $\Q$ for coprime integers $p$ and $q$, we shall always assume $q>0$. We formally write either $1/0$ or $(-1)/0$ for the infinity $\infty$.
\end{convention}

\begin{figure}[ht]
\labellist
\pinlabel {$0/1$} at 304 155
\pinlabel {$1/0=$} at 0 160
\pinlabel {$-1/0$} at 0 140
\pinlabel {$1/1$} at 156 292 
\pinlabel {$-1/1$} at 150 16 
\pinlabel {\footnotesize $1/2$} at 246 262 
\pinlabel {\footnotesize$-1/2$} at 240 42 
\pinlabel {\footnotesize$2/1$} at 68 262
\pinlabel {\footnotesize$-2/1$} at 68 42
\pinlabel {\scriptsize $1/3$} at 270 234
\pinlabel {\scriptsize $3/1$} at 44 234
\pinlabel {\scriptsize $-1/3$} at 276 76
\pinlabel {\scriptsize $-3/1$} at 38 76
\pinlabel {\scriptsize $2/3$} at 210 280
\pinlabel {\scriptsize $3/2$} at 102 280
\pinlabel {\scriptsize $-2/3$} at 206 28
\pinlabel {\scriptsize $-3/2$} at 102 28

\pinlabel {\huge $0/1$} at 550 155
\pinlabel {\huge $1/0=$} at 420 170
\pinlabel {\huge $-1/0$} at 420 134
\pinlabel {\Large $1/1$} at 482 240 
\pinlabel {\Large $-1/1$} at 476 70
\pinlabel {$1/2$} at 545 232 
\pinlabel {$-1/2$} at 540 76 
\pinlabel {$2/1$} at 420 232
\pinlabel {$-2/1$} at 418 76
\endlabellist
\centering
\includegraphics[width=\textwidth]{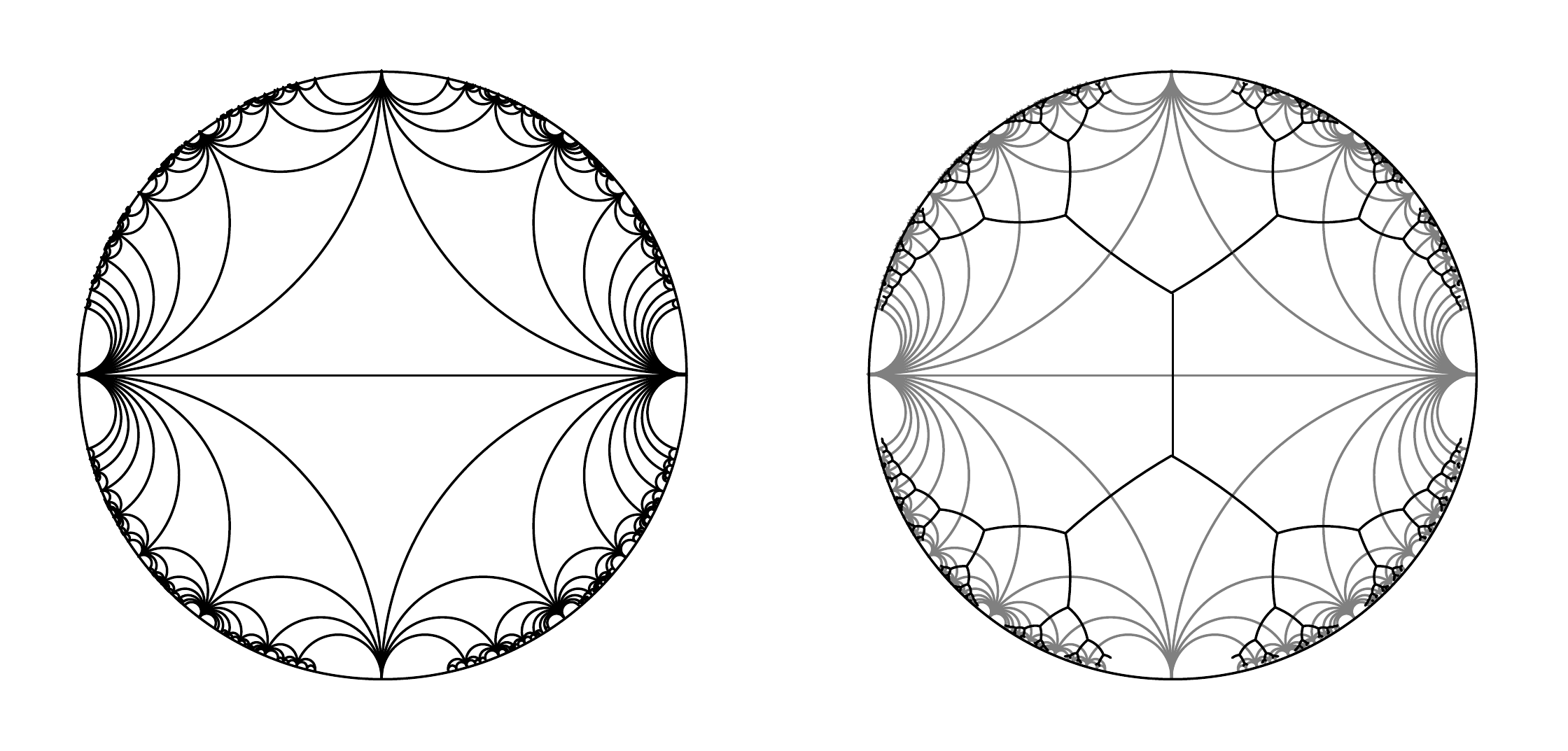}
\caption{(Left) The Farey triangulation $\Fa$. (Right) Its dual tree with rational numbers written inside the complementary regions.}
\label{fig:Farey}
\end{figure}
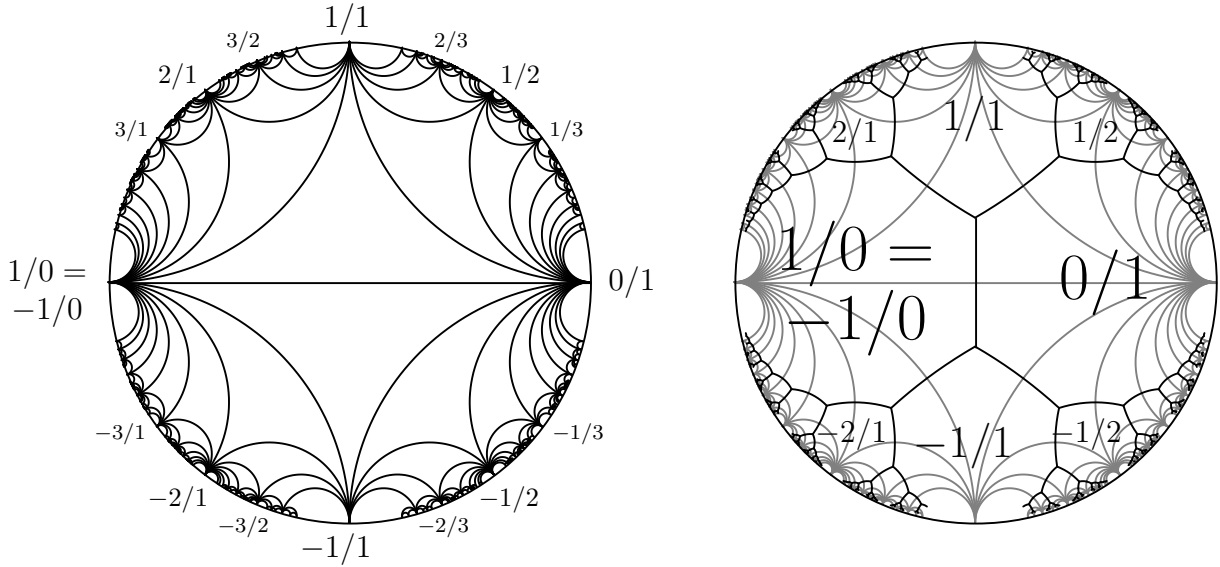

As the action of $\Aut(\F_2)$ on $\Prim/_\sim$ factors through $\PGL(2,\Z)$, we may visualize this action via the classical Farey triangulation. Let us take the upper half-plane of the complement $\p^1(\C)\setminus\p^1(\R)$ as a conformal model of the hyperbolic plane $\H^2$. We denote $\overline{\H^2}=\H^2\cup\partial_\infty\H^2$, where the ideal boundary $\partial_\infty\H^2$ is identified with $\p^1(\R)$. The \emph{Farey triangulation} $\Fa$ is a $2$-dimensional simplicial complex realized as a subset
\[
\Fa=\H^2\cup\p^1(\Q)\subset\overline{\H^2}.
\]
By definition, the set of $0$-simplices are realized as $\p^1(\Q)$. We denote
\[
\V=\p^1(\Q)
\]
and call its elements the Farey vertices. The $1$-simplices, called the Farey edges, are realized as complete geodesics connecting two Farey vertices $p/q$ and $r/s$ such that
\[
\det \begin{pmatrix}p&r\\q&s\end{pmatrix}=\pm1.
\]
The $2$-simplices, called the Farey triangles, are realized as ideal triangles bounded by three Farey edges. See Figure~\ref{fig:Farey}. For more about the Farey triangulation see, for example, \cite{Hat}*{Chapter 1}.

Abusing notations we shall often identify the vertex set $\V=\p^1(\Q)$ with the set $\Prim/_\sim$ of unoriented primitive classes via \eqref{eqn:prim}
\[
\V=\Prim/_\sim.
\]
and say a Farey vertex $[x]\in\V$. The conjugacy class of a basis $(x,y)$ of $\F_2$ determines a \emph{directed} Farey edge connecting vertices from $[x]$ to $[y]$, which we simply denote by
\[
[x,y]
\]
instead of $([x],[y])$. Note that three other conjugacy classes of bases $(x,y^{-1})$, $(x^{-1},y)$ and $(x^{-1},y^{-1})$ also determine the same directed edge $[x,y]$. Compare Theorem~\ref{thm:Nielsen}(a). By convention, however, whenever we write $[x,y]$ it will be assumed that the specific representative basis $(x,y)$ is chosen. The directed edge corresponding to the distinguished basis $\e=(a,b)$ will often be denoted by $[\e]=[a,b]$; it runs from $\infty=1/0$ to $0=0/1$.

We identified the Farey triangulation $\Fa$ as a subset of the disk $\overline{\H^2}$. In most instances, however, we are interested only in its combinatorial structure and draw the pictures accordingly: for example, see Figure~\ref{fig:level}.

\begin{remark}\label{rem:topograph}
The dual object of the Farey triangulation is the complete binary tree. If it is properly embedded in the plane, it is called the \emph{topograph} (see \cite{Con} and also \cite{Hat}). 
The combinatorial structure of topograph is more convenient in Bowditch's theory, where one investigates flows on the topograph. See Bowditch \cite{Bow} as well as \cite{TWZ08} and \cite{GMST}. Although we do not use the topograph directly, its tree structure underlies the inductive construction that we explain below using galleries and levels.
\end{remark}

\subsection{Levels and Christoffel words}\label{sec:chris}

Recall that we fixed a basis $\e=(a,b)$ of $\F_2$ and that the distinguished directed Farey edge is denoted by $[\e]=[a,b]$. This naturally leads us to consider the $2$-fold reflectional symmetry of the Farey triangulation with respect to $[\e]$ among other symmetries in $\PGL(2,\Z)$.

From the embedding $\Fa\subset\overline{\H^2}$ of the Farey triangulation, we can talk about the sides of directed Farey edges. The set of Farey vertices lying on the left-hand (resp. right-hand) side of $[x,y]$ will be denoted by $\V^+[x,y]$ (resp. $\V^-[x,y]$).

A single directed Farey edge $[x,y]$ \emph{induces} directions on other edges as follows. The circle $\partial\H^2$ is the union of two semi-circles connecting $[x]$ and $[y]$, both of which are endowed with the linear order running from $[x]$ to $[y]$. We then direct other Farey edges according to this linear order; we say they are \emph{co-directed} with $[x,y]$. For example, Figure~\ref{fig:level}(Left) shows the directions induced by $[\e]$. This will be used later for the inductive procedure in Definitions~\ref{def:chris} and \ref{def:pal}.

\begin{figure}[ht]
\labellist
\pinlabel {$[\e]$} at 246 144
\pinlabel {$[\e]$} at 796 144
	
\pinlabel {\small \text{level} $0$} at 860 112
\pinlabel \rotatebox{12}{\small \text{level} $1$} at 740 154
\pinlabel \rotatebox{-12}{\small \text{level} $1$} at 740 92
\pinlabel \rotatebox{39}{\small \text{level} $2$} at 650 190
\pinlabel \rotatebox{-39}{\small \text{level} $2$} at 650 60
\pinlabel \rotatebox{66}{\small \text{level} $3$} at 616 230
\endlabellist
\centering
\includegraphics[width=\textwidth]{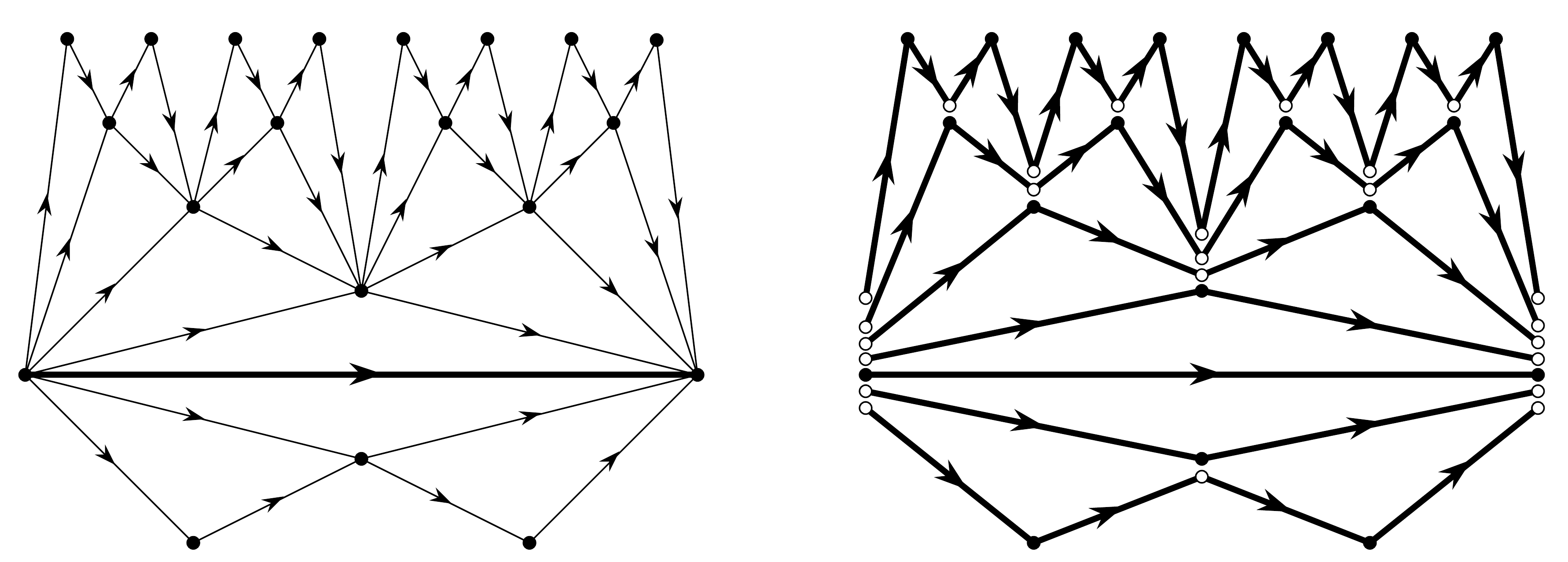}
\caption{(Left) Farey edges with directions induced from the edge $[\e]$. (Right) Levels with respect to the edge $[\e]$.}
\label{fig:level}
\end{figure}
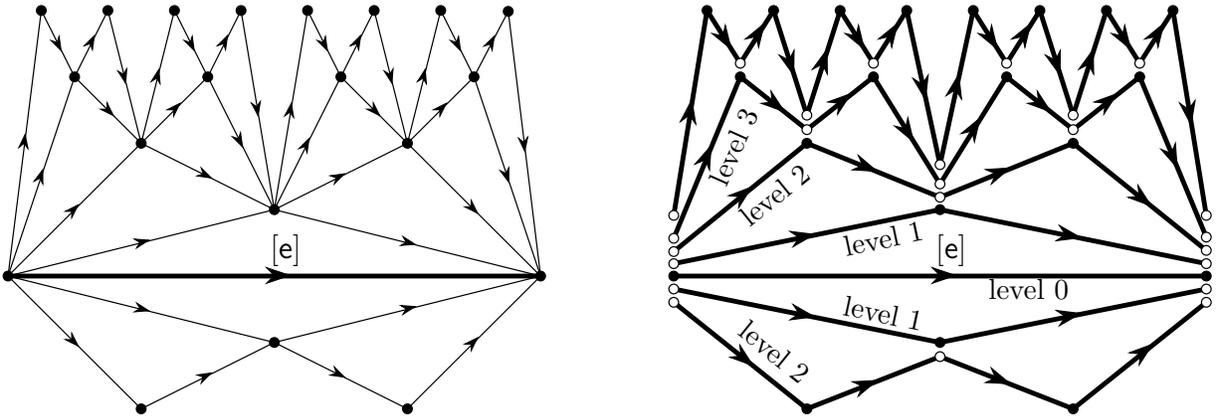

We now introduce the notion of levels relative to the distinguished edge $[\e]$. A \emph{gallery} is the union $T_1\cup T_2\cup\cdots\cup T_k$ of Farey triangles such that the triangles $T_i$ and $T_{i+1}$ are adjacent along an edge for $1\le i\le k-1$. The \emph{length} of a gallery is the number of triangles it contains.

\begin{definition}
(See Figure~\ref{fig:level}(Right).) Let $X$ denote either a vertex distinct from $[a]$ or $[b]$, or an edge distinct from $[\e]$, or a triangle. Then there exists a unique gallery of shortest length containing both $X$ and $[\e]$. The \emph{$\e$-level} $\Lv_\e(X)$ of $X$ is defined to be the length of this gallery. The edge $[\e]$ itself as well as its end points $[a]$ and $[b]$ are defined to be of \emph{$\e$-level} $0$. There is no triangle of $\e$-level $0$.
\end{definition}

\noindent Most of the time the reference edge $[\e]$ is understood and we shall simply say \emph{levels} instead of $\e$-levels. While we defined the level combinatorially, it can also be defined using the geometry of $\H^2$ as in \cite{GK11}*{Definition 3}.

A vertex of level $k>0$ is contained in a unique triangle of level $k$. A triangle of level $k>0$ contains two edges of level $k$ and one edge of level $k-1$; if $k>1$ its vertices are of level $k$, $k-1$ and $j$ (with $j\le k-2$), respectively.

A directed edge $[x,y]$ of level $k>0$ contains vertices of level $k$ and $j$ (with $j\le k-1$), and is contained in two triangles of level $k$ and of level $k+1$, which will be denoted, respectively, by
\[
[w;x,y]\quad\textup{and}\quad[x,y;z],
\]
where $\Lv_\e[w]\le k-1$ and $\Lv_\e[z]=k+1$. The union of these triangles will be denoted by
\[
[w;x,y;z]
\]
and called the \emph{quadrilateral} determined by $[x,y]$.

\bigskip

Following Bowditch \cite{Bow}*{p.704, p.706} we use the notion of levels in order to define functions on $\V^\pm[a,b]$ inductively. Let $S$ be a set and $B:S\times S\to S$ a binary operation. Given an $S$-valued function $f$ defined on the vertices $[a]$ and $[b]$ of level $0$, we can inductively extend it to a map $f:\V^+[a,b]\to S$ as follows: a vertex $[z]\in\V^+[a,b]$ of level $k+1$ $(k\ge0)$ is contained in a unique triangle $[x,y;z]$ of level $k+1$ such that $[x,y]$ is co-directed with $[a,b]$. Then the levels of $[x]$ and $[y]$ are at most $k$. Assuming, inductively, the values $f[x]$ and $f[y]$ have already been defined, we define $f[z]=B(f[x],f[y])$. Similarly, we extend $f$ to $\V^-[a,b]$ starting with (possibly new) initial values of $f$ on $[a]$ and $[b]$.

Note that the level function $\Lv_\e$ itself can be defined in this way by setting $S=\Z$ and $B(p,q)=\max\{p,q\}+1$ starting with $\Lv_\e[a]=\Lv_\e[b]=0$. More interesting examples are the Fibonacci function and the Christoffel function to be defined as follows. See Figure~\ref{fig:fibchris}.

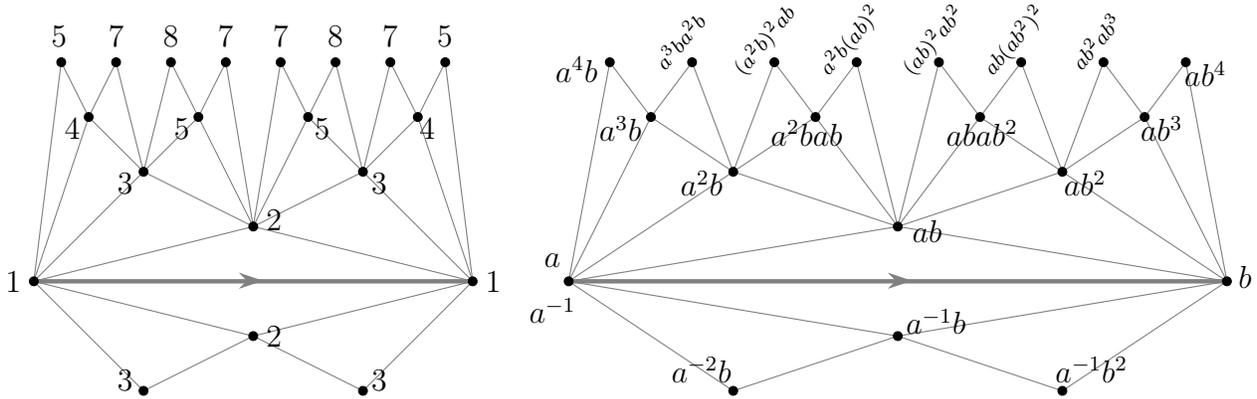
\begin{figure}[ht]
\labellist
\pinlabel {$1$} at 2 120
\pinlabel {$1$} at 440 120
\pinlabel {$2$} at 240 176
\pinlabel {$2$} at 240 70
\pinlabel {$3$} at 104 210
\pinlabel {$3$} at 104 30
\pinlabel {$3$} at 336 210
\pinlabel {$3$} at 336 30
\pinlabel {$4$} at 56 260
\pinlabel {$5$} at 156 260
\pinlabel {$5$} at 284 260
\pinlabel {$4$} at 380 260
\pinlabel {$5$} at 44 344
\pinlabel {$7$} at 96 344
\pinlabel {$8$} at 146 344
\pinlabel {$7$} at 196 344
\pinlabel {$7$} at 246 344
\pinlabel {$8$} at 296 344
\pinlabel {$7$} at 346 344
\pinlabel {$5$} at 396 344

\pinlabel {$b$} at 1126 124
\pinlabel {$a$} at 494 140
\pinlabel {$a^{-1}$} at 494 94
\pinlabel {$ab$} at 836 166
\pinlabel {$a^{-1}b$} at 844 84
\pinlabel {$a^2b$} at 630 210
\pinlabel {$ab^2$} at 980 210
\pinlabel {$a^{-2}b$} at 630 40
\pinlabel {$a^{-1}b^2$} at 986 40
\pinlabel {$a^3b$} at 556 260
\pinlabel {$a^2bab$} at 726 256
\pinlabel {$abab^2$} at 886 256
\pinlabel {$ab^3$} at 1050 260
\pinlabel {$a^4b$} at 516 314
\pinlabel \rotatebox{50}{\scriptsize $a^3ba^2b$} at 610 340
\pinlabel \rotatebox{50}{\scriptsize $(a^2b)^2ab$} at 688 340
\pinlabel \rotatebox{50}{\scriptsize $a^2b(ab)^2$} at 766 340
\pinlabel \rotatebox{50}{\scriptsize $(ab)^2ab^2$} at 842 340
\pinlabel \rotatebox{50}{\scriptsize $ab(ab^2)^2$} at 916 340
\pinlabel \rotatebox{50}{\scriptsize $ab^2ab^3$} at 990 340
\pinlabel {$ab^4$} at 1090 310
\endlabellist
\centering
\includegraphics[width=1\textwidth]{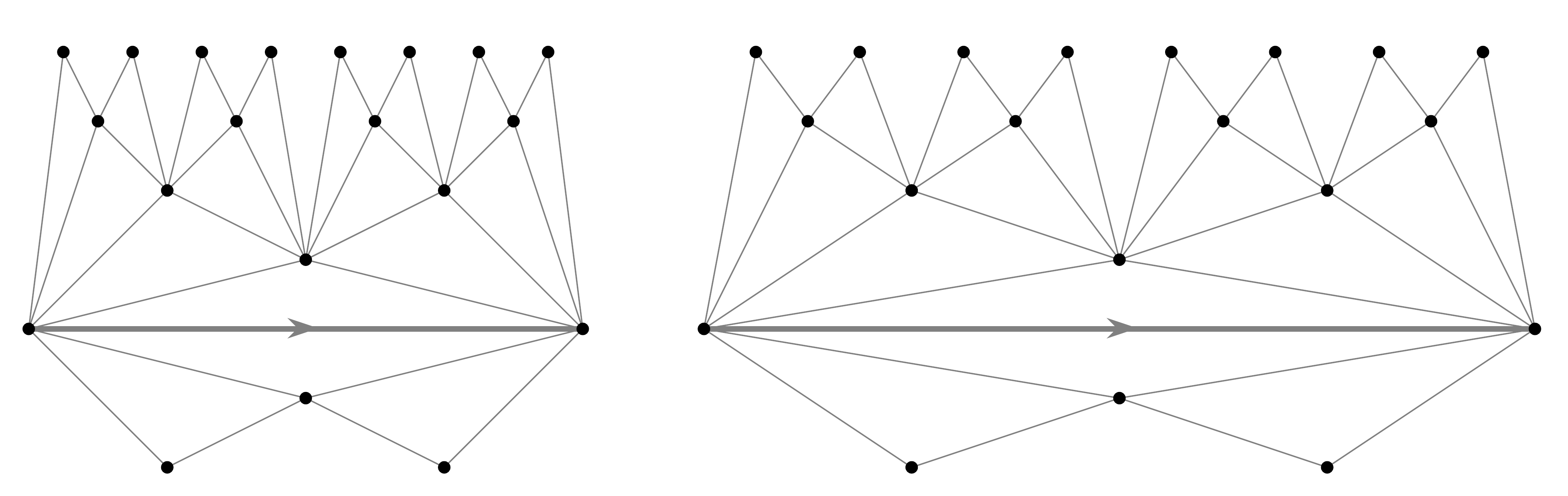}
\caption{(Left) The Fibonacci function. (Right) The Christoffel function.}
\label{fig:fibchris}
\end{figure}

\begin{definition}\label{def:chris}
Let $S$ be a set and $B:S\times S\to S$ a binary operation.
\begin{enumerate}[label=\textup{(\alph*)},nosep,leftmargin=*]
\item Let $S=\N$ and $B(p,q)=p+q$. The \emph{$\e$-Fibonacci function} $\Fi_\e$ is defined inductively on $\V$ with initial values $\Fi_\e[a]=\Fi_\e[b]=1$. 
\item Let $S=\F(a,b)$ and $B(w_1,w_2)=w_1w_2$, the concatenation of $w_1$ and $w_2$. The \emph{$\e$-Christoffel function} $\Ch_\e$ is defined inductively on $\V^+[a,b]$ with initial values $\Ch_\e[a]=a$ and $\Ch_\e[b]=b$, and on $\V^-[a,b]$ with initial values $\Ch_\e[a]=a^{-1}$ and $\Ch_\e[b]=b$. The values of $\Ch_\e$ are called \emph{$\e$-Christoffel words} (or \emph{Christoffel words} in $\{a,b\}$).
\end{enumerate}
\end{definition}

\noindent Strictly speaking, the value $\Ch_\e[a]$ is indeterminate and can be either $a$ or $a^{-1}$ depending on the context. But this will not cause any confusion later.

Before we delve into a few immediate properties of the Fibonacci and Christoffel functions, let us first adopt the following terminology: elements of the free semi-group generated by two symbols $x$ and $y$ are said to be \emph{positive} in $\{x,y\}$.

Let $[w]\in\V$ with $\Lv_\e[w]>0$. Since Christoffel words are obtained by concatenation, it is clear that $\Ch_\e[w]$ is either positive in $\{a,b\}$ (if $[w]\in\V^+[a,b]$) or positive in $\{a^{-1},b\}$ (if $[w]\in\V^-[a,b]$). This implies, in particular, that every $\e$-Christoffel word is \emph{cyclically reduced}, that is, its initial and terminal symbols are not inverse to each other. Observe also that if $[w]\in\V$ corresponds to $p/q\in\p^1(\Q)$ then $\pi_\e(\Ch_\e[w])=(p,q)$ (see \eqref{eqn:abel}, \eqref{eqn:prim} and Convention~\ref{conv:slope}). One can use the continued fraction expansion of $p/q$ to find the corresponding $\e$-Christoffel word. (See Appendix~\ref{sec:cf} for an illustration of this reverse procedure.) Therefore, $\Ch_\e[w]$ is a cyclically reduced representative of its unoriented primitive class $[w]$.

It is also apparent from the definitions that, for each vertex $[w]\in\V$, we have 
\begin{align}\label{eqn:fib1}
\Fi_\e[w]=|\,\Ch_\e[w]\,|_\e,
\end{align}
where $|\cdot|_\e$ denotes the word length with respect to the basis $\e=(a,b)$. See \cite{Bow}*{Lemma 2.2.1}. (As we shall see later in \eqref{eqn:le}, the value $\Fi_\e[w]$ is also equal to the translation length $\ell_\e(w)$ for $[w]\in\V$.) We also observe that
\begin{align}\label{eqn:fib2}
\Fi_\e[w]\ge \Lv_\e[w]+1
\end{align}
for all $[w]\in\V$.

Let us unravel the definition of $\e$-Christoffel words a little further. Let $[x,y;z]$ be a Farey triangle of level $k+1$ $(k\ge0)$ appearing in the inductive procedure and assume the edge $[x,y]$ of level $k$ is co-directed with $[\e]$. We then have $\Ch_\e[z]=\Ch_\e[x]\Ch_\e[y]$ by definition. If the pair $(\Ch_\e[x],\Ch_\e[y])$ is a basis of $\F_2$, then the pairs $(\Ch_\e[x],\Ch_\e[z])$ and $(\Ch_\e[z],\Ch_\e[y])$ are obtained by elementary Nielsen transformations (the last two in \eqref{eqn:elementary}) and thus are bases as well. Therefore, for any Farey edge $[x,y]$ co-directed with $[\e]$, we see inductively that the pair $(\Ch_\e[x],\Ch_\e[y])$ is a basis. We call such a pair an \emph{$\e$-Christoffel basis}.

\begin{remark}\label{rem:monoid}
The above inductive procedure basically amounts to considering the action of the free submonoid of $\Aut(\F_2)$ generated by two automorphisms corresponding to $\begin{psmallmatrix}1&1\\0&1\end{psmallmatrix}$ and $\begin{psmallmatrix}1&0\\1&1\end{psmallmatrix}$ in $\PGL(2,\Z)$. For more details, see Appendix~\ref{sec:cf}.
\end{remark}

\noindent Suppose now $k>0$ and let $[w;x,y]$ be the other Farey triangle adjacent to $[x,y]$. It is of level $k$. Then there are two cases depending on whether $\Lv[x]<\Lv[y]$ or $\Lv[x]>\Lv[y]$. In the former case we have $\Ch_\e[y]=\Ch_\e[x]\Ch_\e[w]$, and in the latter $\Ch_\e[x]=\Ch_\e[w]\Ch_\e[y]$.

To summarize, if $(x,y)$ is an $\e$-Christoffel basis with $\Lv_\e[x,y]\neq 0$ then it determines the quadrilateral $[x^{-1}y;x,y;xy]$. In particular, $\Lv_\e[x^{-1}y]<\Lv_\e[xy]$. We extract this property and apply it to arbitrary bases.
\begin{definition}\label{def:acute}
A basis $\f=(x,y)$ of $\F_2$ is said to be \emph{acute} relative to the basis $\e=(a,b)$ if $\Lv_\e[x^{-1}y]<\Lv_\e[xy]$.
\end{definition}
\noindent For any basis $(x,y)$ with $\Lv_\e[x,y]\neq0$, note that either $(x,y)$ and $(x^{-1},y^{-1})$ are acute relative to the basis $\e$, or else $(x^{-1},y)$ and $(x,y^{-1})$ are acute relative to the basis $\e$. 

Lastly, suppose $[x,y]$ is a Farey edge co-directed with $[\e]$ and $\Lv_\e[x,y]>0$. In order to simplify the notation let us define $\I[x,y]$ to be $\V^+[x,y]$ if $[x,y]$ lies on the left of $[\e]$, or to be $\V^-[x,y]$ if $[x,y]$ lies on the right of $[\e]$. Then the following lemma is also immediate from the inductive construction of the Christoffel words.

\begin{lemma}\label{lem:chris}
Suppose $[x,y]$ is a Farey edge co-directed with $[\e]$ and $\Lv_\e[x,y]>0$. Then the $\e$-Christoffel words in $\I[x,y]$ are positive words in $\{\Ch_\e[x], \Ch_\e[y]\}$. 
\end{lemma}

For example, the unoriented primitive class of the Christoffel word $(a^2b)^2ab$, which corresponds to $5/3$ in $\{a,b\}$, lies in $\V^+[a,ab]$. Since $(a^2b)^2ab=a(ab)a(ab)^2$, it corresponds to the positive slope $2/3$ in $\{a,ab\}$. See Figure~\ref{fig:fibchris}(Right).

\begin{remark}
Although the Christoffel words are classical mathematical objects, the terminology was introduced rather recently in the 1990s. See \cite{BLRS09} and \cite{Aig13} for more detailed accounts on the theory of Christoffel words. See also \cite{Cohn71}, \cite{Cohn72} and \cite{KR07}. In the literature of Kleinian groups and low-dimensional topology, the Christoffel words also appeared in \cite{Jor03}, \cite{JM79} and \cite{OZ81} (with correction \cite{GR99}) as well as in \cite{CMZ81} (in a somewhat weaker form). Compare also \cite{GK11}.
\end{remark}

\subsection{Level-\texorpdfstring{$n$}{n} partition of Farey vertices}

For each integer $n\ge0$ we associate a partition of the vertex set $\V$ as follows. Let $\V_{\le n}$ be the (finite) subset of $\V$ consisting of all vertices of level at most $n$. Under the embedding $\V=\p^1(\Q)\subset\p^1(\R)$, consider the complement of $\V_{\le n}$ in $\p^1(\R)$. It consists of $2^{n+1}$ open intervals $I_j$ ($1\le j\le 2^{n+1}$).
\begin{definition}\label{def:partition}
The \emph{level-$n$ partition} of $\V$ is defined to be the partition
\[
\V=\V_{\le n}\cup\I_1\cup\I_2\cup\cdots\I_{2^{n+1}},
\]
where $\I_j=I_j\cap\p^1(\Q)$ for $1\le j\le 2^{n+1}$. We still call each $\I_j$ an \emph{interval} and cyclically order the indices as in Figure~\ref{fig:n-partition}.
\end{definition}

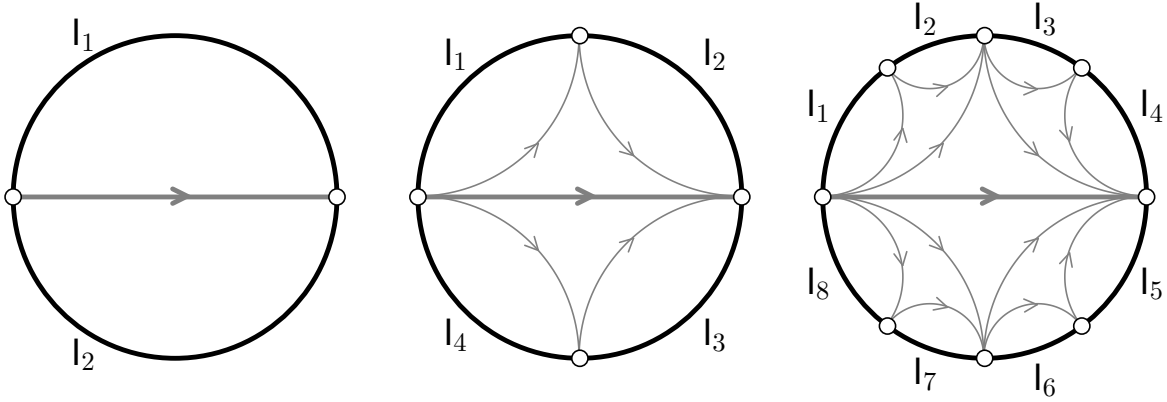
\begin{figure}[ht]
\labellist
\pinlabel {\large $\I_{1}$} at 80 224
\pinlabel {\large $\I_{2}$} at 80 26
\pinlabel {\large $\I_{1}$} at 310 212
\pinlabel {\large $\I_{2}$} at 470 212
\pinlabel {\large $\I_{4}$} at 310 38
\pinlabel {\large $\I_{3}$} at 470 38
\pinlabel {\large $\I_{1}$} at 534 180
\pinlabel {\large $\I_{8}$} at 534 70
\pinlabel {\large $\I_{2}$} at 600 234
\pinlabel {\large $\I_{7}$} at 600 14
\pinlabel {\large $\I_{3}$} at 674 234
\pinlabel {\large $\I_{6}$} at 674 10
\pinlabel {\large $\I_{4}$} at 740 180
\pinlabel {\large $\I_{5}$} at 740 70
\endlabellist
\centering
\includegraphics[width=\textwidth]{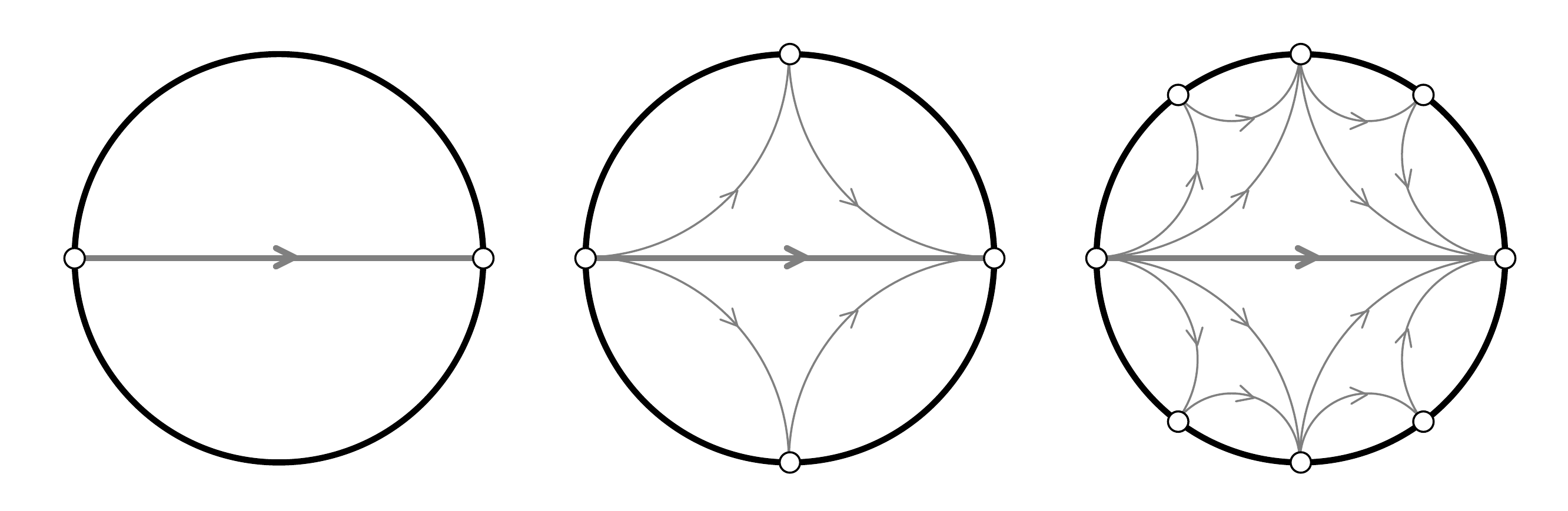}
\caption{The level-$n$ partitions of $\V$ for $n=0,1,2$.}
\label{fig:n-partition}
\end{figure}

Note that each interval $\I_j$ is either of the form $\V^+[x_j,y_j]$ if $1\le j\le2^n$ or of the form $\V^-[x_j,y_j]$ if $2^n+1\le j\le2^{n+1}$. As we did previously, we can simply write $\I_j=\I[x_j,y_j]$ for all $j$. By Lemma~\ref{lem:chris} the vertices in $\I_j$ have representatives that are positive words in $\{\Ch_\e[x_j],\Ch_\e[y_j]\}$. This property will be essential in our proof of Theorem~\ref{I}.

\section{Representations of \texorpdfstring{$\F_2$}{F2} into \texorpdfstring{$\PSL(2,\C)$}{PSL(2,C)}}\label{sec:3}

We discuss some classical facts on irreducible actions of $\F_2$ on the hyperbolic space $\H^3$. Mostly following the book \cite{Fen} of Fenchel, we emphasize the use of half-turns and the Coxeter extension. We also introduce the notion of amplitude of a right-angled hexagon.

\subsection{Isometries of \texorpdfstring{$\H^3$}{H3} and half-turns}

We recall some of the basic definitions and classical results on the orientation-preserving isometries of the $3$-dimensional hyperbolic space $(\H^3,d_\H)$. We mainly refer to the book by Fenchel \cite{Fen}.

As a model of $(\H^3,d_\H)$ we use the upper half-space model $\C\times\R^+$. The group $\Isom^+(\H^3)$ of orientation preserving isometries of $\H^3$ can be identified, via the Poincar\'e extension, with the group $\PSL(2,\C)=\SL(2,\C)/\{\pm\id\}$ of projective automorphisms of $\p^1(\C)$. See \cite{Fen}*{p.26} for more details. For a projective class $X\in\PSL(2,\C)$ we write $\ti{X}$ to denote one of its lifts in $\SL(2,\C)$.

Under the above identification, non-trivial isometries $X\in\PSL(2,\C)$ of $\H^3$ are classified by the trace of their matrix representatives $\ti{X}\in\SL(2,\C)$: $X$ is \emph{elliptic} if $\tr\ti{X}\in(-2,2)$, \emph{parabolic} if $\tr\ti{X}\in\{-2,2\}$, and \emph{loxodromic} if $\tr\ti{X}\notin[-2,2]$. See \cite{Fen}*{p.46}. If $X$ is elliptic or loxodromic, then it preserves a geodesic line $\Axis_X$ in $\H^3$, called the \emph{axis} of $X$. 

An elliptic isometry $P\in\PSL(2,\C)$ is involutory, that is, of order two if and only if $\tr\ti{P}=0$ if and only if $\ti{P}^2=-\id$; such an isometry will be called a \emph{half-turn}. It is not hard to see that every orientation-preserving isometry $X$ can be expressed
\[
X=PQ
\]
as a composition of two half-turns $P$ and $Q$, although such a decomposition is never unique. See again \cite{Fen}*{pp.46-47}.

Below we shall investigate this decomposition using matrices in $\SL(2,\C)$. For our purpose we focus only on the case when the isometry $X$ is elliptic or loxodromic. The \emph{complex translation length} of $X$
\[
\la_\H(X)=\ell_\H(X)+i\theta_\H(X)\in\A=\C/2\pi i\Z
\]
is the complex number such that its real part $\ell_{\H}(X)\ge0$ is the (real) \emph{translation length} along the axis and its imaginary part $\theta_\H(X)\in\R/2\pi\Z\cong(-\pi,\pi]$ is the \emph{rotation angle} measured according to the right-hand rule.

A half-turn $P\in\PSL(2,\C)$ determines a geodesic $\Axis_P$ in $\H^3$. The two lifts $\pm \ti{P}\in\SL(2,\C)$ of $P$ have zero trace. An important fact is that we can consistently associate each of these matrices with one of the orientations of the geodesic $\Axis_P$; when the oriented geodesic is determined by an ordered pair of its ideal end points $(u,u')\in\p^1(\C)\times\p^1(\C)$, an explicit formula of the corresponding matrix is given in terms of $u$ and $u'$ in \cite{Fen}*{p.64, (1)-(3)}. 

\begin{convention}
In order to emphasize this fact we shall deliberately abuse notation: a traceless matrix $\ti{P}\in\SL(2,\C)$ will also denote the corresponding \emph{oriented} geodesic as well.
\end{convention}

Now let $X=PQ$ be a decomposition of an elliptic or loxodromic isometry $X$ into two half-turns $P$ and $Q$. Then the axes of $P$ and $Q$ intersect the axis of $X$ orthogonally. Take arbitrary lifts $\ti{P}, \ti{Q}\in\SL(2,\C)$ and orient $\Axis_X$ arbitrarily. Given such a triple $(\ti{P},\ti{Q};\Axis_X)$ of oriented geodesics, Fenchel \cite{Fen}*{p.67} defines the \emph{width} from $\ti{Q}$ to $\ti{P}$ along $\Axis_X$
\[
\eta(\ti{Q},\ti{P})=\eta(\ti{Q},\ti{P};\Axis_X):=\pm\la_\H(H)\in\A=\C/2\pi i\Z
\]
as the \emph{signed} complex translation length of $H$, where $H$ is the isometry that has axis $\Axis_X$ and maps $\ti{Q}$ to $\ti{P}$, and the sign is provided in accordance with the orientation of $\Axis_X$. (We shall always use the simpler notation $\eta(\ti{Q},\ti{P})$ assuming implicitly that the orientation of $\Axis_X$ is understood.) See Figure~\ref{fig:width}.
\begin{figure}[ht]
\labellist
\pinlabel {$\eta(\ti{Q},\ti{P})=\pm(\ell_\H+i\theta_\H)$} at -40 180
\pinlabel {$\ell_\H$} at 130 110
\pinlabel {$\theta_\H$} at 172 196
\pinlabel \rotatebox{18}{$\Axis_X$} at 140 66
\pinlabel {$\ti{Q}$} at 106 0
\pinlabel {$\ti{P}$} at 220 10
\endlabellist
\centering
\includegraphics[width=0.3\textwidth]{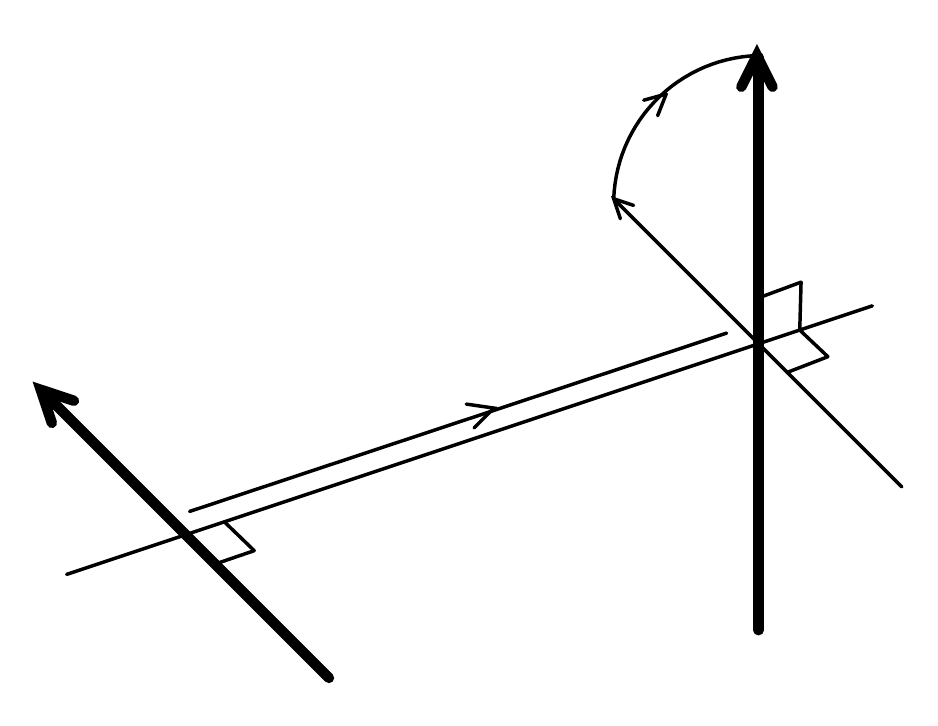}
\caption{The width $\eta(\ti{Q},\ti{P})$.}
\label{fig:width}
\end{figure}
Let $\ti{H}\in\SL(2,\C)$ be an arbitrary lift of $H$. Then Fenchel \cite{Fen}*{p.68} shows that
\[
X=H^2=PQ, \quad\ti{H}^2=-\ti{P}\ti{Q}, \quad\frac{1}{2}\tr\ti{H}^2=\cosh\eta(\ti{Q},\ti{P}).
\]

Let us set $\ti{X}=\ti{H}^2$ and assume that $\Axis_X$ is oriented naturally in the translation direction of $X$. Then $\eta(\ti{Q},\ti{P})$ is one of the two \emph{half} translation lengths of $X$, that is, $\la_\H(X)/2=\eta(\ti{Q},\ti{P})$ or $\eta(\ti{Q},\ti{P})+\pi i$, and we have
\begin{align}\label{eqn:cosh}
\frac{1}{2}\tr\ti{X}=\cosh\eta(\ti{Q},\ti{P})=\pm\cosh\frac{\la_\H(X)}{2}.
\end{align}
For a complex number $\eta=\ell+i\theta$ ($\ell\ge0$) we have $\cosh\eta=\cosh\ell\cos\theta+i\sinh\ell\sin\theta$. So we see that $|\cosh\eta|^2=\cosh^2\ell-\sin^2\theta=\sinh^2\ell+\cos^2\theta$, in particular, $\sinh\ell\le |\cosh\eta| \le \cosh\ell$. Thus from the above formula for $\tr\ti{X}$ we obtain
\begin{align}\label{eqn:ltr}
e^{\frac{1}{2}\ell_\H(X)}-1
\le |\tr\ti{X}|
\le e^{\frac{1}{2}\ell_\H(X)}+1
< e^{\frac{1}{2}\ell_\H(X)+1}
\end{align}

For a number $h>0$ its \emph{parallel angle} $\al(h)\in(0,\pi/2)$ is defined by the formula
\begin{align}\label{eqn:parallel_angle}
\sin\al(h) \cosh h=1.
\end{align}
See Fenchel \cite{Fen}*{p.92}. A geometric meaning of $\al(h)$ is given in Figure~\ref{fig:parallel_angle}(Left).

\begin{figure}[ht]
\labellist
\pinlabel {$\al(h)$} at 70 170
\pinlabel {$h$} at 80 106
\pinlabel {$\H^2$} at 180 200
\pinlabel {$\al$} at 386 74
\pinlabel {$\h$} at 320 160
\pinlabel {$\h_0$} at 450 196
\pinlabel {$Y(\h)$} at 720 210
\pinlabel \rotatebox{18}{$\Axis_Y$} at 640 122
\pinlabel {$\ti{Q}$} at 256 76
\pinlabel {$\ti{P}$} at 520 200
\endlabellist
\centering
\includegraphics[width=\textwidth]{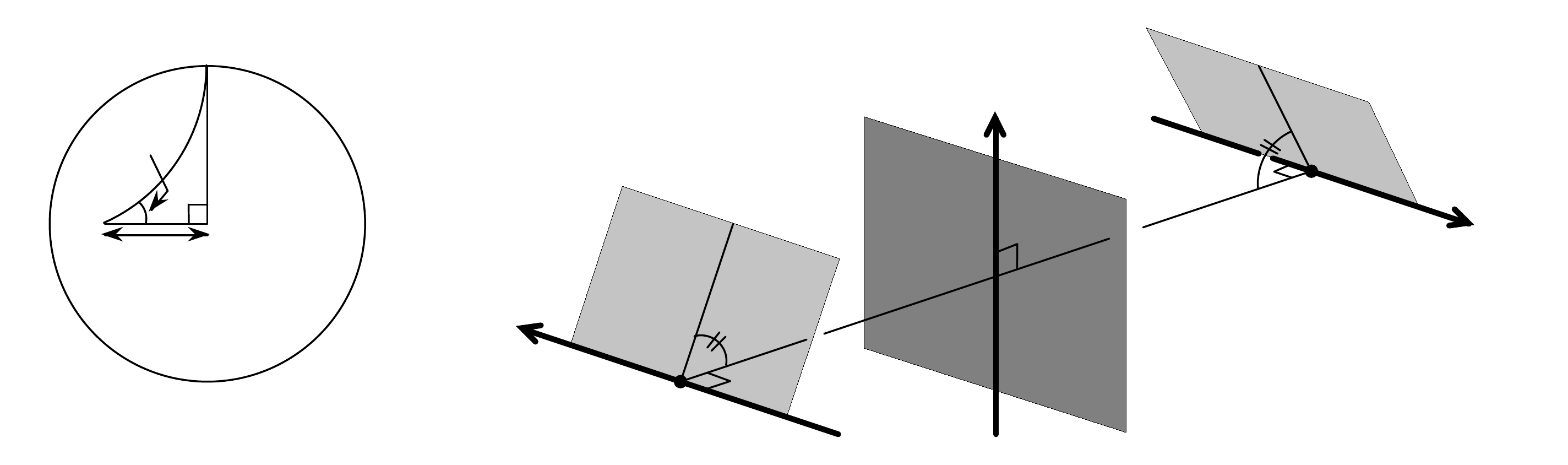}
\caption{(Left) The parallel angle $\al(h)$ of $h$. (Right) Proof of Lemma~\ref{lem:parallel_angle}.}
\label{fig:parallel_angle}
\end{figure}
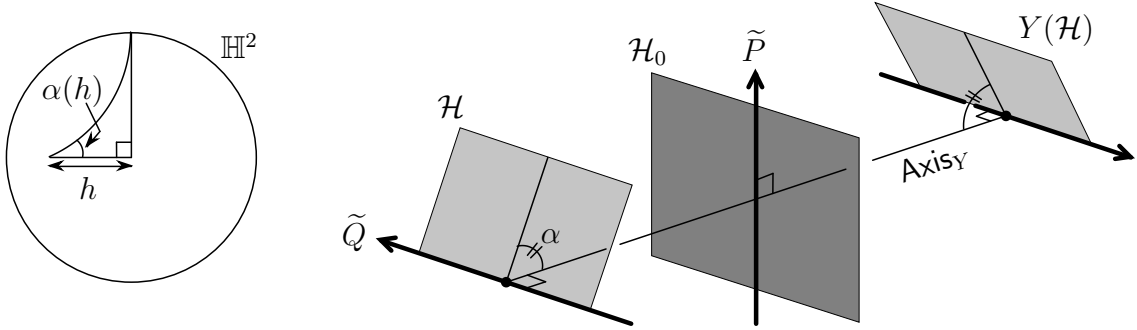

Suppose a (totally geodesic) plane $\h$ intersects the axis of a loxodromic isometry $Y$ transversely. The following lemma says that if the (real) translation length of $Y$ is large enough then, regardless of the rotation angle of $Y$, the plane $\h$ and its image $Y(\h)$ are \emph{ultra-parallel}, that is, their closures in $\overline{\H^3}$ are disjoint.

\begin{lemma}\label{lem:parallel_angle}
Let $Y$ be a loxodromic isometry and $\h\subset\H^3$ a plane intersecting $\Axis_Y$ with angle $\al\in(0,\pi/2]$. If the angle $\al$ is greater than the parallel angle $\al(\ell_\H(Y)/2)$ of the half (real) translation length of $Y$, then the plane $\h$ and the image $Y(\h)$ are ultra-parallel.
\end{lemma}
\begin{proof}
See Figure~\ref{fig:parallel_angle}(Right). We may express the isometry $Y=PQ$ as a composition of two half-turns $P$ and $Q$ so that the axis $\ti{Q}$ is contained in the plane $\h$. Let $\h_0$ denote the plane which contains the axis $\ti{P}$ and is orthogonal to $\Axis_Y$. (Note that $\h_0$ is determined by the half translation length of $Y$ independently of the rotation angle of $Y$.)

Now consider another plane $\h'$ which contains $\Axis_Y$ and is orthogonal to both $\h$ and $\h_0$. Then $\al$ is the angle between the geodesics $\h\cap\h'$ and $\Axis_Y$. By the definition of parallel angle, we conclude that the geodesics $\h\cap\h'$ and $\h_0\cap\h'$ are ultra-parallel in the plane $\h'$. Since both planes $\h$ and $\h_0$ are orthogonal to the plane $\h'$, we see that $\h$ and $\h_0$ are ultra-parallel as well. By symmetry under the half-turn $P$, we also conclude that $P(\h)=P(Q(\h))=Y(\h)$ and $P(\h_0)=\h_0$ are ultra-parallel. Therefore, the two planes $\h$ and $Y(\h)$ are separated by $\h_0$ and are ultra-parallel.
\end{proof}

\subsection{Irreducible representations and their Coxeter extensions}\label{sec:coxeter}

A homomorphism $\rho:\F(a,b)\to\PSL(2,\C)$ represents an isometric action of $\F(a,b)$ on $\H^3$. Let $(x,y)$ be another basis of $\F_2$. For simplicity we shall write
\[
X=\rho(x),\quad Y=\rho(y),\quad Z=\rho(xy)^{-1}.
\]
The trace of the commutator of $X$ and $Y$ is well-defined without ambiguity of sign for arbitrary lifts $\ti{X}$ and $\ti{Y}$. It is also independent of the choice of basis $(x,y)$ by Theorem~\ref{thm:Nielsen}(b), since $\tr A=\tr A^{-1}$ for $A\in\SL(2,\C)$. We denote this number by
\[
\ka_\rho:=\tr (XYX^{-1}Y^{-1}).
\] 
By definition, the representation $\rho$ is \emph{irreducible} if its action on $\p^1(\C)$ has no fixed point. It is known that this is the case if and only if $\ka_\rho\neq2$. See, for example, \cite{Gol09}*{Proposition 2.3.1}.

Let $\W(p,q,r)=\langle p,q,r \mid p^2=q^2=r^2=1\rangle=\Z/2\ast\Z/2\ast\Z/2$ be the free (or universal) Coxeter group of rank three. The elements $pq$ and $qr$ freely generate a subgroup $\langle pq, qr \rangle$ of index two. (The reader may recall the discussion in the introduction and Figure~\ref{fig:torus}.)

Given a basis $\f=(x,y)$ of $\F_2$, let 
\[
\psi_\f:\F(x,y)\to\W(p,q,r)
\]
denote the embedding defined by $\psi_\f(x)=pq$ and $\psi_\f(y)=qr$. We say the group $\W(p,q,r)$, together with $\psi_\f$, is the \emph{Coxeter extension} of $\F_2$ associated with the basis $\f=(x,y)$.

\begin{remark}\label{rem:2fold}
The lack of ``notational" $3$-fold symmetry in the definition of $\psi_\f$ is intentional and this way we emphasize the $2$-fold symmetry of the Farey triangulation. In Section~\ref{sec:pal}, however, we will define $\psi_\f$ differently in order to respect a $3$-fold symmetry.
\end{remark}

For irreducible representations of $\F_2$ into $\PSL(2,\C)$ we have the following classical theorem. For a proof see, for example, \cite{Jor03}*{p.185} or \cite{Fen}*{II, p.94}. See also \cite{Gol09}*{Theorem B and Theorem 3.2.2} for a more comprehensive account.
\begin{theorem}[Coxeter extension]\label{thm:coxeter}
Let $\f=(x,y)$ be a basis of $\F_2$. If $\rho:\F(x,y)\to\PSL(2,\C)$ is an irreducible representation, then there exists a unique representation $\rho_\f:\W(p,q,r)\to\PSL(2,\C)$ such that $\rho=\rho_\f\circ\psi_\f$.
\end{theorem}
\noindent We also call $\rho_\f$ the \emph{Coxeter extension} of $\rho$ associated with the basis $\f=(x,y)$. We simplify the notations by setting
\begin{align}\label{eqn:pqr}
P=\rho_\f(p),\quad Q=\rho_\f(q),\quad R=\rho_\f(r).
\end{align}
We then have $X=PQ$, $Y=QR$ and $Z=RP$.

Henceforth, we focus on the case of our most interest, namely, when $X$, $Y$ and $Z$ are all loxodromic. Then irreducibility of $\rho$ implies that the half-turns $P$, $Q$ and $R$ are uniquely defined by the following properties. The axis $\Axis_Q$ is the common perpendicular of $\Axis_X$ and $\Axis_Y$. Likewise, the axis $\Axis_P$ (resp. $\Axis_R$) is the common perpendicular of $\Axis_Z$ and $\Axis_X$ (resp. $\Axis_Y$). We denote by
\[
\hex(\rho,\f):=(\Axis_R,\; \Axis_Y,\; \Axis_Q,\; \Axis_X,\; \Axis_P,\; \Axis_Z)
\]
the cyclically ordered sextuple of the geodesics, and call it the \emph{right-angled hexagon} associated with $\rho$ and $\f=(x,y)$.

\subsection{Right-angled hexagons and their amplitudes}\label{sec:amplitude}

We continue with the previous discussion, so assume $\rho:\F(x,y)\to\PSL(2,\C)$ is irreducible and $X$, $Y$ and $Z$ are all loxodromic. Henceforth, we further assume that the geodesics $\Axis_X$, $\Axis_Y$, and $\Axis_Z$ are oriented naturally in accordance with the transformation directions of $X$, $Y$ and $Z$.

Given the right-angled hexagon $\hex(\rho,\f)$ we will be mostly interested in the angular part of the width $\eta(\Axis_Y,\Axis_X)$ between $\Axis_X$ and $\Axis_Y$. For the purpose of computation it will be convenient if we lift $\rho$ to $\SL(2,\C)$ and use the trace identities for matrices in $\SL(2,\C)$.

More precisely, we take arbitrary lifts $\ti{P}, \ti{Q}, \ti{R}\in\SL(2,\C)$ of $P, Q, R$, respectively. This gives rise to a lift $\tho:\F(x,y)\to\SL(2,\C)$ of $\rho$ defined by
\begin{align*}
\tho(x)=\ti{X}=-\ti{P}\ti{Q},\quad\tho(y)=\ti{Y}=-\ti{Q}\ti{R}
\end{align*}
In this case we have $\ti{Z}=\ti{Y}^{-1}\ti{X}^{-1}=(-\ti{R}\ti{Q})(-\ti{Q}\ti{P})=-\ti{R}\ti{P}$. Conversely, a lift $\tho:\F(x,y)\to\SL(2,\C)$ of $\rho$ determines the lifts $\ti{P}, \ti{Q}, \ti{R}\in\SL(2,\C)$ uniquely (up to reverting their signs simultaneously).

The right-angled hexagon $\hex(\rho,\f)$ is now oriented (in the sense of Fenchel \cite{Fen}*{p.79}). The cyclically ordered sextuple of the oriented geodesics will be denoted by 
\[
\hex(\tho,\f):=(\ti{R},\; \Axis_Y,\; \ti{Q},\; \Axis_X,\; \ti{P},\; \Axis_Z).
\]
See Figure~\ref{fig:hexagon}.

\begin{figure}[ht]
\labellist
\pinlabel {$\ti{P}$} at 0 90
\pinlabel {$\ti{Q}$} at 174 10
\pinlabel {$\ti{R}$} at 396 230
\pinlabel \rotatebox{-16}{$\Axis_X$} at 124 118
\pinlabel \rotatebox{16}{$\Axis_Y$} at 280 116
\pinlabel \rotatebox{-9}{$\Axis_Z$} at 240 206
\endlabellist
\centering
\includegraphics[width=0.4\textwidth]{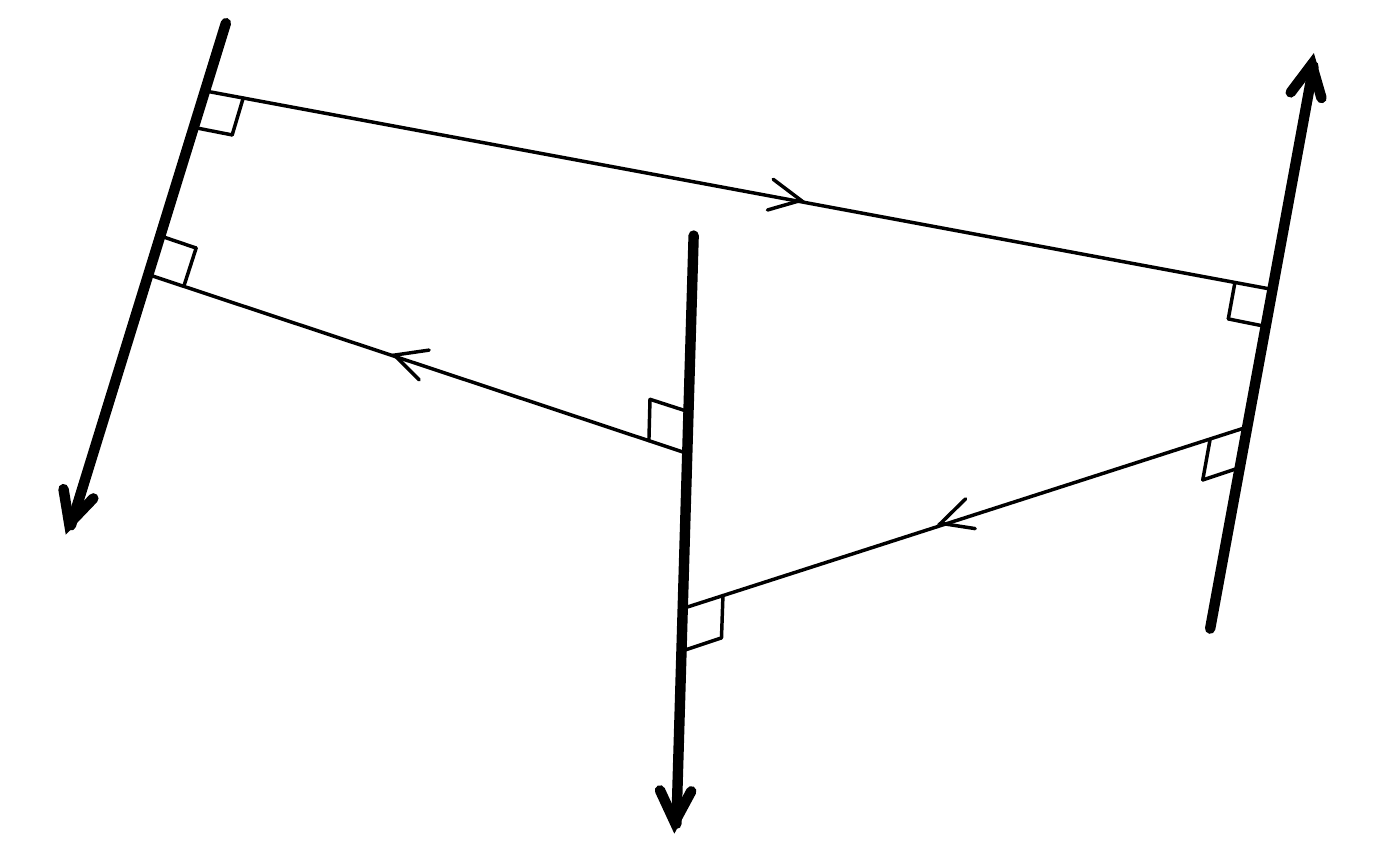}
\caption{The right-angled hexagon $\hex(\tho,\f)$.}
\label{fig:hexagon}
\end{figure}

For brevity let us set
\[
\eta_X=\eta(\ti{Q},\ti{P}),\quad
\eta_Y=\eta(\ti{R},\ti{Q}),\quad
\eta_Z=\eta(\ti{P},\ti{R}),\quad
\eta_Q=\eta(\Axis_Y,\Axis_X).
\]
In fact, the width $\eta_Q$ is determined by $\tho$ only up to sign, since we may change the signs of $\ti{P}, \ti{Q}, \ti{R}$ simultaneously and obtain the same lift $\tho$. However, if we set
\begin{align}\label{eqn:angle}
\theta(\rho,\f)=|\mathrm{Im}\,\eta_Q|\in[0,\pi]
\end{align}
this number is independent of the choice of the lift $\tho$. We call it the \emph{angle} associated to $\rho$ and $\f$.

We shall use the following form of the law of Cosines for $\hex(\tho,\f)$:
\begin{align*}
\cosh\eta_Z
=\cosh\eta_X\cosh\eta_Y+\sinh\eta_X\sinh\eta_Y\cosh\eta_Q.
\end{align*}
See \cite{Fen}*{p.83}. Since we have from \eqref{eqn:cosh}
\begin{align*}
\tr\ti{X}&=-\tr\ti{P}\ti{Q}=2\cosh \eta_X\\
\tr\ti{Y}&=-\tr\ti{Q}\ti{R}=2\cosh \eta_Y\\
\tr\ti{Z}&=-\tr\ti{R}\ti{P}=2\cosh \eta_Z
\end{align*}
the law of Cosines can be written as
\begin{align}\label{eqn:cosine}
\frac{2\,\tr\ti{Z}}{\tr\ti{X}\tr\ti{Y}}=\frac{\cosh\eta_Z}{\cosh\eta_X\cosh\eta_Y}
=1+\tanh\eta_X\tanh\eta_Y\cosh\eta_Q,
\end{align}
which we shall need later. Note that
\begin{align}\label{eqn:tanh}
\mathrm{Re}(\tanh\eta_X)>0\textup{ if and only if }\mathrm{Re}(\eta_X)>0,
\end{align}
the latter of which is the case as we assumed $X$ is loxodromic and $\Axis_X$ is oriented naturally.

Lastly, to the oriented right-angled hexagon $\hex(\tho,\f)$, Fenchel \cite{Fen}*{p.102} assigns a complex number, called its \emph{amplitude}, defined by
\begin{align*}
\am(\tho,\f)=-\frac{1}{2}\tr\ti{P}\ti{Q}\ti{R}.
\end{align*}
He then draws many formulas for the amplitude in geometric terms of the right-angled hexagon $\hex(\tho,\f)$. One of the formulas we shall need later is the following \cite{Fen}*{p.103, (3)}:
\begin{align}\label{eqn:amplitude}
\am(\tho,\f)=-i\sinh\eta_X\sinh\eta_Y\sinh\eta_Q.
\end{align}

The amplitude is closely related to the commutator trace $\ka_\rho$ as follows:
\begin{align}\label{eqn:kappa}
\begin{aligned}
\ti{X}\ti{Y}^{-1}\ti{X}^{-1}\ti{Y}&=\ti{P}\ti{Q}\ti{R}(\ti{Q}\ti{Q})\ti{P}\ti{Q}\ti{R}=-(\ti{P}\ti{Q}\ti{R})^2,\\
\ka_\rho=\tr\ti{X}\ti{Y}^{-1}\ti{X}^{-1}\ti{Y}&=-\tr(\ti{P}\ti{Q}\ti{R})^2=2-\tr^2\ti{P}\ti{Q}\ti{R}=2-4\am^2(\tho,\f).
\end{aligned}
\end{align}
In particular, since $\ka_\rho$ is a constant independent of the basis $\f=(x,y)$, so is the modulus $|\am(\tho,\f)|$.

\section{Primitive stability and the Q-conditions}\label{sec:4}

We define two open conditions on isometric actions $\rho:\F_2\to\PSL(2,\C)\cong\Isom^+(\H^3)$ of $\F_2$ on $(\H^3,d_\H)$: Minsky's primitive stability (Section~\ref{sec:ps}) and Bowditch's Q-conditions (Section~\ref{sec:bq}). In Section~\ref{sec:psbq} we show that the former implies the latter.

\subsection{The isometric action of \texorpdfstring{$\F_2$}{F2} on itself}

Recall that we fixed a basis $\e=(a,b)$ of $\F_2$. As before, we write $|w|_\e=|w|_{(a,b)}$ to denote the \emph{word length} of $w\in\F_2$ with respect to the basis $\e=(a,b)$. 

The \emph{word metric} $d_\e$ on $\F_2$ is defined by $d_\e(u,v)=|u^{-1}v|_\e$ for $u,v\in\F_2$. Then $\F_2$ acts freely and isometrically on the metric space $(\F_2,d_\e)$ by left multiplication. For an element $w\in\F_2$ its \emph{translation length} $\ell_\e(w)$ is defined by
\begin{align}\label{eqn:le}
\ell_\e(w):=\inf_{u\in\F_2}d_\e(u,wu)
\left(=\inf_{u\in\F_2}|u^{-1}wu|_\e=|r|_\e=:\|w\|_\e\right).
\end{align}
Algebraically, this can be interpreted as the \emph{cyclically reduced length} $\|w\|_\e$ of $w$, that is, the word length $|r|_\e$ of a cyclically reduced representative $r$ of the conjugacy class of $w$. Thus the translation length $\ell_\e(w)$ depends only on the $\sim$-equivalence class of $w$.

Note that if $w\in\F_2$ is not the identity then the action of $w$ on $(\F_2,d_\e)$ is \emph{axial}: there exists a (discrete) geodesic $\Z\to(\F_2,d_\e)$ which is invariant under the action of $w$ and is translated non-trivially through the amount $\ell_\e(w)$. We denote the image of this geodesic by $\Axis^\e(w)$ (or simply $\Axis(w)$ when $\e$ is understood) and call it the \emph{axis} of $w$. Note that if $w=uru^{-1}$ then $\Axis(w)=u\Axis(r)$. If $r$ is cyclically reduced, then $\Axis(r)$ contains the cyclic orbit $\{r^i\mid i\in\Z\}$ of the identity; in particular, it passes through the identity.

Similarly, for any basis $\f=(x,y)$, we denote by $\Axis^\f(w)\subset(\F_2,d_\f)$ the axis of a non-trivial element $w$ with respect to $\f$. 

\begin{lemma}\label{lem:axis}
Let $\f=(x,y)$ be an $\e$-Christoffel basis of $\F_2$ with $\Lv_\e[x,y]>0$. If $w$ is a positive word in $\{x,y\}$, then the axis $\Axis^\f(w)$ is of the form $\{g_i \mid i\in\Z\}$, where $g_0=1$ and $g_i^{-1}g_{i+1}\in\{x,y\}$ for all $i\in\Z$, and is contained in the axis $\Axis(w)\subset(\F_2,d_\e)$.
\end{lemma}

\begin{proof}
If $w$ is positive in $\{x,y\}$ then it is of the form $w=s_1s_2\ldots s_n$ for some $n>0$ where each $s_i\in\{x,y\}$. Since $\f=(x,y)$ is an $\e$-Christoffel basis, $x$ and $y$ are $\e$-Christoffel words and they are both positive in $\{a,b\}$ or both positive in $\{a^{-1},b\}$. (Recall the discussion after Definition~\ref{def:chris}.) Therefore, in either case, $w=s_1s_2\ldots s_n$ is linearly and cyclically reduced as a word in $\{a,b\}$. The lemma now follows.
\end{proof}

\subsection{Minsky's primitive stability}\label{sec:ps}

Minsky \cite{Min} introduced the notation of primitive stability and studied its basic properties. Although the definition makes sense for the free group $\mathsf{F}_n$ of arbitrary rank $n\ge2$, we restrict to the case $n=2$ for simplicity.

Recall that we endowed $\F_2$ with a metric $d_\e$, where $\e=(a,b)$ is the distinguished basis. Suppose a representation $\rho:\F_2\to\PSL(2,\C)$ is given. Upon choosing a base point $o\in\H^3$ the orbit map $\tau_{\rho,o}:\F_2\to\H^3$ is defined by $\tau_{\rho,o}(w)=\rho(w)(o)$. Obviously, it is $\rho$-equivariant: $\tau_{\rho,o}(uw)=\rho(uw)(o)=\rho(u)(\tau_{\rho,o}(w))$ for $u, w\in\F_2$.

\begin{definition}\label{def:ps}
A representation $\rho:\F_2\to\PSL(2,\C)$ is said to be \emph{primitive stable} if there exist a point $o\in\H^3$ and constants $M,c>0$ such that the $\tau_{\rho,o}$-images of the axes $\Axis(x)\subset(\F_2,d_\e)$ of primitive elements $x$ of $\F_2$ are $(M,c)$-quasi-geodesics in $\H^3$, that is,
\[
\frac{1}{M}d_\e(u,v)-c\le d_\H(\rho(u)(o),\rho(v)(o))\le M d_\e(u,v)+c
\]
for all $u,v\in\Axis(x)$.
\end{definition}

\noindent See \cite{Min}*{Definition 3.1}. Note that the property is preserved under the $\PSL(2,\C)$-conjugacy action, so we can consider this definition in the character variety $\X(\F_2)$. As in the introduction we denote by $\PS\subset\X(\F_2)$ the subset consisting of all $\PSL(2,\C)$-conjugacy classes of primitive stable representations.

The definition of primitive stability can be simplified a little. First of all, from the equivariance of $\tau_{\rho,o}$ and the triangle inequality in $\H^3$, one can deduce that the constants $(M,c)$ satisfying the upper inequality above exist for \emph{any} isometric action $\rho$ and for any choice of base point $o\in\H^3$. See \cite{BH}*{Lemma I.8.18} for example. Thus the lower inequality is the only essential requirement for primitive stability.

Furthermore, the requirement for each primitive element is rather for each (unoriented) primitive conjugacy class. For if $x=uru^{-1}$ then $\Axis(x)=u\Axis(r)$ and hence $\tau_{\rho,o}(\Axis(x))=\rho(u)\tau_{\rho,o}(\Axis(r))$ by the equivariance of $\tau_{\rho,o}$. So it suffices to test the inequality only for cyclically reduced representatives.

To summarize, we can simplify the above definition as follows. Henceforth we shall always use this alternative definition.

\begin{definition}[Alternative]\label{def:altps}
A representation $\rho:\F_2\to\PSL(2,\C)$ is said to be \emph{primitive stable} if there exist a point $o\in\H^3$ and constants $m,c>0$ satisfying the following condition: for each unoriented primitive class $[x]\in\V$, there is a cyclically reduced representative $x$ such that
\begin{align*}
m\cdot d_\e(u,v)-c\le d_\H(\rho(u)(o),\rho(v)(o))
\end{align*}
for all $u,v\in\Axis(x)$.
\end{definition}

\subsection{Bowditch's Q-conditions}\label{sec:bq}

Note that, for an element $X\in\PSL(2,\C)$ and its two lifts $\pm\ti{X}\in\SL(2,\C)$, we have $\tr(-\ti{X})=-\tr\ti{X}$ and $\tr\ti{X} = \tr\ti{X}^{-1}$. So, given a representation $\rho:\F_2\to\PSL(2,\C)$, the complex modulus $|\tr(\rho(x))|$ is a well-defined function on the set $\V=\Prim/_\sim$ of unoriented primitive classes of $\F_2$.

\begin{definition}\label{def:bq}
A representation $\rho:\F_2\to\PSL(2,\C)$ is said to satisfy the \emph{Q-conditions} if
\begin{enumerate}[label=\textup{(\roman*)},nosep,leftmargin=*]
\item $\rho(x)$ is loxodromic for all $[x]\in\V$, and
\item $|\tr \rho(x)|\le2$ for only finitely many $[x]\in\V$.
\end{enumerate}
\end{definition}

\noindent See \cite{Bow}*{p.702} and \cite{TWZ08}*{p.765}. The conditions (i) and (ii) will be referred to as the first and second Q-condition, respectively. Since the trace function is invariant under conjugacy, this definition also makes sense in the character variety $\X(\F_2)$. We denote by $\BQ\subset\X(\F_2)$ the subset consisting of all $\PSL(2,\C)$-conjugacy classes of representations satisfying the Q-conditions.

One of the main results of Bowditch \cite{Bow}*{Theorem 2 and Proposition 4.9} and Tan, Wong and Zhang \cite{TWZ08}*{Theorem 3.3} is the following characterization of the second Q-condition. We rephrased their theorem using the equality (\ref{eqn:fib1}).

\begin{theorem}\label{thm:bqpd}
Suppose a representation $\rho:\F_2\to\PSL(2,\C)$ satisfies the first Q-condition. Then $\rho$ satisfies the Q-conditions if and only if there exist uniform positive constants $m=m(\rho)$ and $c=c(\rho)$ such that
\[
\log|\tr\rho(x)| \ge m \cdot\ell_\e(x)-c
\]
for every $[x]\in\V$.
\end{theorem}
\noindent As a consequence we see that, for a representation $\rho$ satisfying the Q-conditions, the set $\{\pm\tr\rho(x) \mid [x]\in\V\}$ is a discrete subset of $\C$.

In the course of proving the above theorem, the authors investigated flows on the topograph (see Remark~\ref{rem:topograph}) and showed the existence of a ``finite attracting subtree." See \cite{Bow}*{Corollary 3.12 and Lemma 3.15} and \cite{TWZ08}*{Lemmas 3.21 and 3.24}. This fact can be reformulated in our terminology as follows: 
\begin{proposition}\label{prop:subtree}
If a representation $\rho:\F_2\to\PSL(2,\C)$ satisfies the Q-conditions, then there exists a uniform constant $N_\rho>0$ with the following property: if a Farey edge $[x,y]$ has $\Lv[x,y]\ge N_\rho$ and determines the quadrilateral $[w;x,y;z]$ then 
\[
|\tr\rho(w)|\le|\tr\rho(z)|.
\]
\end{proposition}

Consider the level-$N_\rho$ partition (see Definition~\ref{def:partition}). Very roughly speaking, then the proposition implies that, for a representation $\rho$ satisfying the Q-conditions, if we start with a Farey triangle and run the process of minimizing the complex modulus of vertex traces then the process always ends up with only a finitely many Farey triangles that are of level $\le N_\rho$. Compare with the $\PSL(2,\R)$ case in the introduction: if the axes of two hyperbolic translations intersect, then the trace minimizing process always stops at an acute (or right-angled) triangle.

Let $(x,y)$ be an acute basis of $\F_2$ with respect to $\e=(a,b)$ so that the Farey edge $[x,y]$ determines the quadrilateral $[x^{-1}y;x,y;xy]$. (See Definition~\ref{def:acute}.) For an arbitrary lift $\tho$ to $\SL(2,\C)$ one can deduce the trace identity
\[
\tr\tho(x^{-1}y)+\tr\tho(xy)=\tr\tho(x)\tr\tho(y)
\]
from the Cayley-Hamilton theorem. (See, for example, \cite{Gol09}*{Theorem A and Section 2.2} for a proof.) Then the inequality in Proposition~\ref{prop:subtree} is easily seen to be equivalent to
\begin{align}\label{eqn:re}
\mathrm{Re}\left(\frac{\tr\tho(xy)}{\tr\tho(x)\tr\tho(y)}\right)
\ge\frac{1}{2}.
\end{align}

\subsection{Primitive stability implies the Q-conditions}\label{sec:psbq}
Although the converse to our Theorem~\ref{I} is well-known, we provide a proof for the sake of completeness. Our proof is essentially the same as the proof in \cite{DGLM}*{Proposition 4.0.5}.

\begin{proposition}\label{prop:psbq}
If $\rho:\F_2\to\PSL(2,\C)$ is primitive stable, then it satisfies the Q-conditions.
\end{proposition}

\begin{proof}
Suppose $\rho$ is primitive stable. By Definition~\ref{def:altps} (with $v=1\in\F_2$) there exist a point $o\in\H^3$ and constants $m,c>0$ such that, for each unoriented primitive class $[x]\in\V$, there is a cyclically reduced representative $x$ such that
\begin{align*}
m|u|_\e-c\le d_\H(o,\rho(u)(o))
\end{align*}
for all $u\in\Axis(x)$. In particular, the $\rho$-orbit $\{\rho(x^n)(o)\mid n\in\Z\}$ of the point $o$ under the cyclic subgroup $\langle x\rangle$ generated by $x$ is a quasi-geodesic. This is the case only if $\rho(x)$ is loxodromic. Therefore, $\rho$ satisfies the first Q-condition.

Moreover, setting $u=x^n\in\Axis(x)$ in the above inequality, we have
\begin{align*}
m |x^n|_\e-c
\le d_\H(o,\rho(x^n)(o))
\end{align*}
for any $n\in\N$. Therefore, we see that
\begin{align*}
\ell_{\H}(\rho(x))
&\ge\lim_{n\to\infty}\frac{1}{n}d_\H(o,\rho(x^n)(o))\\
&\ge\lim_{n\to\infty}\frac{1}{n}\left(m |x^n|_\e-c\right)\\
&=m |x|_\e\\
&=m \cdot\ell_\e(x),
\end{align*}
where the first inequality (which is actually equality in the present case) follows from the triangle inequality and the fact that the limit on the right-hand side is independent of the point $o$. See, for example, \cite{CDP90}*{Section 10.6} or \cite{BH}*{Exercise II.6.6}. Since $|\tr\rho(x)|\ge e^{\frac{1}{2}\ell_{\H}(\rho(x))}-1$ by (\ref{eqn:ltr}), we have $|\tr\rho(x)|\ge e^{\frac{m}{2}\ell_\e(x)}-1$ for all classes $[x]\in\V$.

On the other hand, for $[x]\in\V$ we have $\ell_\e(x)=|\,\Ch_\e[x]\,|_\e=\Fi_\e[x]$ from \eqref{eqn:fib1} and \eqref{eqn:le} since $\Ch_\e[x]$ is cyclically reduced. Thus we can possibly have $e^{\frac{m}{2}\ell_\e(x)}-1\le2$, that is, $\Fi_\e[x]\le\frac{2}{m}\log3$ for only finitely many classes $[x]\in\V$; recall Definition~\ref{def:chris}(a) and Figure~\ref{fig:fibchris}. So we conclude that
\[
|\tr\rho(x)|\ge e^{\frac{m}{2}\ell_\e(x)}-1>2
\] for all but finitely many classes $[x]\in\V$, and thus $\rho$ satisfies the second Q-condition.
\end{proof}

\begin{remark}
Since $\H^3$ is Gromov $\de$-hyperbolic, we may use the Morse lemma (stability of quasi-geodesics) to give an alternative proof of the above proposition as in \cite{Lup}*{Proposition 2.9}. See also \cite{Can15}*{Proposition 2.1}. The above proof is stronger in that it requires no assumption on the metric space on which $\F_2$ acts isometrically.
\end{remark}

\section{The Q-conditions imply primitive stability}\label{sec:pf}

In this section we prove Theorem~\ref{I}. We first establish Lemma~\ref{lem:angle} as a consequence of Theorem~\ref{thm:bqpd} and Proposition~\ref{prop:subtree}. This lemma is crucial in that it enables us to consider the level-$N$ partition of $\V$ (Definition~\ref{def:partition}) for some $N>0$ and to focus only on one interval $\I_j$ thereof. This is done in Theorem~\ref{thm:main}, which is in fact slightly more general than what is needed for the proof of Theorem~\ref{I}, in the sense that, for such an interval $\I_j=\I(x,y)$, it deals not only with the primitive classes $[w]\in\I(x,y)$ but with \emph{all} positive words in $x$ and $y$. 

The proof of Theorem~\ref{I} starts in Section~\ref{sec:pf1}.

\subsection{Lemma on the angle \texorpdfstring{$\theta(\rho,\f)$}{theta(rho,f)}}

Suppose $\rho:\F_2\to\PSL(2,\C)$ satisfies the Q-conditions. It is known \cite{TWZ081}*{Theorem 1.4} that $\rho$ is irreducible. Since $\rho$ is irreducible and satisfies the first Q-condition, we can follow the discussion in Section~\ref{sec:amplitude}. In particular, the angle $\theta(\rho,\f)$ is defined for any basis $\f=(x,y)$ of $\F_2$. See \eqref{eqn:angle}.

The following lemma will be crucial in our proof of Theorem~\ref{I}. It says that, for an acute basis $\f$, its angle $\theta(\rho,\f)$ is bounded above by a quantity depending only on its level. The same statement also appears in \cite{TX18}*{Proposition 4.12} with a different proof. Recall Definition~\ref{def:acute} for acute bases.

\begin{lemma}\label{lem:angle}
Suppose $\rho:\F_2\to\PSL(2,\C)$ satisfies the Q-conditions. Let $\f=(x,y)$ be a basis of $\F_2$ that is acute relative to the basis $\e=(a,b)$. Then for any $\varepsilon\in(0,\pi)$ there exists an integer $N>0$ such that if $\Lv_\e[x,y]\ge N$ then $0\le\theta(\rho,\f)<\varepsilon$.
\end{lemma}

\begin{proof}
Let $\f=(x,y)$ be a basis with $\Lv_\e[x,y]>0$. Without loss of generality, we may assume that $\Lv_\e[x]<\Lv_\e[y]$. (The other case $\Lv_\e[x]>\Lv_\e[y]$ can be dealt with in a similar fashion switching the roles of $x$ and $y$.) Then we have
\[
\Lv_\e[x,y]=\Lv_\e[y].
\]
See Figure~\ref{fig:level}(Right). Furthermore, by \eqref{eqn:fib1}, \eqref{eqn:fib2} and \eqref{eqn:le}, we have the inequality $\ell_\e(y)\ge\Lv_\e[y]+1$ as well.

We now proceed as in Section~\ref{sec:amplitude} and employ the same notations therein. We take a lift $\tho$ to $\SL(2,\C)$ and obtain the oriented right-angled hexagon $\hex(\tho,\f)$ by choosing a set of lifts of $P,Q,R$ accordingly. Then, as we observed in \eqref{eqn:amplitude} and \eqref{eqn:kappa}, the number
\[
|\am(\tho,\f)|=|\sinh\eta_X\sinh\eta_Y\sinh\eta_Q|
\]
is a constant which depends only on $\rho$.

Below we shall adopt the following terminology: if a number $a(\f)$ is a function of a variable basis $\f$, by \emph{uniform} convergence (or divergence) of the number $a(\f)$ we mean that the convergence is controlled by a quantity depending only on the level $\Lv_\e[\f]$ of $\f$.

We first note that
\[
\sinh\eta_X=0\quad\Longleftrightarrow\quad\eta_X=0\textup{ or }i\pi\quad\Longleftrightarrow\quad\tr\ti{X}=2\cosh\eta_X=\pm2.
\]
Since $\rho$ satisfies the Q-conditions, we know from Theorem~\ref{thm:bqpd} that the set $\{\tr\ti{X} \mid [x]\in\V\}$ does not accumulate at $\pm2$. Thus there exists a uniform bound $\ep_\rho>0$ such that $|\sinh\eta_X|\ge\ep_\rho$ for all $[x]\in\V$. On the other hand, by Theorem~\ref{thm:bqpd} again, we have 
\begin{align}\label{eqn:coshy}
\log|2\cosh\eta_Y|=\log|\tr\ti{Y}|\ge m \cdot\ell_\e(y)-c\ge m(\Lv_\e[y]+1)-c
\end{align}
for some uniform constants $m,c>0$. So $|\sinh\eta_Y|$, as well as $|\cosh\eta_Y|$, diverges uniformly to infinity as $\Lv_\e[x,y]=\Lv_\e[y]$ tends to infinity. In conclusion, we have
\[
|\sinh\eta_Q|
=\frac{|\am(\tho,\f)|}{|\sinh\eta_X\sinh\eta_Y|}
\le\frac{|\am(\tho,\f)|}{\ep_\rho|\sinh\eta_Y|}
=\frac{\textup{Const.}}{|\sinh\eta_Y|},
\]
and thus $\sinh\eta_Q$ must converge uniformly to $0$ as $\Lv_\e[x,y]\to\infty$.

This means that the possible accumulation points of the set
\[
\{ \cosh\eta_Q \mid \f\textup{ is a basis of }\F_2 \}
\]
are $1$ and $-1$. However, we claim that if we consider only the bases that are \emph{acute} relative to $\e$ then $-1$ cannot be an accumulation point of the corresponding subset.

To see this, suppose on the contrary that there is a sequence of acute bases $\{\f_i=(x_i,y_i)\}_i$ such that $\cosh\eta_{Q_i}$ converges to $-1$ as $i\to\infty$. By passing to a subsequence we may further assume, without loss of generality, that $\Lv_\e[x_i]<\Lv_\e[y_i]$ for all $i$. Now recall the equation \eqref{eqn:cosine}:
\[
\frac{2\,\tr\ti{Z_i}}{\tr\ti{X_i}\tr\ti{Y_i}}
=1+\tanh\eta_{X_i}\tanh\eta_{Y_i}\cosh\eta_{Q_i},
\]
where $\ti{Z_i}=(\ti{X_i}\ti{Y_i})^{-1}$. Note from \eqref{eqn:tanh} that $\mathrm{Re}(\tanh\eta_{X_i})>0$, and from \eqref{eqn:coshy} that $\tanh\eta_{Y_i}$ converges to $1$.
Thus we have $\mathrm{Re}\left(\frac{\tr\ti{Z_i}}{\tr\ti{X_i}\tr\ti{Y_i}}\right)
<\frac{1}{2}$ for all sufficiently large $i$. But this contradicts Proposition~\ref{prop:subtree} and the inequality \eqref{eqn:re}, since we would eventually have $\Lv[x_i,y_i]\ge N_\rho$ for large $i$.

Therefore, we conclude that the set 
\[
\{ \cosh\eta_Q \mid \f\textup{ is an acute basis of }\F_2\textup{ relative to }\e\}
\]
has a unique accumulation point $1$. This means that, for an acute basis $\f=(x,y)$, the width $\eta_Q$, as well as $\theta(\rho,\f)=|\mathrm{Im}\,\eta_Q|$, converges uniformly to $0$ as $\Lv_\e[x,y]\to\infty$.
\end{proof}

\begin{remark}
By a similar argument one can show that the widths
\[
\eta(\Axis_X,\Axis_Z)\;\textup{ and }\;\eta(\Axis_Z,\Axis_Y)
\]
both converge uniformly to $\pi i\in\A=\C/2\pi i\Z$ as $\Lv_\e[x,y]\to\infty$. See also \cite{TX18}*{Proposition 4.12}. In this sense we may say that, as $\Lv_\e[\f]\to\infty$, the hexagon $\hex(\rho,\f)$ asymptotically becomes a ``degenerate obtuse triangle" as in the case of $\PSL(2,\R)\cong\Isom^+(\H^2)$ illustrated in the introduction.
\end{remark}

\subsection{Positive words in a high level Christoffel basis}

The following theorem \emph{is} our main theorem and contains the most general statement, from which Theorem~\ref{I} will follow almost immediately. 

\begin{theorem}\label{thm:main}
Suppose $\rho:\F_2\to\PSL(2,\C)$ satisfies the Q-conditions. Then there exists a number $N>0$ with the following property: if $\f=(x,y)$ is an $\e$-Christoffel basis of $\F_2$ with $\Lv_\e[x,y]\ge N$ then there exist a point $o_\f\in\H^3$ and constants $m_\f,c_\f>0$ such that, for every positive word $w$ in $\{x,y\}$, the inequality
\begin{align*}
m_\f\cdot d_\e(u,v)-c_\f\le d_\H(\rho(u)(o_\f),\rho(v)(o_\f))
\end{align*}
holds for all $u,v\in\Axis(w)$.
\end{theorem}

\begin{proof}
Suppose $\rho$ satisfies the Q-conditions. By Lemma~\ref{lem:angle} we can take $N_1>0$ large enough that we have $0\le\theta(\rho,\f)<\pi/4$ for every acute basis $\f=(x,y)$ with $\Lv_\e[x,y]\ge N_1$.

By Theorem~\ref{thm:bqpd}, there are constants $m, c>0$ such that $\log|\tr\rho(w)| \ge m \cdot\ell_\e(w) -c$
for all $[w]\in\V$. We have from \eqref{eqn:fib1} and \eqref{eqn:fib2} that $\ell_\e(w)=|\,\Ch_\e[w]\,|_\e=\Fi_\e[w]\ge \Lv_\e[w]+1$. Note also from \eqref{eqn:ltr} that $\log|\tr\rho(w)|<\frac{1}{2}\ell_\H(\rho(w))+1$. Thus we see that
\[
\frac{1}{2}\ell_{\H}(\rho(w))
> \log|\tr\rho(w)| -1
\ge m\cdot\ell_\e(w) -(c+1)
\ge m (\Lv_\e[w]+1) -(c+1)  
\]
for all $[w]\in\V$. Take $N_2>0$ big enough so that if $\Lv_\e[w]\ge N_2$ then the parallel angle \eqref{eqn:parallel_angle} of the half translation length $\frac{1}{2}\ell_{\H}(\rho(w))$ is less than $\pi/4$.

Set $N=\max\{N_1,N_2\}$.

Suppose $\f=(x,y)$ is an $\e$-Christoffel basis and $\Lv_\e[x,y]\ge N$. In particular, $\f$ is acute relative to $\e$ (see the discussion preceding Definition~\ref{def:acute}). Without loss of generality, we may further assume $\Lv_\e[x]<\Lv_\e[y]$ so that $\Lv_\e[x,y]=\Lv_\e[y]$.

Proceeding as in Section~\ref{sec:amplitude} we take a lift $\tho$ to $\SL(2,\C)$ and obtain the oriented right-angled hexagon $\hex(\tho,\f)$ by choosing a set of lifts of $P,Q,R$ accordingly. Take the (totally geodesic) plane $\h\subset\H^3$ which contains the geodesic $\ti{Q}$ and is perpendicular to $\Axis_X$. See Figure~\ref{fig:hyperplane}. 

\begin{figure}[ht]
\labellist
\pinlabel {$\h$} at 300 346
\pinlabel \rotatebox{45}{$\rho(x)(\h)$} at 50 350
\pinlabel \rotatebox{-20}{$\rho(y)^{-1}(\h)$} at 550 354
\pinlabel {$\ti{P}$} at 160 160 
\pinlabel {$\ti{Q}$} at 280 90
\pinlabel {$\ti{R}$} at 450 176 
\pinlabel \rotatebox{-10}{$\Axis_X$} at 210 180 
\pinlabel \rotatebox{12}{$\Axis_Y$} at 360 170 
\pinlabel \rotatebox{5}{$\Axis_{XY}$} at 314 260
\pinlabel {(Top view along $\ti{Q}$)} at 650 30
\pinlabel {$\theta(\rho,\f)$} at 480 60
\pinlabel {$\theta'$} at 580 90
\pinlabel {$\h$} at 540 120
\endlabellist
\centering
\includegraphics[width=\textwidth]{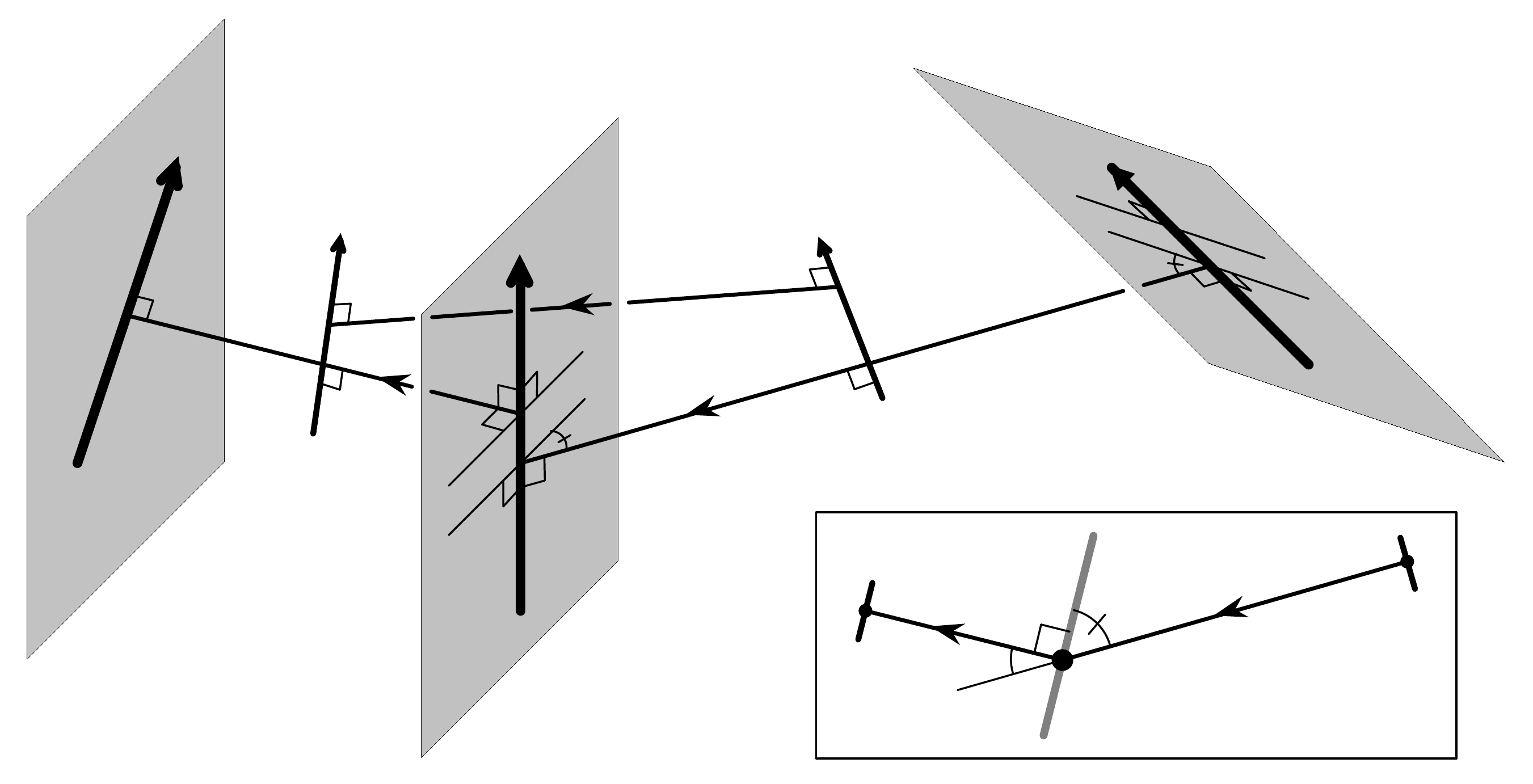}
\caption{Taking the plane $\h$.}
\label{fig:hyperplane}
\end{figure}

Since the planes $\h$ and $\rho(x)(\h)$ are both orthogonal to $\Axis_X$, they are ultra-parallel. Since $\Lv_\e[y]=\Lv_\e[x,y]\ge N_2$, the parallel angle of $\frac{1}{2}\ell_{\H}(\rho(y))$ is less than $\pi/4$. Since $\Lv_\e[x,y]\ge N_1$ and $\f$ is acute relative to $\e$, we have $0\le\theta(\rho,\f)<\pi/4$. This means that the angle $\theta'$ between $\h$ and $\Axis_{Y}$ is greater than $\pi/4$. By Lemma~\ref{lem:parallel_angle} we see that the planes $\h$ and $\rho(y)^{-1}(\h)$ are also ultra-parallel. Everything combined, we can conclude that for any choice of $g,h\in\{x,y\}$ the three planes
\begin{align}\label{eqn:triple}
\rho(g)^{-1}(\h),\; \h,\; \rho(h)(\h)
\end{align}
are pairwise ultra-parallel and $\h$ separates the other two.

Now we choose an \emph{arbitrary} point $o\in \h$ as a base point and consider the associated orbit map $\tau_{\rho,o}$. Let $w$ be a positive word in $\{x,y\}$. From Lemma~\ref{lem:axis} we know that the axis $\Axis(w)\subset(\F_2,d_\e)$ properly contains the set $\Axis^\f(w)=\{g_i \mid i\in\Z\}$, where $g_0=1$ and $g_i^{-1}g_{i+1}\in\{x,y\}$ for all $i\in\Z$. See Figure~\ref{fig:axis}.

We claim that the points in the image $\tau_{\rho,o}(\Axis^\f(w))$ in $\H^3$ belong to various translates of $\h$ which are pairwise ultra-parallel and \emph{well-ordered}. In order to see this, note that any three consecutive points in $\Axis^\f(w)$ are of the form $(g_ig^{-1},\,g_i,\,g_ih)$ for $g,h\in\{x,y\}$. Under the orbit map $\tau_{\rho,o}$, such a triple is mapped to $(\rho(g_i)\rho(g)^{-1}(o),\,\rho(g_i)(o),\,\rho(g_i)\rho(h)(o))$. Since $o\in \h$, we see that, up to isometry $\rho(g_i)$, these three points in this order belong to the three planes in (\ref{eqn:triple}), respectively. The claim is proved.

For an illustration of the claim, see Figure~\ref{fig:orbit}, where the various translates of $\h$ are shown to be well-ordered in the case of the positive (Christoffel) word $w=x^2yxy$.

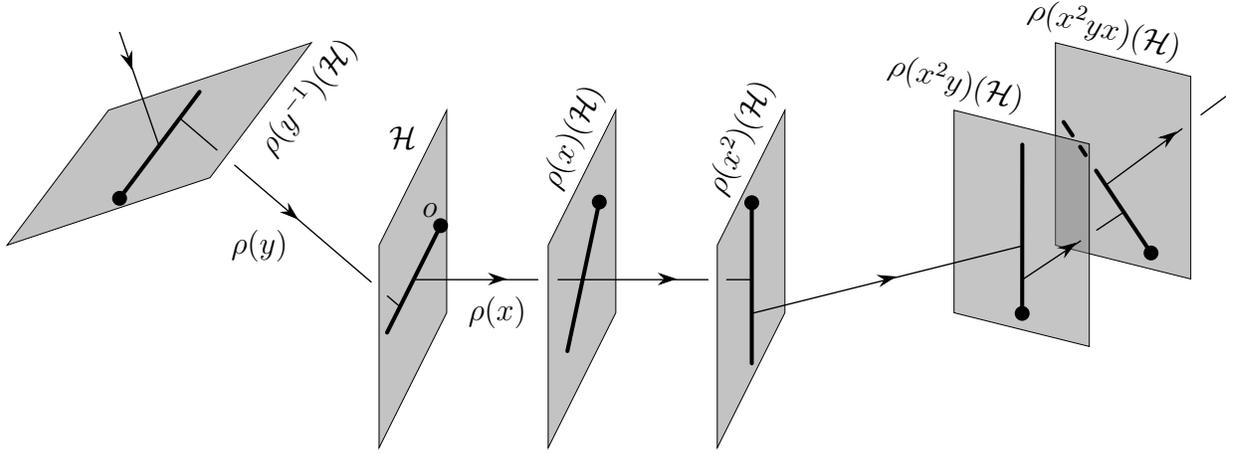
\begin{figure}[ht]
\labellist
\pinlabel {$\rho(x)$} at 380 110
\pinlabel {$\rho(y)$} at 204 160
\pinlabel {$o$} at 330 186
\pinlabel {$\h$} at 310 240
\pinlabel \rotatebox{52}{$\rho(y^{-1})(\h)$} at 240 270
\pinlabel \rotatebox{66}{$\rho(x)(\h)$} at 436 240
\pinlabel \rotatebox{66}{$\rho(x^2)(\h)$} at 560 240
\pinlabel \rotatebox{-15}{$\rho(x^2y)(\h)$} at 720 280
\pinlabel \rotatebox{-15}{$\rho(x^2yx)(\h)$} at 830 320
\endlabellist
\centering
\includegraphics[width=\textwidth]{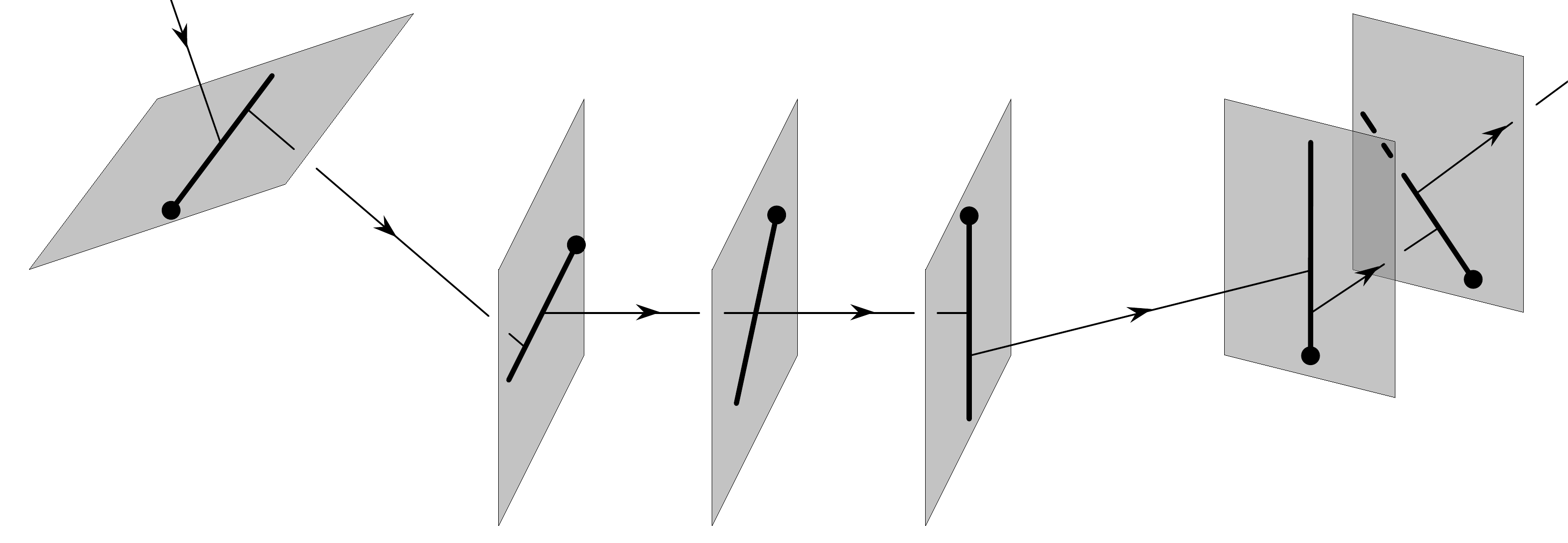}
\caption{An example for the positive word $w=x^2yxy$.}
\label{fig:orbit}
\end{figure}

We set the following positive constants:
\begin{align*}
d&=\min\{d_\H(\h,\rho(x)(\h)),d_\H(\h,\rho(y)(\h))\},\\
D&=\max\{d_\H(o,\rho(a)(o)),d_\H(o,\rho(b)(o))\},\\
L&=\max\{|x|_\e,|y|_\e\}.
\end{align*}
These constants do not depend on $w$, but they depend only on the representation $\rho$, the bases $\e=(a,b)$ and $\f=(x,y)$, and the chosen point $o\in \h$. Let $u,v\in\Axis(w)$. Then there are indices $j,k\in\Z$ such that $u$ and $v$ lie in the subintervals $[g_{j-1},g_j]$ and $[g_k,g_{k+1}]$ of $\Axis(w)$, respectively. Without loss of generality we may assume $j-1\le k$. Then $u=g_{j-1}s$ and $v=g_kt$ where $s$ and $t$ are certain subwords either of $x$ or of $y$. See Figure~\ref{fig:axis}.

\begin{figure}[ht]
\labellist
\pinlabel {$1$} at 386 54
\pinlabel {$g_{j-1}$} at 80 16
\pinlabel {$g_{j}$} at 170 40
\pinlabel {$g_{-1}$} at 310 70
\pinlabel {$g_1$} at 480 76
\pinlabel {$g_{k}$} at 610 102
\pinlabel {$g_{k+1}$} at 690 82
\pinlabel {$u$} at 110 62
\pinlabel {$v$} at 646 128
\endlabellist
\centering
\includegraphics[width=\textwidth]{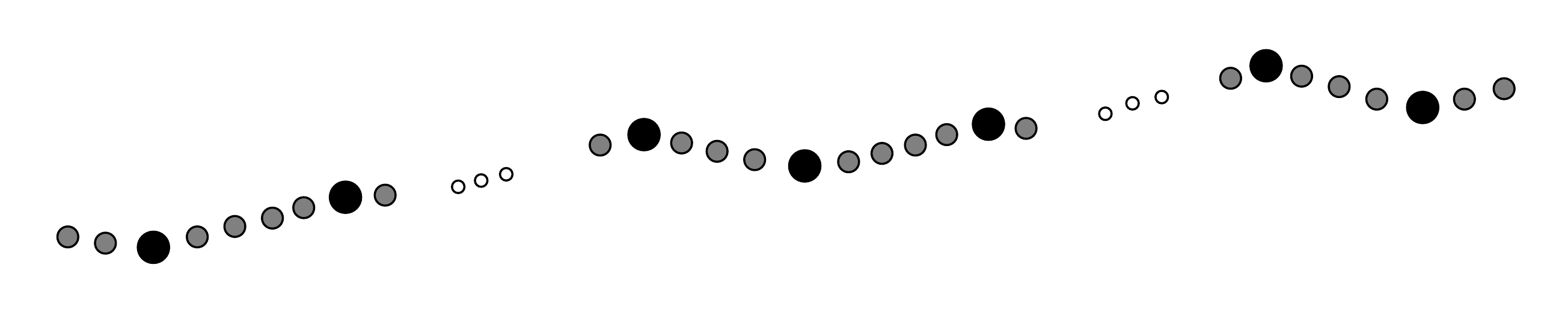}
\caption{The discrete geodesic $\Axis(w)$ in $(\F_2,d_\e)$.}
\label{fig:axis}
\end{figure}
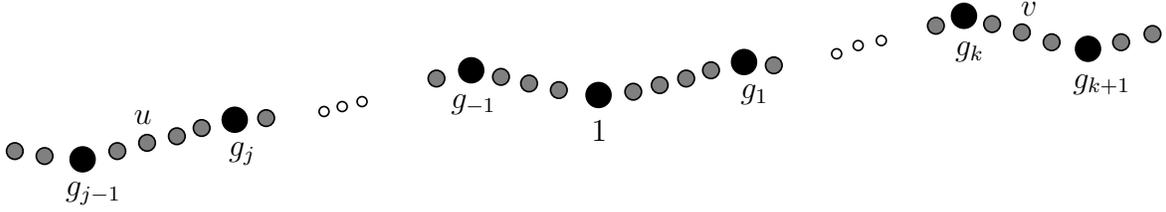

From the triangle inequality, we have
\begin{align*}
d_{\H}&(\rho(u)(o),\rho(v)(o)) + 2 D L\\
&\ge d_{\H}(\rho(g_j)(o),\rho(u)(o)) + d_{\H}(\rho(u)(o),\rho(v)(o)) + d_{\H}(\rho(v)(o),\rho(g_k)(o))\\
&\ge d_{\H}(\rho(g_j)(o),\rho(g_k)(o)),
\end{align*}
and thus
\begin{align*}
d_{\H}(\rho(u)(o),\rho(v)(o))
&\ge d_{\H}(\rho(g_j)(o),\rho(g_k)(o)) - 2 D L\\
&\ge d(k-j) - 2 D L\\
&\ge \frac{d}{L}d_\e(u,v) - (2 d + 2 D L), 
\end{align*}
since $L(k-j+2)\ge d_\e(g_{j-1},g_{k+1})\ge d_\e(u,v)$. The proof is complete if we set $o_\f=o$, $m_\f=d/L$ and $c_\f=2d+2 D L$.
\end{proof}

\subsection{Proof of Theorem~\ref{I}}\label{sec:pf1}

\begin{proof}[Proof of Theorem I]
Suppose $\rho$ satisfies the Q-conditions. In order to show $\rho$ is primitive stable, choose any base point $o\in\H^3$ and consider the associated orbit map $\tau_{\rho,o}$.

Take the number $N>0$ guaranteed by Theorem~\ref{thm:main} and consider the level-$N$ partition of $\V$ (Definition~\ref{def:partition}):
\[
\V=\V_{\le N}\cup\I_1\cup\I_2\cup\cdots\I_{2^{N+1}}.
\]
Recall that each interval $\I_i$ is of the form $\I_i=\I[x_i,y_i]$ and $\Lv_\e[x_i,y_i]=N$ for $1\le i\le2^{N+1}$.

Fix an index $i$ ($1\le i\le2^{N+1}$) and let $[w]\in\I_i=\I[x_i,y_i]$. Then
the vertex $[w]$ has a representative $\Ch_\e[w]$ which is an $\e$-Christoffel word (Definition~\ref{def:chris}). In particular, $\Ch_\e[w]$ is cyclically reduced. Furthermore, the word $\Ch_\e[w]$ is a positive word in $\{\Ch_\e[x_i],\Ch_\e[y_i]\}$ by Lemma~\ref{lem:chris}, and the pair $\f_i=(\Ch_\e[x_i],\Ch_\e[y_i])$ is an $\e$-Christoffel basis. Therefore, by Theorem~\ref{thm:main} again, there exist $o_i\in\H^3$ and $m_i, c_i>0$ such that
\[
m_i\cdot d_\e(u,v) - c_i \le d_\H(\rho(u)(o_i),\rho(v)(o_i)),
\]	
for all $u,v\in\Axis(w)$. Let $r_i=d_{\H}(o,o_i)$. Then, by the triangle inequality, we have
\begin{align*}
d_\H(\rho(u)(o),\rho(v)(o))
&\ge d_{\H}(\rho(u)(o_i),\rho(v)(o_i))
-d_{\H}(\rho(u)(o_i),\rho(u)(o))-d_{\H}(\rho(v)(o),\rho(v)(o_i))\\
&\ge m_i\cdot d_\e(u,v)-(c_i+2r_i),
\end{align*}
for all $u,v\in\Axis(w)$.

These constants $m_i, c_i$ and $r_i$ depend only on the interval $\I_i$. Let $(m_0,c_0)$ be some positive constants for which the analogous inequalities of Definition~\ref{def:altps} hold for the \emph{finite} number of vertices in $\V_{\le N}$ with respect to the base point $o$. By setting
\begin{align*}
m&=\min\{m_0,m_i\mid 1\le i\le 2^{N+1}\}\\
c&=\max\{c_0,c_i+2r_i\mid 1\le i\le 2^{N+1}\},
\end{align*}
we conclude that $\rho$ satisfies Definition~\ref{def:altps} with constants $m,c>0$ and base point $o$.
\end{proof}

\subsection{On Lupi's proof}\label{sec:lupi}

Lupi recently showed in his thesis \cite{Lup}*{Proposition 3.4 and Theorem 4.3} that, for representations $\rho:\F(a,b)\to\PSL(2,\R)$, primitive stability is equivalent to the Q-conditions. In order to sketch his idea of proof, we first remark that, for $\PSL(2,\R)$ representations, the only non-trivial case to analyze is the non-Schottky representations where the axes of hyperbolic transformations $a$ and $b$ intersect, namely, the case we discussed in the introduction. See Figure~\ref{fig:torus}.

There we have a quadrilateral $Q$ whose opposite edges are paired by $a$ and $b$. Thus as a quotient of $Q$ we obtain a hyperbolic torus with one cone-type singularity. The image of such representation is not discrete in general, unless the cone-angle is a rational multiple of $\pi$. Nevertheless, we may still think of $Q$ as a sort of fundamental domain and investigate its translates under the action. In fact, a crucial observation made by Lupi is that, for a primitive element $w$ with $\Axis(w)=\{u_i\}_{i\in\Z}$ in $\F(a,b)$, the translates $\rho(u_i)(Q)$ of $Q$ do not overlap and form an infinite staircase-like strip. See \cite{Lup}*{Figure 3.3}. From this one can easily see that the image of $\Axis(w)$ under an orbit map is a uniform quasi-geodesic.

For $\PSL(2,\C)$-representations satisfying the Q-conditions, however, there seems to be no ``pseudo-fundamental domain" in $\H^3$ as simple as the quadrilateral $Q$ in $\H^2$. For example, consider a slight perturbation of the above example out of $\PSL(2,\R)$ and its Coxeter extension $W(p,q,r)$. We obtain four complete geodesics which are axes of $pqr$ and of its three other conjugates as in Figure~\ref{fig:torus}. The four axes are in skew position so that, in general, no (totally geodesic) plane in $\H^3$ contains any two of them. So there is no direct analogue of the quadrilateral $Q$.

We may try to form a domain using other types of surfaces than planes. For example, the three axes of $pqr$, $p$ and $qrp$ (see Figure~\ref{fig:torus} again) can be used to define a ruled surface, that is, the union of all complete geodesics which intersect the three axes simultaneously. For a slight perturbation of the above example, we may then form a ``twisted" quadrilateral prism bounded by such ruled surfaces. However, it seems rather hard to analyze such an object. For example, it is not immediate to know how long its combinatorial structure is preserved under perturbation and how far the distance is between two opposite surfaces of the quadrilateral.

\subsection{On primitive displacing actions}\label{sec:pspd}

Let us explain why we may view Theorem~\ref{I} in analogy with a theorem of Delzant, Guichard, Labourie and Mozes \cite{DGLM}.

Let $(X, d_X)$ be a metric space and $\rho:\Ga\to\Isom(X)$ an isometric action of a group $\Ga$. For an element $w\in\Ga$ its translation length $\ell_X(w)$ for the $\rho$-action is defined as
\[
\ell_X(w)=\inf_{x\in X}d_X(x,\rho(w)(x)).
\]
Endow $\Ga$ with a word metric $d_\e$ for a generating set $\e$, and consider the isometric action of $\Ga$ on itself by left multiplication. The translation length of $w\in\Ga$ for this action will be denoted by $\ell_\e(w)$. 

The action $\rho:\Ga\to\Isom(X)$ is called \emph{displacing} if
there are positive constants $m$ and $c$ such that
\[
\ell_X(w)\ge m\cdot\ell_\e(w)-c
\]
for all $w\in\Ga$. Delzant et al showed that if $\Ga$ is word hyperbolic then the action $\rho$ is displacing if and only if an orbit map of $\rho$ is a quasi-isometric embedding. See \cite{DGLM}*{Lemma 2.0.1, Proposition 2.2.1, Corollary 4.0.6}. 

Since the free group $\F_2$ is word hyperbolic, we may view our Theorem~\ref{I} in analogy with the above theorem. Namely, if we focus only on actions of primitive elements, the Q-conditions, by Theorem~\ref{thm:bqpd} and (\ref{eqn:ltr}), can be thought of as a weakening of the displacing property, while primitive stability is obviously a weakening of quasi-isometric embedding.  

Primitive stability is originally defined by Minsky \cite{Min} for free groups of arbitrary rank. The Q-conditions also make sense for such free groups, but Theorem~\ref{thm:bqpd} is only known to hold for $\F_2$. One may then ask if it holds for any $\F_n$. Or, regardless, we rather consider the following definition: a representation $\rho:\F_n\to\Isom(X)$ ($n\ge2$) is said to be \emph{primitive displacing} if there are positive constants $m$ and $c$ such that
\[
\ell_X(w)\ge m\cdot\ell_\e(w)-c
\]
for all primitive elements $w\in\F_n(\e)$. Then it is natural to ask:
\begin{question}
For which metric space $X$, is it true that $\rho:\F_n\to\Isom(X)$ ($n\ge2$) is primitive stable if and only if it is primitive displacing?
\end{question}
In fact, one direction, primitive stability implies primitive displacing, is true with the same proof as in Proposition~\ref{prop:psbq}. So the question is really about the other. Our Theorem~\ref{I} tells us that the equivalence is true for $n=2$ and $X=\H^3$. The next case one may want to explore seems to be when $n=2$ and $X$ is a rank-$1$ symmetric space of non-compact type, e.g., $\H^4$ and $\C\H^2$. In the absence of trigonometric rules, one may employ certain large-scale geometric arguments in order to draw a conclusion that is analogous to Lemma~\ref{lem:angle}, assuming the primitive displacing property and using Theorem~\ref{thm:Nielsen}(b).

\section{Bounded intersection property}\label{sec:bi}

In this section we investigate representations $\rho:\F_2\to\PSL(2,\C)$ satisfying a new condition which we call the bounded intersection property. It is described using primitive elements in $\F_2$ that are palindromic in a fixed basis. Our definition is motivated by the work of Gilman and Keen \cite{GK09}*{Theorem 6.6}, where they use palindromic elements of $\F_2$ to give a sufficient condition for a representation to be discrete.

In Section~\ref{sec:pf2} we prove Theorem~\ref{II} that the bounded intersection property is implied by the Q-conditions.

\subsection{Tri-coloring}

For the forthcoming discussion we need the \emph{tri-coloring} of the Farey triangulation.

We fix a basis
\[
\e=(a,b)
\]
with abelianization $\pi_\e:\F_2\to\Z^2$ as in \eqref{eqn:abel}. Consider the homomorphism given by reduction modulo $2$ 
\begin{align*}
\PGL(2,\Z)\twoheadrightarrow\PGL(2,\Z/2)\cong\Sym_3.
\end{align*}
Its kernel, the \emph{level two congruence subgroup} $\PGL(2,\Z/2)_{(2)}$, is isomorphic to the ideal triangle reflection group $T^*(\infty,\infty,\infty)\cong\Z/2\ast\Z/2\ast\Z/2$ with fundamental domain any Farey triangle.

Reduction modulo $2$
\begin{align*}
\V=\p^1(\Q)\twoheadrightarrow\p^1(\Z/2)=\{0/1,\;1/1,\;1/0\}
\end{align*}
partitions the vertex set $\V$ into three $\PGL(2,\Z/2)_{(2)}$-orbits ``colored" by $0/1$, $1/1$ and $1/0$, respectively. Note that the three vertices of a Farey triangle are all colored differently. A Farey edge is then colored by the same color of its opposite vertex. See Figure~\ref{fig:pal1}. For example, the vertex $[a]$ is colored by $1/0$ (black), the vertex $[b]$ by $0/1$ (white), and the edge $[a,b]$ by $1/1$ (gray). If a vertex $[w]\in\V$ is colored by $c\in\{0/1,\;1/1,\;1/0\}$ then we shall often say $[w]$ is a $c$-vertex. Similarly for edges.

For a more detailed discussion of tri-coloring, see \cite{GMST}*{Section 2}.

\subsection{Primitive palindromes}

We shall be interested in primitive elements that are palindromic.

\begin{definition}
A word $w=w(a,b)\in\F_2$ is said to be \textit{palindromic} in a basis $\e=(a,b)$ (or, simply, an \emph{$\e$-palindrome}) if it reads the same forward and backward, that is,
\begin{align}\label{eqn:pal}
w(a,b)^{-1}=w(a^{-1},b^{-1}).
\end{align}
\end{definition}

Gilman and Keen \cite{GK11} introduced an algorithm of finding a palindromic representative in each primitive conjugacy class whose cyclically reduced length \eqref{eqn:le} is odd. This algorithm can be described by slightly modifying the inductive procedure for the Christoffel function (Definition~\ref{def:chris}(b)) as follows. Co-direct all Farey edges with the directed edge $[\e]=[a,b]$ as in Section~\ref{sec:chris}.

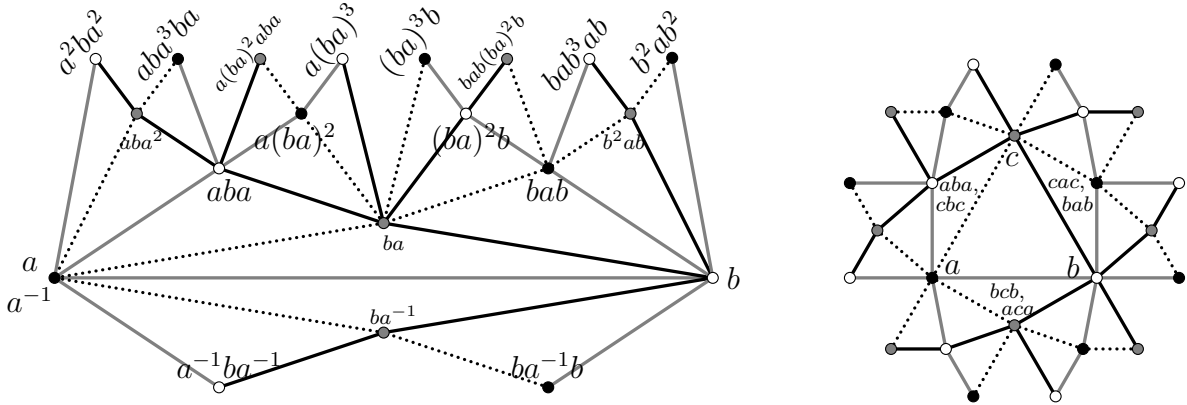
\begin{figure}[ht]
\labellist
\pinlabel {$a$} at 8 156
\pinlabel {$a^{-1}$} at 8 124
\pinlabel {$b$} at 650 144
\pinlabel {\scriptsize $ba$} at 340 176
\pinlabel {\scriptsize $ba^{-1}$} at 340 110
\pinlabel {$aba$} at 190 224
\pinlabel {$bab$} at 480 224
\pinlabel {$a^{-1}ba^{-1}$} at 190 64
\pinlabel {$ba^{-1}b$} at 480 64
\pinlabel {\scriptsize $aba^2$} at 110 270
\pinlabel {$a(ba)^2$} at 250 270
\pinlabel {$(ba)^2b$} at 410 270
\pinlabel {\scriptsize $b^2ab$} at 550 270

\pinlabel \rotatebox{50}{$a^2ba^2$} at 54 360
\pinlabel \rotatebox{50}{$aba^3ba$} at 130 360
\pinlabel \rotatebox{50}{\scriptsize $a(ba)^2aba$} at 206 350
\pinlabel \rotatebox{50}{$a(ba)^3$} at 284 360
\pinlabel \rotatebox{50}{$(ba)^3b$} at 356 360
\pinlabel \rotatebox{50}{\scriptsize $bab(ba)^2b$} at 430 350
\pinlabel \rotatebox{50}{$bab^3ab$} at 504 350
\pinlabel \rotatebox{50}{$b^2ab^2$} at 580 360

\pinlabel {$a$} at 850 154
\pinlabel {$b$} at 960 154
\pinlabel {$c$} at 904 252
\pinlabel {\scriptsize $aba$,} at 856 228
\pinlabel {\scriptsize $cbc$} at 848 212
\pinlabel {\scriptsize $cac$,} at 954 228
\pinlabel {\scriptsize $bab$} at 964 212
\pinlabel {\scriptsize $bcb$,} at 900 130
\pinlabel {\scriptsize $aca$} at 910 114

\endlabellist
\centering
\includegraphics[width=1\textwidth]{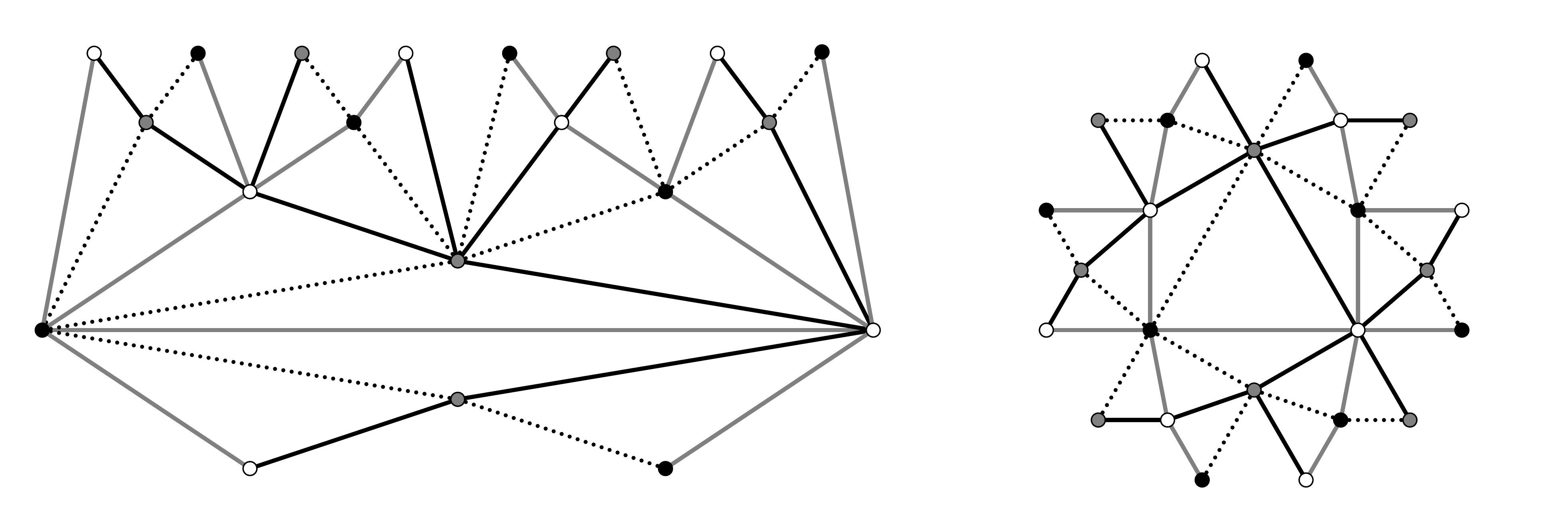}
\caption{(Left) The $\e$-palindrome function. (Right) The $\ti{\e}$-palindromes. }
\label{fig:pal1}
\end{figure}

\begin{definition}\label{def:pal}
Let $S=\F(a,b)$ and define $B:S\times S \to S$ by
\[
B(w_1,w_2)=
\left\{
\begin{array}{ll}
w_2w_1& \text{if $[w_1,w_2]$ and $[a,b]$ have the same color,}\\
w_1w_2& \text{otherwise.}
\end{array}
\right.
\]
The \emph{$\e$-palindrome function} $\Pal_\e$ is defined inductively on $\V^+[a,b]$ with initial values $\Pal_\e[a]=a$ and $\Pal_\e[b]=b$, and on $\V^-[a,b]$ with initial values $\Pal_\e[a]=a^{-1}$ and $\Pal_\e[b]=b$.
\end{definition}

See Figure~\ref{fig:pal1}(Left). Observe that $\Pal_\e[w]$ is an $\e$-palindrome if and only if $[w]\in\V$ has color $1/0$ or $0/1$ (black or white). On the other hand, $\Pal_\e[w]$ is a product of two $\e$-palindromes if and only if $[w]$ has color $1/1$ (gray). These are re-statements of \cite{GK11}*{Theorem 2.1}, since $[w]$ has color $1/1$ (gray) if and only if its cyclically reduced length $\|w\|_\e=\Fi_\e[w]$ is even, that is, $1+1$ (mod $2$). We also note that, according to \cite{Pig}*{Theorem 1.(1)}, $\Pal_\e[w]$ and its inverse are the only $\e$-palindromic representatives belonging to the unoriented conjugacy class $[w]\in\V$.

An \emph{$\e$-palindromic basis} $(x,y)$ is a basis such that both $x$ and $y$ are $\e$-palindromes. As with the $\e$-Christoffel bases, it is apparent that if $[x,y]$ is a Farey edge then the pair $(\Pal_\e[x],\Pal_\e[y])$ is a basis. If, in addition, $[x,y]$ and $[a,b]$ have the same color, then $(\Pal_\e[x],\Pal_\e[y])$ is an $\e$-palindromic basis. Compare \cite{Pig}*{Theorem 2} and \cite{KR07}*{Theorem 5.5}.

\begin{lemma}\label{lem:pal1}
Let $\f=(x,y)=(\Pal_\e[x],\Pal_\e[y])$ be an $\e$-palindromic basis (so that $[x,y]$ and $[a,b]$ have the same color). Then
\begin{enumerate}[label=\textup{(\alph*)},nosep,leftmargin=*]
\item the basis $\e=(a,b)$ is an $\f$-palindromic basis; 
\item the set of $\e$-palindromes is the same as the set of $\f$-palindromes.
\end{enumerate}
\end{lemma}

\begin{proof}
(a) Without loss of generality, we may assume that $[x]$ and $[y]$ are both in $\V^+[a,b]$ and the edge $[x,y]$ is co-directed with $[a,b]$. Let $\Gal$ denote the shortest gallery  containing both $[a,b]$ and $[x,y]$. Let $\{[g_k]\in\V \mid 1\le k\le n\}$ be the set of all $1/1$-vertices (gray vertices) in $\Gal$. Then we can decompose $\Gal$ into a chain of \emph{simple} galleries $\Gal_k$ $(1\le k\le n)$ 
\[
\Gal=\Gal_1\cup\Gal_2\cup\cdots\cup\Gal_n,
\]
where each $\Gal_k$ is the star of $[g_k]$ in $\Gal$, namely, the union of all Farey triangles in $\Gal$ containing $[g_k]$. The intersection $\Gal_k\cap\Gal_{k+1}$ of two adjacent simple galleries is a gray edge $[a_k,b_k]$, so that each $\Gal_k$ is the shortest gallery connecting two gray edges $[a_{k-1},b_{k-1}]$ and $[a_k,b_k]$, where $(a_0,b_0)=(a,b)$ and $(a_n,b_n)=(x,y)$. See Figure~\ref{fig:pal2}(Bottom).

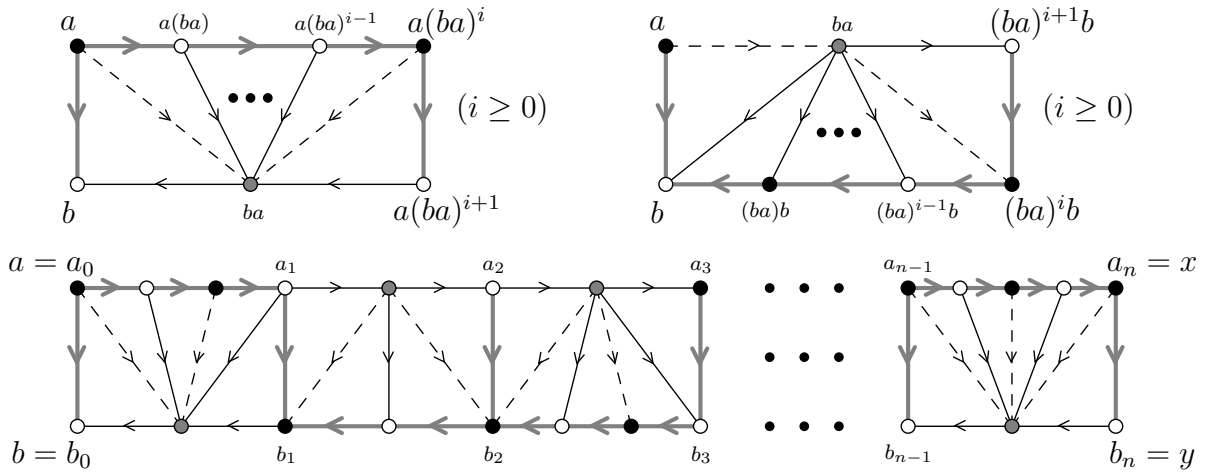
\begin{figure}[ht]
\labellist
\pinlabel {$a$} at 66 342
\pinlabel {\scriptsize$a(ba)$} at 150 342
\pinlabel {\scriptsize$a(ba)^{i-1}$} at 260 342
\pinlabel {$a(ba)^{i}$} at 340 342

\pinlabel {$b$} at 66 206
\pinlabel {\scriptsize$ba$} at 200 206
\pinlabel {$a(ba)^{i+1}$} at 340 206
\pinlabel {($i\ge0$)} at 380 280

\pinlabel {$a$} at 492 342
\pinlabel {\scriptsize$ba$} at 626 342
\pinlabel {$(ba)^{i+1}b$} at 770 342
\pinlabel {($i\ge0$)} at 804 280

\pinlabel {$b$} at 492 206
\pinlabel {\scriptsize$(ba)b$} at 570 206
\pinlabel {\scriptsize$(ba)^{i-1}b$} at 680 206
\pinlabel {$(ba)^{i}b$} at 770 206

\pinlabel {$a=a_0$} at 54 166
\pinlabel {$b=b_0$} at 54 30
\pinlabel {\scriptsize$a_1$} at 224 166
\pinlabel {\scriptsize$b_1$} at 224 30
\pinlabel {\scriptsize$a_2$} at 374 166
\pinlabel {\scriptsize$b_2$} at 374 30
\pinlabel {\scriptsize$a_3$} at 524 166
\pinlabel {\scriptsize$b_3$} at 524 30
\pinlabel {\scriptsize$a_{n-1}$} at 674 166
\pinlabel {\scriptsize$b_{n-1}$} at 674 30
\pinlabel {$a_{n}=x$} at 850 166
\pinlabel {$b_{n}=y$} at 850 30

\endlabellist
\centering
\includegraphics[width=1\textwidth]{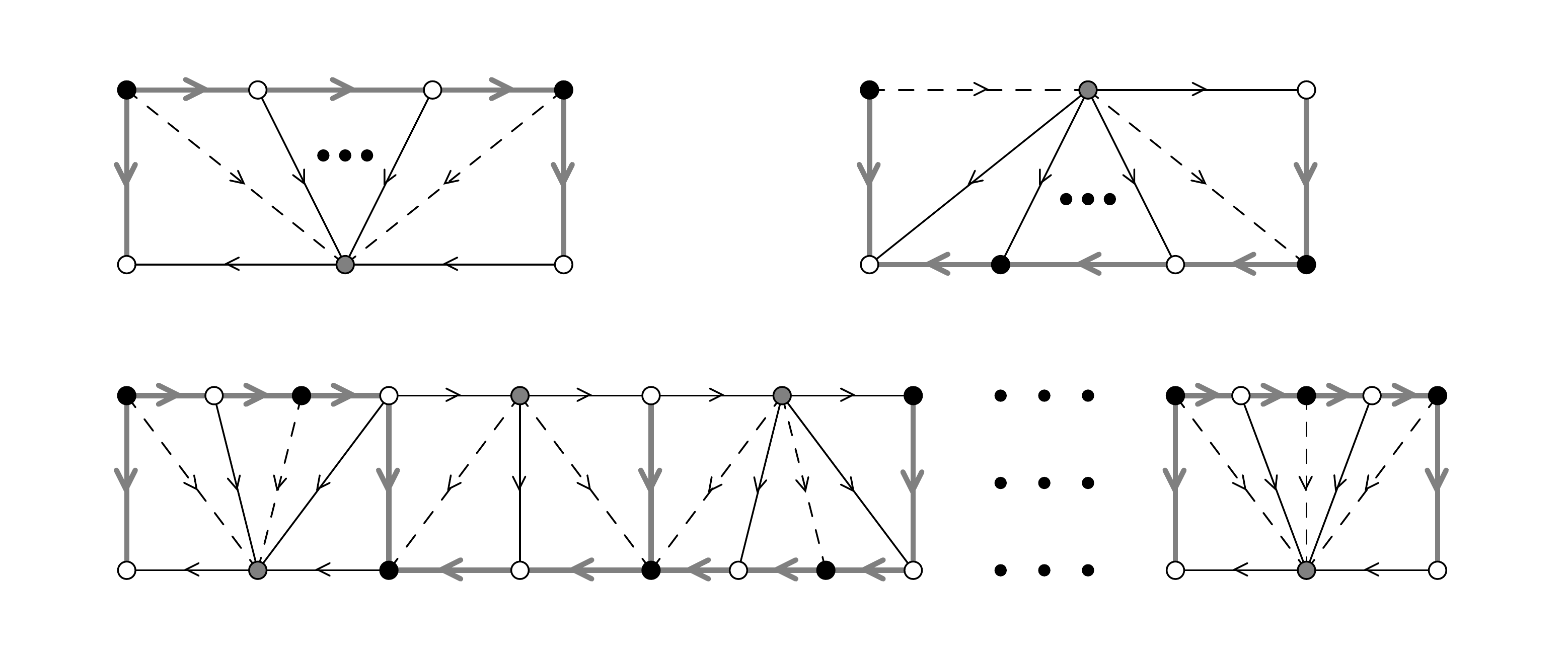}
\caption{(Top) Two simple galleries. (Bottom) A generic gallery connecting $[a,b]$ and $[x,y]$.}
\label{fig:pal2}
\end{figure}

Note that the property of being palindromic is transitive in the following sense: if a word $w=w(x',y')$ is an $(x',y')$-palindrome and if $(x',y')$ is an $(x,y)$-palindromic basis, then $w=w(x,y)$ is an $(x,y)$-palindrome. Thus, by induction, it is enough to prove the lemma when $n=1$ and the gallery $\Gal$ is simple.  

There are two possibilities of a simple gallery:
\[
(x,y)=(a(ba)^i,a(ba)^{i+1})\quad\textup{or}\quad
(x,y)=((ba)^{i+1}b,(ba)^ib)
\]
for some integer $i\ge0$. See Figure~\ref{fig:pal2}(Top). As the proofs are similar we shall consider only the first case. Then we have $x^{-1}y=ba$ and
\begin{align*}
a&=x(ba)^{-i}=x(y^{-1}x)^i,\\
b&=x^{-1}ya^{-1}=x^{-1}y(x^{-1}y)^ix^{-1}=(x^{-1}y)^{i+1}x^{-1}.
\end{align*}
Therefore, both $a$ and $b$ are $\f$-palindromes. The proof is complete.

(b) This follows from (a) and the transitive property of palindromes we noted above.
\end{proof}

Following \cite{GMST}*{Section 2.2} we define a \emph{basic triple} as an ordered triple $(x,y,z)\in\Prim^3$ such that $(x,y)$ is a basis and $xyz=1$. A basis $\f=(x,y)$ extends uniquely to a basic triple $(x,y,z)$, which we denote by $\ti{\f}$. It is clear that this correspondence is one-to-one between bases and basic triples.

\begin{definition}\label{def:fpal}
Let $\ti{\f}=(x,y,z)$ be a basic triple. An element $w\in\F_2$ is said to be an \emph{$\ti{\f}$-palindrome} if it is palindromic in either $(x,y)$ or $(y,z)$ or $(z,x)$ (when written in these bases). 
\end{definition}

Consider the basic triple $\ti{\e}=(a,b,c)$ associated to the distinguished basis $\e=(a,b)$. In addition to the $(a,b)$-palindrome function (Definition~\ref{def:pal}) we may consider the $(b,c)$-palindrome function and the $(c,a)$-palindrome function as well. See Figure~\ref{fig:pal1}(Right), where the vertices $[a]$, $[b]$ and $[c]$ are colored black, white and gray, respectively. We noted previously that $\Pal_{(a,b)}[w]$ is an $(a,b)$-palindrome unless $[w]$ is colored gray. Analogous statements are true of the functions $\Pal_{(b,c)}$ and $\Pal_{(c,a)}$. Thus, for example, a black vertex contains an $(a,b)$-palindrome and a $(c,a)$-palindrome. Therefore, we can conclude as follows. 

\begin{lemma}\label{lem:pal2}
If a primitive element is not conjugate to $a^{\pm1}$, $b^{\pm1}$ or $c^{\pm1}$, then its conjugacy class contains at least one and at most two $\ti{\e}$-palindromes.
\end{lemma}

\noindent One can actually show that the two $\ti{\e}$-palindromic representatives are distinct elements of $\F_2$, but we shall not need this fact.

\subsection{Geometry of palindromes}\label{sec:pal}

Let $\e=(a,b)$ be the distinguished basis of $\F_2$ and $\ti{\e}=(a,b,c)$ the associated basic triple. As in Section~\ref{sec:coxeter} (but recall Remark~\ref{rem:2fold}) we denote by $\W(p,q,r)$ the free Coxeter group of rank three, and by 
\[
\psi_\e:\F(a,b)\to\W(p,q,r)
\]
the embedding defined by $\psi_\e(a)=qr$ and $\psi_\e(b)=rp$. We then have $\psi_\e(c)=pq$. Observe that the conjugation by the involutory element $r$ inverts $\psi_\e(a)$ and $\psi_\e(b)$:
\begin{align*}
r\cdot\psi_\e(a)\cdot r&=rq=\psi_\e(a^{-1})\\ r\cdot\psi_\e(b)\cdot r&=pr=\psi_\e(b^{-1}).
\end{align*}
Thus for any word $w=w(a,b)\in\F(a,b)$ we have
\begin{align}\label{eqn:r}
r\cdot\psi_\e(w(a,b))\cdot r=\psi_\e(w(a^{-1},b^{-1})).
\end{align}
See also the end of Appendix~\ref{sec:isom}.

Let $\rho:\F_2\to\PSL(2,\C)$ be an irreducible representation and $\rho_\e$ be the Coxeter extension of $\rho$ such that $\rho=\rho_\e\circ\psi_\e$ (see Theorem~\ref{thm:coxeter}). We simplify the notations as in \eqref{eqn:pqr} so that, for example, $\Axis_R$ denotes the axis of $\rho_\e(r)$.

The $\rho$-image of a palindrome has the following geometric property. See also \cite{GK09}*{Lemma 5.1}.

\begin{lemma}\label{lem:pal}
Suppose $\rho:\F_2\to\PSL(2,\C)$ is irreducible and $w=w(a,b)$ is palindromic in $\e=(a,b)$. If $\rho(w)$ is elliptic or loxodromic, then its axis intersects $\Axis_R$ orthogonally. If $\rho(w)$ is parabolic, then its fixed point is an ideal end point of $\Axis_R$.	
\end{lemma}

\begin{proof}
Applying $\rho_\e$ to \eqref{eqn:r} we have $R\cdot\rho(w(a,b))\cdot R=\rho(w(a^{-1},b^{-1}))$. Since $w=w(a,b)$ is palindromic, we have from \eqref{eqn:pal} that
\[
R\cdot\rho(w)\cdot R=\rho(w)^{-1}.
\]
Thus the involution $R$ preserves the set of fixed points of $\rho(w)$ in the ideal boundary $\partial\H^3$. If $\rho(w)$ is elliptic or loxodromic, then $R$ exchanges the end points of $\Axis_{\rho(w)}$, hence $\Axis_R$ is orthogonal to $\Axis_{\rho(w)}$. If $\rho(w)$ is parabolic, then its fixed point is one of the end points of $\Axis_R$.
\end{proof}

\begin{remark}
The converse of the lemma is not true in general. That is, for a non-palindromic element $w'$, $\rho(w')$ may satisfy the same conclusion. Such a counterexample arises, for example, when there is an element $u\in\F_2$ such that $\rho(u)$ is elliptic or loxodromic and $\Axis_{\rho(u)}$ coincides with $\Axis_R$. For if $w'=uwu^{-1}$ with $w$ palindromic then both $\rho(w)$ and $\rho(w')$ have axes orthogonal to $\Axis_R$. But $w'$ cannot be palindromic if there is no cancellation in $uwu^{-1}$ as a word in $(a,b)$.

In order to find an explicit counterexample, let $\rho$ be a $\PSL(2,\R)$-representation such that $\rho(a)$ and $\rho(b)$ are hyperbolic translations with the same translation length $\ell$ where $\cosh\ell=(1+\sqrt{3})/2$ and their axes intersect orthogonally. Then one can verify that the element $\rho(a^2ba^{-1}b^{-2})$ is elliptic with axis $\Axis_R$. Let $w=(a^2ba^{-1}b^{-2})a(b^2ab^{-1}a^{-2})$. Then $w$ is not palindromic but $\rho(w)$ is hyperbolic with axis intersecting $\Axis_R$ orthogonally.
\end{remark}

\subsection{Bounded intersection property}

Let $\e=(a,b)$ and $\ti{\e}=(a,b,c)$ as before, and we continue to use the same notation as in the preceding section. 

Of particular interest are irreducible representations satisfying the first Q-condition. Given such a representation $\rho:\F_2\to\PSL(2,\C)$, the image $\rho(w)$ of any primitive element $w$ admits an axis $\Axis_{\rho(w)}$ in $\H^3$. If a primitive element $w$ is palindromic in $(a,b)$ (resp. $(b,c)$ and $(c,a)$) then by Lemma~\ref{lem:pal} the axis $\Axis_{\rho(w)}$ is orthogonal to $\Axis_R$ (resp. $\Axis_P$ and $\Axis_Q$).

Consider the following subsets of $\H^3$:
\begin{align*}
J_P&=\{\Axis_P\cap\Axis_{\rho(w)} \mid w\textup{ is a primitive $(b,c)$-palindrome}\},\\
J_Q&=\{\Axis_Q\cap\Axis_{\rho(w)} \mid w\textup{ is a primitive $(c,a)$-palindrome}\},\\
J_R&=\{\Axis_R\cap\Axis_{\rho(w)} \mid w\textup{ is a primitive $(a,b)$-palindrome}\}.
\end{align*}

\begin{definition}\label{def:bi}
A representation $\rho:\F_2\to\PSL(2,\C)$ has the \textit{bounded intersection property} with respect to $\e=(a,b)$ if it satisfies the following conditions:
\begin{enumerate}[label=\textup{(\roman*)},nosep,leftmargin=*]
\item it is irreducible;
\item all primitive elements are sent to loxodromic elements;
\item the sets $J_P$, $J_Q$ and $J_R$ are all bounded.
\end{enumerate}
\end{definition}

\noindent Note that the property (ii) is the first Q-condition (Definition~\ref{def:bq}(i)). The bounded intersection property of $\rho$ is clearly invariant under conjugation of $\PSL(2,\C)$, so the definition also makes sense in the character variety $\X(\F_2)$. We denote by $\BI\subset\X(\F_2)$ the subset consisting of all $\PSL(2,\C)$-conjugacy classes of representations having the bounded intersection property.

The above definition was made after a choice of basis $\e=(a,b)$. As we shall see below, however, it is in fact independent of the choice. Thus the set $\BI$ is invariant under the action of $\Out(\F_2)$.

\begin{proposition}\label{prop:bi}
The bounded intersection property is independent of the choice of basis.
\end{proposition}

\begin{proof}
Suppose a representation $\rho:\F_2\to\PSL(2,\C)$ has the bounded intersection property with respect to $\e=(a,b)$ and the associated basic triple $(a,b,c)$. Let $(x,y)$ be another basis and $(x,y,z)$ the associated basic triple. Since (i) and (ii) of Definition~\ref{def:bi} are satisfied regardless of bases, it remains to check the third property (iii) for $(x,y)$. So we consider the Coxeter extension $(S,T,U)$ associated to $(x,y,z)$, so that $\rho(x)=TU$, $\rho(y)=US$ and $\rho(z)=ST$. Then our goal is to show that the associated subsets of $\H^3$
\begin{align*}
J_S&=\{\Axis_S\cap\Axis_{\rho(w)} \mid w\textup{ is a primitive $(y,z)$-palindrome}\},\\
J_T&=\{\Axis_T\cap\Axis_{\rho(w)} \mid w\textup{ is a primitive $(z,x)$-palindrome}\},\\
J_U&=\{\Axis_U\cap\Axis_{\rho(w)} \mid w\textup{ is a primitive $(x,y)$-palindrome}\}
\end{align*}
are all bounded.

We start with the set $J_S$ associated with the basis $(y,z)$. The color of the Farey edge $[y,z]$ matches one of the colors of $[a,b]$, $[b,c]$ and $[c,a]$. Without loss of generality we may assume that it matches the color of $[\e]=[a,b]$. Consider the $\e$-palindromic basis
\[
\f:=(\Pal_\e[y],\Pal_\e[z]).
\]
Since $[\f]=[y,z]$ as directed edges, $\f$ must be conjugate to one of the bases
\[
\g\in\{(y,z),\;(y^{-1},z),\;(y,z^{-1}),\;(y^{-1},z^{-1})\}.
\]
So there is an element $g\in\F_2$ such that
\[
g\f g^{-1}=\g.
\]
Then the following equalities hold:
\begingroup
\allowdisplaybreaks
\begin{align*}
J_S
&=\{\Axis_S\cap\Axis_{\rho(w)} \mid w\textup{ is a primitive $(y,z)$-palindrome}\}\\
&=\{\Axis_S\cap\Axis_{\rho(w)} \mid w\textup{ is a primitive $\g$-palindrome}\}\\
&=\{\Axis_S\cap\Axis_{\rho(w)} \mid w\textup{ is a primitive $(g\f g^{-1})$-palindrome}\}\\
&=\{\Axis_S\cap\Axis_{\rho(gwg^{-1})} \mid w\textup{ is a primitive $\f$-palindrome}\}\\
&=\rho(g)\cdot\{\Axis_{R}\cap\Axis_{\rho(w)} \mid w\textup{ is a primitive $\f$-palindrome}\}\\
&=\rho(g)\cdot\{\Axis_{R}\cap\Axis_{\rho(w)} \mid w\textup{ is a primitive $\e$-palindrome}\}\\
&=\rho(g)\cdot J_R.
\end{align*}
\endgroup
For the fourth equality, note that $(g\f g^{-1})$-palindromes are of the form $gwg^{-1}$ for some $\f$-palindrome $w$. The sixth equality holds since $\f$ is an $\e$-palindromic basis: by Lemma~\ref{lem:pal1}(b) $w$ is an $\e$-palindrome if and only if it is an $\f$-palindrome, and by Lemma~\ref{lem:pal} the $\Axis_R$, which is the common perpendicular of $\Axis_{\rho(a)}$ and $\Axis_{\rho(b)}$, is also the common perpendicular of $\Axis_{\rho(\Pal_\e[y])}$ and $\Axis_{\rho(\Pal_\e[z])}$. For the fifth equality $\rho(g)\cdot\Axis_R=\Axis_S$, note from $g\f g^{-1}=\g$ that $\Axis_S$ is the common perpendicular of $\Axis_{\rho(y)}=\Axis_{\rho(g\Pal_\e[y]g^{-1})}$ and $\Axis_{\rho(z)}=\Axis_{\rho(g\Pal_\e[z]g^{-1})}$. All the remaining equalities are either immediate or by definition. 

Since $\rho(g)$ is an isometry of $\H^3$ and $J_R$ is a bounded subset, the set $J_S$ is also bounded. In a similar fashion we can show that $J_T$ and $J_U$ are bounded as well. The proof is complete.
\end{proof}

\begin{remark}
Note that the $\Z/3$-symmetry in Definition~\ref{def:bi}(iii) is essential in the above proof. Under the conditions (i) and (ii) of Definition~\ref{def:bi}, we do not know if the boundedness of one set, say, $J_R$ implies that of the other sets $J_P$ and $J_Q$. Compare with the definition in \cite{GK09}, where only one set $J_R$ is taken into consideration.
\end{remark}

By Proposition~\ref{prop:bi} and Theorem~\ref{II} to be proved below, we know that $\BI\subset\X(\F_2)$ is invariant under the action of $\Out(\F_2)$ and contains the open subset $\BQ$. 

\begin{question}
Is the subset $\BI$ open? Does $\Out(\F_2)$ act on $\BI$ properly discontinuously? Do we have the equality $\BQ=\BI$?
\end{question}

\subsection {Proof of Theorem~\ref{II}}\label{sec:pf2}

\begin{proof}
In order to prove the first claim of the theorem, suppose $\rho:\F_2\to\PSL(2,\C)$ satisfies the Q-conditions. Then it satisfies Definition~\ref{def:bi}(i) by \cite{TWZ081}*{Theorem 1.4}. It also satisfies Definition~\ref{def:bi}(ii) which is the first Q-condition.

To prove Definition~\ref{def:bi}(iii) fix a basis $\f=(x,y)$ and the associated basic triple $(x,y,z)$. As in the proof of Proposition~\ref{prop:bi} we consider the Coxeter extension $(S,T,U)$ associated to $(x,y,z)$, so that $\rho(x)=TU$, $\rho(y)=US$ and $\rho(z)=ST$. Then we have to show that the sets $J_S$, $J_T$ and $J_U$ are all bounded. We first show that the set $J_U$ is bounded. So consider the basis $(x,y)$ and the corresponding word metric $d_{(x,y)}$ on $\F_2$.

Since $\rho$ satisfies the Q-conditions, it is primitive stable by Theorem~\ref{I}. Let $o\in\H^3$ be a point and $(M,c)$ a pair of positive constants given by the definition of primitive stability (Definition~\ref{def:ps}). Denote by $\tau_{\rho,o}$ the corresponding orbit map sending $\F_2$ into $\H^3$. Then the $\tau_{\rho,o}$-image of all primitive axes in $(\F_2,d_{(x,y)})$ are sent to $(M,c)$-quasi-geodesics in $\H^3$. By the Morse lemma (see, for example, \cite{BH}*{Theorem III.H.1.7}), we know that there is a constant $D>0$ depending on $(M,c)$ such that for any primitive element $w$ we have
\begin{align*}
d_{\textup{Haus}}(\tau_{\rho,o}(\Axis(w)),\Axis_{\rho(w)})\le D,
\end{align*}
where $d_{\textup{Haus}}$ is the Hausdorff distance for subsets of $\H^3$.
	
Assume that $w$ is palindromic in $(x,y)$ as well. Then $w$ is cyclically reduced and its axis $\Axis(w)$ in $(\F_2,d_{(x,y)})$ passes through the identity. Therefore, the image $\tau_{\rho,o}(\Axis(w))$ contains the point $o$. This implies that
\begin{align*}
d_{\H}(o,\Axis_{\rho(w)})\le D.
\end{align*}

Let $w$ and $w'$ be primitive $(x,y)$-palindromes. By Proposition \ref{lem:pal}, we know that both $\Axis_{\rho(w)}$ and $\Axis_{\rho(w')}$ are orthogonal to $\Axis_U$. Therefore the distance between them is realized by their intersection points with $\Axis_U$. Let $n$ (resp. $n'$) be the nearest-point projection of $o$ to the geodesic $\Axis_{\rho(w)}$ (resp. $\Axis_{\rho(w')}$). Then we have
\begin{align*}
d_{\H}(\Axis_{\rho(w)},\Axis_{\rho(w')})
\le d_{\H}(n,n')
\le d_{\H}(n,o)+d_{\H}(o,n')
\le 2D,
\end{align*}
that is, the distance between the two intersection points in $\Axis_U$ is bounded above by $2D$. Therefore, we conclude that the diameter of the set $J_U$ is bounded above by $4D$.
	
For the subsets $J_S$ and $J_T$ we consider the metrics $d_{(y,z)}$ and $d_{(z,x)}$ on $\F_2$, respectively. Then the same proof shows that they are bounded as well. Therefore, the representation $\rho$ has the bounded intersection property.

\medskip

We now prove the second claim of the theorem. So suppose $\rho$ has a discrete image (which, in the present case of $G=\PSL(2,\C)$, is equivalent to having a discrete orbit in $\H^3$) and satisfies the bounded intersection property with respect to a basic triple $\ti{\e}=(a,b,c)$. Since $\rho$ satisfies the first Q-condition by Definition~\ref{def:bi}(ii), it remains to check the second Q-condition.

So suppose, on the contrary, that $\rho$ does not satisfy the second Q-condition. Then there are infinitely many unoriented primitive classes $[w]\in\V$ such that $|\tr\rho(w)|\le2$. Recall that there are only three types of $\ti{\e}$-palindromes (Definition~\ref{def:fpal}) and each class $[w]\in\V$ contains at least one $\ti{\e}$-palindromes (Lemma~\ref{lem:pal2}). Thus we may assume, without loss of generality, that there are infinitely many primitive elements $w$ that are $(x,y)$-palindromes and satisfy $|\tr\rho(w)|\le2$.

Then we claim that, for any point $o\in\Axis_R$, its infinitely many images under such $\rho(w)$'s will be contained in a compact subset of $\H^3$ and hence have an accumulation point. The claim follows from the bounded intersection property, since the axes of $\rho(w)$ are all orthogonal to $\Axis_R$ and the translation lengths of $\rho(w)$ are bounded above by a uniform constant (recall \eqref {eqn:ltr}). This contradicts the assumption that $\rho$ has discrete image. Therefore, $\rho$ must satisfy the second Q-condition as well.
\end{proof}

\begin{question}
In the second statement of Theorem~\ref{II}, can we remove the injectivity condition? That is, if $\rho$ has discrete image and has the bounded intersection property, then can we still conclude that $\rho$ satisfies the Q-conditions?
\end{question}

\appendix

\section{Appendix: The structure of \texorpdfstring{$\Aut(\F_2)$}{Aut(F2)}}\label{sec:a}

We review the structure of $\Aut(\F_2)$ via the isomorphism $\Aut(\F_2)\cong\Aut(\W_3)$, where $\W_3$ denotes the free Coxeter group of rank three. One may compare our discussion with those in \cite{Gol03}*{Appendix} and \cite{GMST}*{Section 2}, for example.

As an application we explain the inductive algorithm of generating the Christoffel words in terms of $\Aut(\W_3)$.

\subsection{On \texorpdfstring{$\Aut(\W_3)$}{Aut(W3)}}

Let $\W_3=\Z/2\ast\Z/2\ast\Z/2=\langle r_1,r_2,r_3 \mid r_1^2=r_2^2=r_3^2=1\rangle$ be the free (or universal) Coxeter group of rank three. Here we collect a few facts scattered in the literature regarding the structure of $\Aut(\W_3)$. We find it helpful to visualize the group $\Aut(\W_3)$ from the representation of $\W_3$ in the group $\Isom^+(\mathbb{E}^2)$ or $\Isom^+(\H^2)$ rather than from the Cayley graph of $\W_3$. See Figure~\ref{fig:aut1}.

\begin{figure}[ht]
\labellist
\pinlabel {$r_1$} at 206 166
\pinlabel {$r_2$} at 276 116
\pinlabel {$r_3$} at 308 166
\pinlabel {$\si_{12}$} at 142 220
\pinlabel {$\si_{13}$} at 106 164
\pinlabel {$\si_{23}$} at 234 60
\pinlabel {$\si_{21}$} at 352 60
\pinlabel {$\si_{31}$} at 410 164
\pinlabel {$\si_{32}$} at 340 222
\pinlabel {$r_1r_2r_1$} at 46 256
\pinlabel {$r_1r_3r_1$} at -4 176
\pinlabel {$r_2r_3r_2$} at 230 10
\pinlabel {$r_2r_1r_2$} at 380 10
\pinlabel {$r_3r_1r_3$} at 510 176
\pinlabel {$r_3r_2r_3$} at 410 256
\pinlabel {$f_1$} at 90 206
\pinlabel {$f_2$} at 290 50
\pinlabel {$f_3$} at 410 210
\endlabellist
\centering
\includegraphics[width=0.7\textwidth]{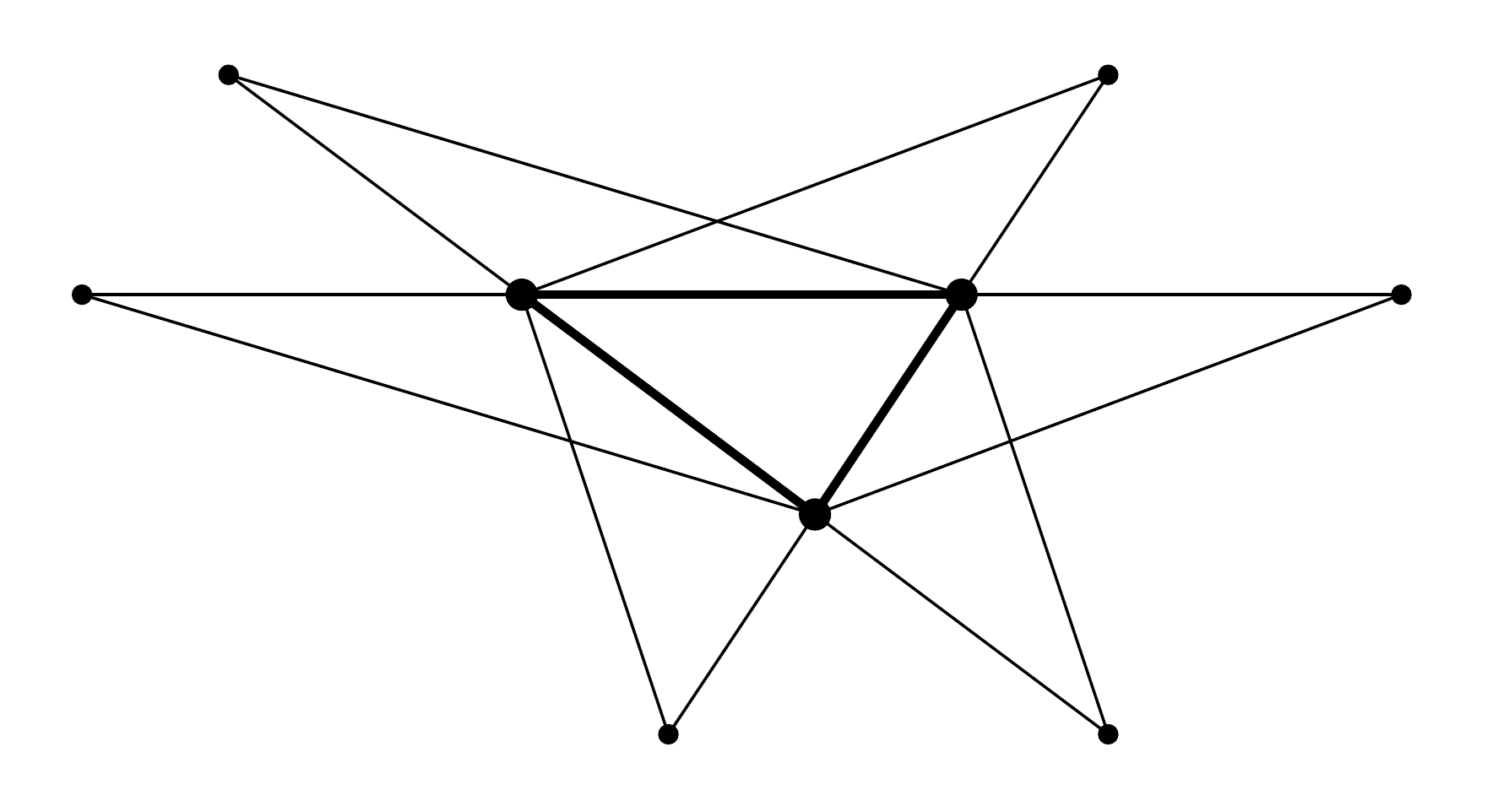}
\caption{Special involutory automorphisms $\si_{ij}$ in $\Aut(\W_3)$.}
\label{fig:aut1}
\end{figure}
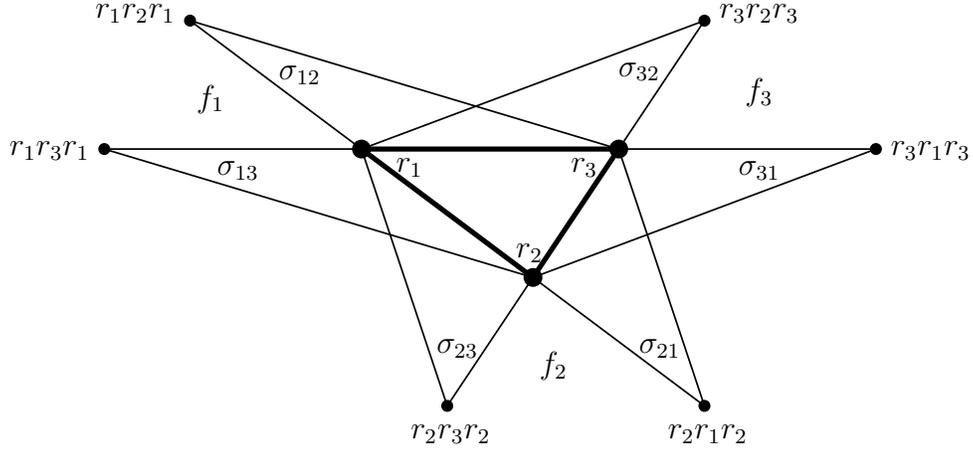

Let $\Sym_3=\{1,s_{123},s_{321},t_{12},t_{23},t_{31}\}$ denote the group of automorphisms of $\W_3$ permuting the free factors. Let $\Spe(\W_3)$ denote the group of automorphisms of $\W_3$ that preserve the conjugacy classes of $r_1$, $r_2$ and $r_3$; this group is known to be generated by the following \emph{involutory} automorphisms of $\W_3$ (see \cite{Muhl98}*{Theorem}):
\begin{align*}
\si_{12}(r_1,r_2,r_3)&=(r_1,r_1r_2r_1,r_3),& \si_{13}(r_1,r_2,r_3)&=(r_1,r_2,r_1r_3r_1),\\
\si_{23}(r_1,r_2,r_3)&=(r_1,r_2,r_2r_3r_2),&
\si_{21}(r_1,r_2,r_3)&=(r_2r_1r_2,r_1,r_3),\\
\si_{31}(r_1,r_2,r_3)&=(r_3r_1r_3,r_2,r_3),&
\si_{32}(r_1,r_2,r_3)&=(r_1,r_3r_2r_3,r_3).
\end{align*}
We set and choose
\begin{align*}
\io_1&=\si_{12}\si_{13}=\si_{13}\si_{12},&
\io_2&=\si_{21}\si_{23}=\si_{23}\si_{21},&
\io_3&=\si_{31}\si_{32}=\si_{32}\si_{31},\\
f_1&\in\{\si_{12},\si_{13}\},&
f_2&\in\{\si_{23}, \si_{21}\},&
f_3&\in\{\si_{31}, \si_{32}\}.
\end{align*}
Note that $\io_k$ ($k=1,2,3$) is the conjugation by $r_k$ and thus we have $\Inn(\W_3)=\langle \io_1, \io_2, \io_3\rangle\cong\W_3$. Then it can be shown that
\begin{align*}
\Aut(\W_3)
&= \Spe(\W_3)\rtimes \Sym_3\\
&=(\langle \io_1, \io_2, \io_3\rangle
\rtimes\langle f_1, f_2, f_3 \rangle)
\rtimes \Sym_3\\
&= (\Inn(\W_3)\rtimes\W_3)\rtimes \Sym_3.
\end{align*}
See \cite{Fran02}*{Theorem 2.11} or \cite{CG90}*{Lemma 3.5 and \S 7} for a proof. The action of $\langle f_1, f_2, f_3 \rangle$ on $\langle \io_1, \io_2, \io_3\rangle$ corresponds to the action of $\langle f_1, f_2, f_3 \rangle$ on $\W_3=\langle r_1, r_2, r_3\rangle$. For example, if $f_1=\si_{12}$ then we have
\[
f_1\io_1f_1=\io_1,\quad f_1\io_2f_1=\io_1\io_2\io_1,\quad f_1\io_3f_1=\io_3.
\]
From the above we then have
\begin{align*}
\Out(\W_3)
&=\Aut(\W_3)/\Inn(\W_3)\\
&\cong \W_3\rtimes \Sym_3\\
&= \langle [f_1], [f_2], [f_3] \rangle
\rtimes \{1,s_{123},s_{321},t_{12},t_{23},t_{31}\}\\
&\cong\PGL(2,\Z)\\
&\cong T^*(2,3,\infty),
\end{align*}
where the last group denotes the $(2,3,\infty)$-triangle reflection group of $\H^2$. See Figure~\ref{fig:aut3}, where we draw $\H^2$ in the Klein projective model and shaded a fundamental domain for the action of $T^*(2,3,\infty)$.

\begin{figure}[ht]
\labellist
\pinlabel {$(1/0)\;3$} at 42 88
\pinlabel {$1\;(0/1)$} at 416 88
\pinlabel {$2\;(1/1)$} at 254 352
\pinlabel {$L$} at 172 160
\pinlabel {$R$} at 266 130
\pinlabel {$t_{12}$} at 138 106
\pinlabel {$t_{23}$} at 324 106
\pinlabel {$t_{31}$} at 216 286
\pinlabel {$[f_1]$} at 124 260
\pinlabel {$[f_2]$} at 200 40
\pinlabel {$[f_3]$} at 360 210
\endlabellist
\centering
\includegraphics[width=0.5\textwidth]{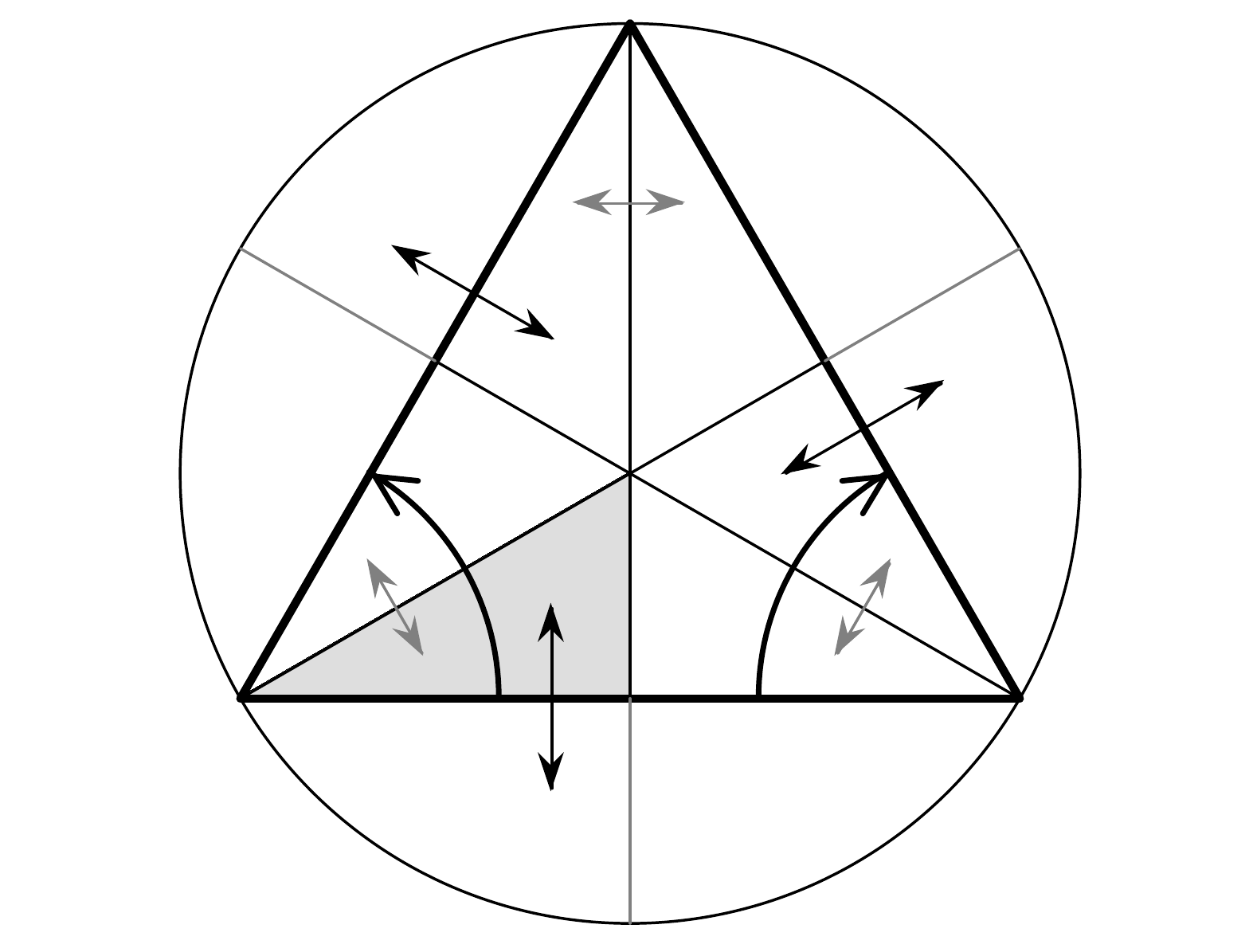}
\caption{$\Out(\W_3)$ is isomorphic to $\PGL(2,\Z)\cong T^*(2,3,\infty)$.}
\label{fig:aut3}
\end{figure}

The short exact sequence
\[
1\to\Inn(\W_3)\to\Aut(\W_3)\to\Out(\W_3)\to1
\]
does not split, that is, the whole group $\Out(\W_3)$ does not lift to $\Aut(\W_3)$. This is because the $f_i$'s cannot be chosen to have the full symmetry of $\Sym_3$; see Figure~\ref{fig:aut1}. At best, they can be chosen so as to admit the symmetry of $\Z/3=\{1,s_{123},s_{321}\}$, say, $f_1=\si_{12}$, $f_2=\si_{23}$ and $f_3=\si_{31}$, and we can lift only the part $\W_3\rtimes\Z/3$ of $\Out(\W_3)$.

\subsection{The isomorphism \texorpdfstring{$\Aut(\W_3)\cong\Aut(\F_2)$}{Aut(W3)=Aut(F2)}}\label{sec:isom}

Given a basis triple $(r_1,r_2,r_3)$ of $\W_3$, the pair of elements $(a,b):=(r_1r_2,r_2r_3)$ freely generate a subgroup $\F_2=\F(a,b)$ of index two. One can show that the correspondence
\begin{align}\label{eqn:isom}
(r_1,r_2,r_3)\longleftrightarrow(r_1r_2,r_2r_3)=(a,b)
\end{align}
induces an isomorphism
\[
\Aut(\W_3)\cong\Aut(\F_2).
\]
See Figure~\ref{fig:aut2}. Compare \cite{CD02}*{p.146} or \cite{PRW10}*{Remark 2}.

\begin{figure}[ht]
\labellist
\pinlabel {\Large $L$} at 160 206
\pinlabel {\large $L^{-1}$} at 340 74
\pinlabel {\Large $R$} at 330 210
\pinlabel {\large $R^{-1}$} at 236 74
\pinlabel {$a$} at 242 140
\pinlabel {$b$} at 290 140
\pinlabel {$ab$} at 266 158
\pinlabel {$a^2b$} at 200 224
\pinlabel {$ab^2$} at 290 226
\pinlabel {$r_1$} at 180 184
\pinlabel {$r_2$} at 280 80
\pinlabel {$r_3$} at 324 184
\pinlabel {$r_1r_2r_1$} at 46 250
\pinlabel {$r_3r_2r_3$} at 410 250
\pinlabel {$r_2r_3r_2$} at 230 6
\pinlabel {$r_2r_1r_2$} at 380 6
\endlabellist
\centering
\includegraphics[width=0.7\textwidth]{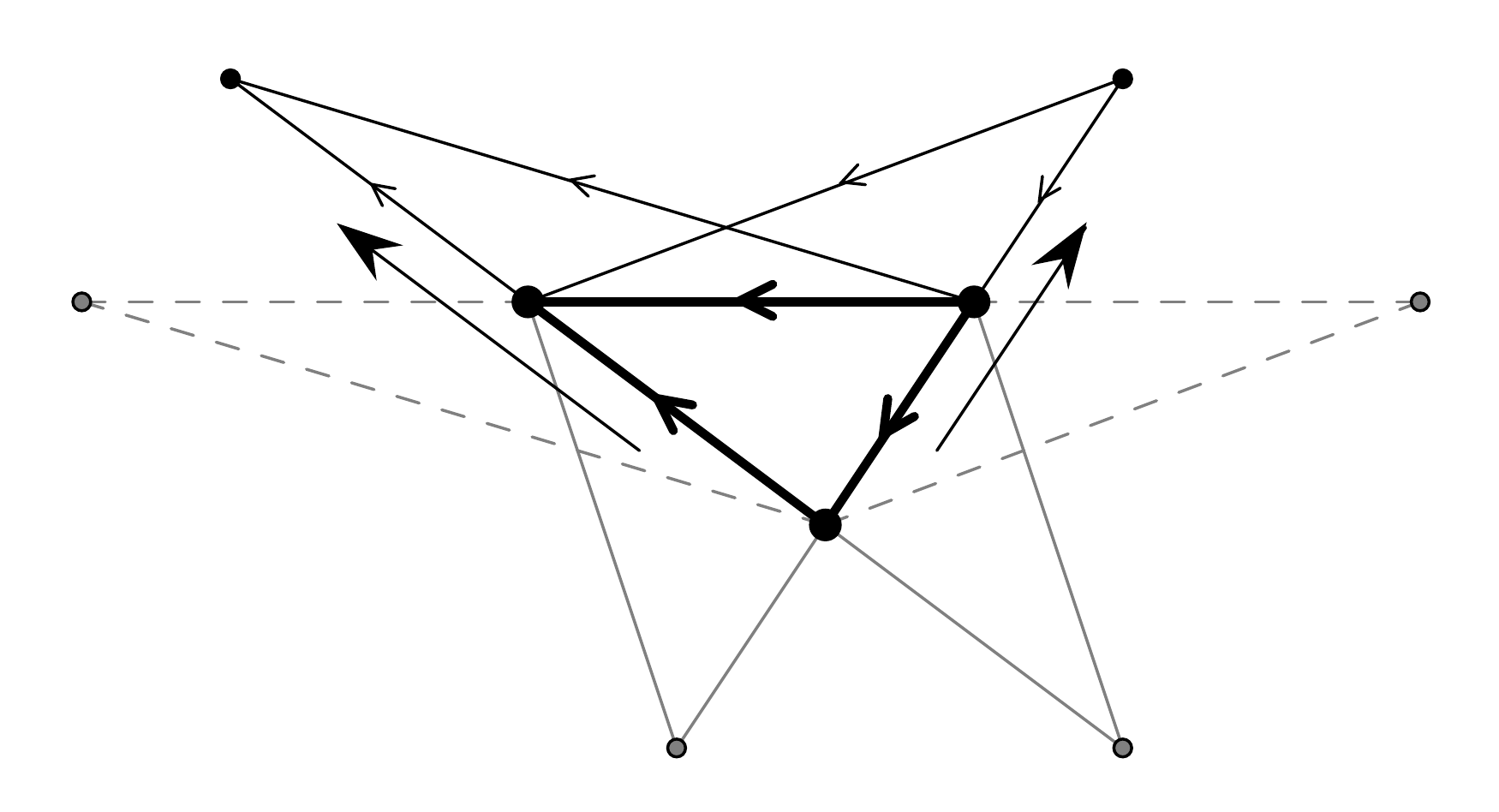}
\caption{The isomorphism $\Aut(\W_3)\cong\Aut(\F_2)$.}
\label{fig:aut2}
\end{figure}

Under the above isomorphism, the automorphisms $\si_{12}$, $\si_{23}$ and $\si_{31}$ correspond to the automorphisms $\widetilde{\mathfrak{J}}_2$, $\widetilde{\mathfrak{J}}_3$ and $\widetilde{\mathfrak{J}}_1$ of $\Aut(\F_2)$ defined by the following bases of $\F(a,b)$:
\begin{align}\label{eqn:J}
\begin{aligned}
(r_1,r_1r_2r_1,r_3)&\longleftrightarrow(r_2r_1,r_1r_2r_1r_3)=(a^{-1},a^2b),\\
(r_1,r_2,r_2r_3r_2)&\longleftrightarrow(r_1r_2,r_3r_2)=(a,b^{-1}),\\
(r_3r_1r_3,r_2,r_3)&\longleftrightarrow(r_3r_1r_3r_2,r_2r_3)=(b^{-1}a^{-1}b^{-1},b).
\end{aligned}
\end{align}
These appear also in \cite{GMST}*{\S.2.1 (6)}. On the other hand, the (opposite) automorphisms $\si_{13}$, $\si_{21}$ and $\si_{32}$ correspond to
\begin{align*}
(r_1,r_2,r_1r_3r_1)&\longleftrightarrow(r_1r_2,r_2r_1r_3r_1)=(a,a^{-1}b^{-1}a^{-1}),\\
(r_2r_1r_2,r_2,r_3)&\longleftrightarrow(r_2r_1,r_2r_3)=(a^{-1},b),\\
(r_1,r_3r_2r_3,r_3)&\longleftrightarrow(r_1r_3r_2r_3,r_3r_2)=(ab^2,b^{-1}).
\end{align*}

The inner automorphism $\io_2\in\Aut(\W_3)$ corresponds to the involutory automorphism $\mathfrak{e}\in\Aut(\F_2)$ (defined in Section~\ref{sec:farey}):
\begin{align*}
\io_2(r_1,r_2,r_3)=(r_2r_1r_2,r_2,r_2r_3r_2)
\longleftrightarrow
(r_2r_1,r_3r_2)=(a^{-1},b^{-1})=\mathfrak{e}(a,b).
\end{align*}

\subsection{Finding primitive representatives}\label{sec:cf}

We observed that $\Out(\W_3)$ does not lift to $\Aut(\W_3)$. Its subgroup $\W_3\cong\langle [f_1], [f_2], [f_3] \rangle$, however, can be lifted to $\Aut(\W_3)$ and still acts transitively on the set of Farey triangles. Therefore, in order to find representatives of the $\sim$-equivalence classes in $\V=\Prim/_\sim\subset\F_2/_\sim$, we may take the orbit of a basis triple of $\W_3$ under the action of 
the (lifted) subgroup $\W_3\cong\langle f_1, f_2, f_3 \rangle$ of $\Aut(\W_3)$, then convert its members to bases of $\F_2$ using the correspondence \eqref{eqn:isom}. For example, if we choose $f_1=\si_{12}$, $f_2=\si_{23}$ and $f_3=\si_{31}$, then the action is given by $\widetilde{\mathfrak{J}}_2$, $\widetilde{\mathfrak{J}}_3$ and $\widetilde{\mathfrak{J}}_1$ as in \eqref{eqn:J}. But this example has a drawback that the formulas look somewhat complicated due to lack of symmetry.

Instead of lifting $\W_3\subset\Out(\W_3)$ we rather lift the free submonoid of $\Out(\W_3)\cong\PGL(2,\Z)$ generated by $\begin{psmallmatrix}1&1\\0&1\end{psmallmatrix}$ and $\begin{psmallmatrix}1&0\\1&1\end{psmallmatrix}$. See Remark~\ref{rem:monoid}. In this way we can find ``half" of the representatives. Compare \cite{ADP99}*{Definition 2.6 and Lemma 5.2}. The lifted automorphisms are given by
\begin{align*}
L:=\si_{12}t_{12}&:(r_1,r_2,r_3)\mapsto(r_1r_2r_1,r_1,r_3)\longleftrightarrow(r_1r_2,r_1r_3)=(a,ab),\\
R:=\si_{32}t_{23}&:(r_1,r_2,r_3)\mapsto(r_1,r_3,r_3r_2r_3)\longleftrightarrow(r_1r_3,r_2r_3)=(ab,b).
\end{align*}
See Figures~\ref{fig:aut3} and \ref{fig:aut2} again. Let $\lt$ and $\rt$ denote the corresponding automorphisms of $\F_2$:
\begin{align*}
\lt(a)=a,\;\lt(b)=ab,\quad\textup{and}\quad
\rt(a)=ab,\;\rt(b)=b.
\end{align*}
We also consider the functions $\un{\lt},\un{\rt}:\F_2\times\F_2\to\F_2\times\F_2$ defined by
\begin{align*}
\un{\lt}(x,y)&=(x,xy)\\
\un{\rt}(x,y)&=(xy,y).
\end{align*}
Using the pair $\lt$ and $\rt$ (or the pair $\un{\lt}$ and $\un{\rt}$) we can generate the $\e$-Christoffel bases with $\e=(a,b)$. (See the paragraph prior to Remark~\ref{rem:monoid}.)

Suppose a positive rational number $p/q$ has continued fraction expansion
\[
p/q=[n_0; n_1, n_2, \ldots, n_k]=[n_0; n_1, n_2, \ldots, n_k-1,1].
\]
Let $\Ch_\e(p/q)$ denote the $\e$-Christoffel word associated to $p/q$. Then we have
\[
\Ch_\e(p/q)=a'b',
\]
where
\begin{align*}
(a',b')
&=\lt^{n_0}\rt^{n_1}\lt^{n_2}\cdots \lt^{n_k-1}(a,b)\\
&=\un{\lt}^{n_k-1}\cdots\un{\lt}^{n_2}\un{\rt}^{n_1}\un{\lt}^{n_0}(a,b),
\end{align*}
or
\begin{align*}
(a',b')
&=\lt^{n_0}\rt^{n_1}\lt^{n_2}\cdots \rt^{n_k-1}(a,b)\\
&=\un{\rt}^{n_k-1}\cdots\un{\lt}^{n_2}\un{\rt}^{n_1}\un{\lt}^{n_0}(a,b)
\end{align*}
depending on the parity of $k$.

Let us illustrate this using an example. Consider the continued fraction expansion
\[
\frac{17}{10}=1+\cfrac{1}{{1+\cfrac{1}{2+\cfrac{1}{3}}}}=[1;1,2,3]=[1;1,2,2,1].
\]
Then we have
\begin{align*}
(a,b)&\overset{\rt^2}{\longrightarrow}(ab^2,b) &
(a,b)&\overset{\un{\lt}^1}{\longrightarrow}(a,ab)\\
&\overset{\lt^2}{\longrightarrow}(a(a^2b)^2,a^2b) &
&\overset{\un{\rt}^1}{\longrightarrow}(a^2b,ab)\\
&\overset{\rt^1}{\longrightarrow}(ab((ab)^2b)^2,(ab)^2b) &
&\overset{\un{\lt}^2}{\longrightarrow}(a^2b,(a^2b)^2ab)\\
&\overset{\lt^1}{\longrightarrow}(a^2b((a^2b)^2ab)^2,(a^2b)^2ab),&
&\overset{\un{\rt}^2}{\longrightarrow}(a^2b((a^2b)^2ab)^2,(a^2b)^2ab).
\end{align*}
That is, 
\[
(a',b')
=\lt^1\rt^1\lt^2\rt^2(a,b)
=\un{\rt}^2\un{\lt}^2\un{\rt}^1\un{\lt}^1(a,b)
=(a^2b((a^2b)^2ab)^2,(a^2b)^2ab)
\]
Thus we obtain
\[
\Ch_\e(17/10)=a'b'=a^2b((a^2b)^2ab)^3.
\]

We emphasize the difference between $\lt^1\rt^1\lt^2\rt^2$ and $\un{\rt}^2\un{\lt}^2\un{\rt}^1\un{\lt}^1$. The former is a sequence of ``substitutions" as $\lt$ and $\rt$ are automorphisms of $\F_2$ defined by their actions on the distinguished basis $\e=(a,b)$, while the latter is a sequence of left or right ``concatenation" dictated by the functions $\un{\lt}$ and $\un{\rt}$. For practical purposes the latter method is preferable since it is compatible with the inductive generation of $\e$-Christoffel words as in Figure~\ref{fig:torus2} and Definition~\ref{def:chris}.

\medskip

Lastly, we remark that both $L$ and $R$ preserve the element $r_1r_2r_3$:
\begin{align*}
L(r_1r_2r_3)=(r_1r_2r_1)(r_1)(r_3)=r_1r_2r_3,\\
R(r_1r_2r_3)=(r_1)(r_3)(r_3r_2r_3)=r_1r_2r_3.
\end{align*}
Since we have
\[
(r_1r_2r_3)^2=(r_1r_2r_3)(r_2r_2)(r_1r_2r_3)=ab^{-1}a^{-1}b
\]
(as in \eqref{eqn:kappa}), the corresponding automorphisms $\lt$ and $\rt$ in $\Aut(\F_2)$ preserve the commutator $[a,b^{-1}]=ab^{-1}a^{-1}b$. Of course, one can check this directly from their definitions. Compare \cite{CD02}*{p.145 and p.147}, \cite{ASWY07}*{Lemma 2.1.7} and \cite{KR07}*{Remark 5.6(c)}.

\vspace{3\baselineskip}\noindent
Jaejeong Lee\\
\url{jjlee@kias.re.kr}

\vspace{1\baselineskip}\noindent
Binbin Xu\\
\url{binbin.xu@uni.lu}

\end{document}